\documentclass[10pt]{amsart}
\usepackage{amssymb}
\usepackage{dsfont}
\usepackage{amscd}
\usepackage[mathscr]{euscript} 
\usepackage[all]{xy}
\usepackage{xcolor}
\usepackage{url}
\usepackage{stmaryrd}
\usepackage{comment}
\usepackage[retainorgcmds]{IEEEtrantools}
\usepackage{hyperref}
\usepackage{graphicx}
\usepackage{mathtools}
\usepackage{centernot}
\usepackage{marvosym}
\usepackage{tikz}
\usepackage{soul}

\usepackage{tikz-cd}

\usepackage{color}

\makeatletter
\def\author@andify{%
  \nxandlist {\unskip ,\penalty-1 \space\ignorespaces}%
    {\unskip {} \@@and~}%
    {\unskip \penalty-2 \space \@@and~}%
}

\title{Co-theory of sorted profinite groups for PAC structures}

\author[D. M. HOFFMANN]{Daniel Max Hoffmann$^{\dagger}$}
\thanks{2010 \textit{Mathematics Subject Classification}. 
Primary 03C95;
Secondary 03C45, 03C52.}
\thanks{\textit{Key words and phrases}. pseudo-algebraically closed structures, Galois groups, Kim-independence}
\thanks{$^{\dagger}$SDG. The first author is supported by
the Polish Natonal Agency for Academic Exchange and
 the National Science Centre (Narodowe Centrum Nauki, Poland) grants no. 2016/21/N/ST1/01465,
and 2015/19/B/ST1/01150.}
\address{$^{\dagger}$ Instytut Matematyki\\
Uniwersytet Warszawski\\
Warszawa\\
Poland
\newline \indent {\em and}\newline
\indent \hspace{1mm} Department of Mathematics \\ University of Notre Dame \\ Notre Dame \\ IN \\ USA}
\email{daniel.max.hoffmann@gmail.com}
\urladdr{https://sites.google.com/site/danielmaxhoffmann/}

\author[J. LEE]{Junguk Lee$^{\ast}$}
\thanks{$^{\ast}$The second author is supported by the Narodowe Centrum Nauki grant no. 2016/22/E/ST1/00450 and by KAIST Advanced Institute for Science-X fellowship.}
\address{$^{\ast}$Instytut Matematyczny\\
Uniwersytet Wroc{\l}awski\\
Wroc{\l}aw\\
Poland
\newline \indent {\em Current address}: Department of Mathematical Sciences, KAIST, 291, Daehak-Ro, Yuseong-Gu, Daejeon, 34141, Republic of Korea}
\email{ljwhayo@kaist.ac.kr} 
\urladdr{https://sites.google.com/site/leejunguk0323/}

\DeclareMathOperator{\Sym}{Sym}

\DeclareMathOperator{\fratt}{Fratt}

\DeclareMathOperator{\acl}{acl} \DeclareMathOperator{\dcl}{dcl} 
 \DeclareMathOperator{\aut}{Aut} \DeclareMathOperator{\id}{id}

\DeclareMathOperator{\im}{im}

 \DeclareMathOperator{\gal}{Gal}

 \DeclareMathOperator{\theo}{Th}

 \DeclareMathOperator{\eq}{eq}
 
\DeclareMathOperator{\tp}{tp}

\DeclareMathOperator{\ddf}{DF}\DeclareMathOperator{\dcf}{DCF}\DeclareMathOperator{\scf}{SCF}
\DeclareMathOperator{\acf}{ACF}

\DeclareMathOperator{\qftp}{qftp}

\DeclareMathOperator{\res}{res}
\DeclareMathOperator{\ST}{ST}
\DeclareMathOperator{\scs}{SCS}

\DeclareMathOperator{\Ms}{M^{\ast\ast}}
\DeclareMathOperator{\Fs}{F^{\ast\ast}}

\def\conj{\operatorname{Conj}}
\def\pr{\operatorname{Pr}}
\def\U{\operatorname{U}}

\newtheorem{theorem}{Theorem}[section]
\newtheorem{prop}[theorem]{Proposition}
\newtheorem{lemma}[theorem]{Lemma}
\newtheorem{cor}[theorem]{Corollary}
\newtheorem{fact}[theorem]{Fact}
\theoremstyle{definition}
\newtheorem{definition}[theorem]{Definition}
\newtheorem{example}[theorem]{Example}
\newtheorem{remark}[theorem]{Remark}
\newtheorem{question}[theorem]{Question}

\newtheorem{notation}[theorem]{Notation}
\theoremstyle{remark}
\newtheorem*{theorem*}{Theorem}
\newtheorem*{cor*}{Corollary}

\theoremstyle{definition}
\theoremstyle{definition}

\theoremstyle{definition}

\theoremstyle{remark}

\makeatletter
\providecommand*{\cupdot}{%
  \mathbin{%
    \mathpalette\@cupdot{}%
  }%
}
\newcommand*{\@cupdot}[2]{%
  \ooalign{%
    $\m@th#1\cup$\cr
    \sbox0{$#1\cup$}%
    \dimen@=\ht0 %
    \sbox0{$\m@th#1\cdot$}%
    \advance\dimen@ by -\ht0 %
    \dimen@=.5\dimen@
    \hidewidth\raise\dimen@\box0\hidewidth
  }%
}

\providecommand*{\bigcupdot}{%
  \mathop{%
    \vphantom{\bigcup}%
    \mathpalette\@bigcupdot{}%
  }%
}
\newcommand*{\@bigcupdot}[2]{%
  \ooalign{%
    $\m@th#1\bigcup$\cr
    \sbox0{$#1\bigcup$}%
    \dimen@=\ht0 %
    \advance\dimen@ by -\dp0 %
    \sbox0{\scalebox{2}{$\m@th#1\cdot$}}%
    \advance\dimen@ by -\ht0 %
    \dimen@=.5\dimen@
    \hidewidth\raise\dimen@\box0\hidewidth
  }%
}
\makeatother

\def\Ind#1#2{#1\setbox0=\hbox{$#1x$}\kern\wd0\hbox to 0pt{\hss$#1\mid$\hss}
\lower.9\ht0\hbox to 0pt{\hss$#1\smile$\hss}\kern\wd0}

\def\ind{\mathop{\mathpalette\Ind{}}}

\def\notind#1#2{#1\setbox0=\hbox{$#1x$}\kern\wd0
\hbox to 0pt{\mathchardef\nn=12854\hss$#1\nn$\kern1.4\wd0\hss}
\hbox to 0pt{\hss$#1\mid$\hss}\lower.9\ht0 \hbox to 0pt{\hss$#1\smile$\hss}\kern\wd0}

\def\nind{\mathop{\mathpalette\notind{}}}

\def\x{\bar{x}}
\def\y{\bar{y}}

\def\a{\bar{a}}

\def\p{\bar{p}}

\begin{document}
\newcommand{\ov}{\overline}
\newcommand{\FC}{\mathfrak{C}}

\newcommand{\twoc}[3]{ {#1} \choose {{#2}|{#3}}}
\newcommand{\thrc}[4]{ {#1} \choose {{#2}|{#3}|{#4}}}
\newcommand{\Rr}{{\mathds{R}}}
\newcommand{\Kk}{{\mathds{K}}}

\newcommand{\dlog}{\mathrm{ld}}
\newcommand{\ga}{\mathbb{G}_{\rm{a}}}
\newcommand{\gm}{\mathbb{G}_{\rm{m}}}
\newcommand{\gaf}{\widehat{\mathbb{G}}_{\rm{a}}}
\newcommand{\gmf}{\widehat{\mathbb{G}}_{\rm{m}}}
\newcommand{\gdf}{\mathfrak{g}-\ddf}
\newcommand{\gdcf}{\mathfrak{g}-\dcf}
\newcommand{\fdf}{F-\ddf}
\newcommand{\fdcf}{F-\dcf}
\newcommand{\mw}{\scf_{\text{MW},e}}

\newcommand{\BC}{{\mathbb C}}
\newcommand{\BZ}{{\mathbb Z}}

\newcommand{\CC}{{\mathcal C}}
\newcommand{\CG}{{\mathcal G}}
\newcommand{\CK}{{\mathcal K}}
\newcommand{\CL}{{\mathcal L}}
\newcommand{\CN}{{\mathcal N}}
\newcommand{\CS}{{\mathcal S}}
\newcommand{\CJ}{{\mathcal J}}
\newcommand{\CU}{{\mathcal U}}
\newcommand{\CF}{{\mathcal F}}
\newcommand{\CP}{{\mathcal P}}
\newcommand{\CI}{{\mathcal I}}

\def\Th{\operatorname{Th}}

\begin{abstract}
We achieve several results. 
First, we develop a variant of the theory of absolute Galois groups in the context of many sorted structures. 
Second, we provide a method for coding absolute Galois groups of structures, so they can be interpreted in some monster model with an additional predicate. 
Third, we prove the ``Weak Independence Theorem" for PAC substructures of an ambient structure with nfcp and the property B(3).
Fourth, we describe Kim-dividing in these PAC substructures and show several results related to the SOP$_n$ hierarchy.
Fifth, we characterize the algebraic closure in PAC structures.
\end{abstract}
\maketitle

\tableofcontents
\section{Introduction}
A major aim of this paper is to generalize research on the ``co-logic" of profinite groups initiated in an unpublished work of Cherlin, van den Dries and Macintyre (\cite{cherlindriesmacintyre}) and then continued with successful applications by Chatzidakis (e.g. \cite{zoeIP}, \cite{chatzidakis2002}, \cite{chatzidakis2017}). 
Originally, the ``co-logic" was introduced as a tool to describe subfields (in particular pseudo-algebraically closed subfields) of a one-sorted saturated algebraically (or separably) closed field.
Our modification of the ``co-logic" serves as a tool in studying substructures of an arbitrary, possibly many-sorted, monster model. However the most interesting results are obtained under additional assumptions on the monster model (like stability, nfcp, or the property $B(3)$) and for the class of \emph{pseudo-algebraically closed} substructures.

The notion of a pseudo algebraically closed substructure (PAC substructure, see Definition \ref{def:PAC}) is a natural generalization of the notion of a pseudo-algebraically closed field (PAC field), which occurs in works of 
Ax (\cite{Ax1}, \cite{Ax2}) and Frey (\cite{FREY}) and which arises from studying pseudo-finite fields. 
A field $K$ is PAC if and only if each absolutely irreducible $K$-variety has a $K$-rational point (or equivalently: it is existentially closed in every regular extension).
Because of the so-called ``Elementary Equivalence Theorem" (see \cite[Theorem 20.3.3]{FrJa} and \cite[Theorem 3.2]{JardenKiehne}), PAC fields were extensively studied in the second half of the 20th century as a natural class of fields determined by the properties of their absolute Galois groups. Also model theory recognizes PAC fields as a source of interesting phenomena (\cite{ershov1980}, \cite{ChaPil}, \cite{chahru04}). 
For example,
PAC fields played an important role in the studies on (geometric) simplicity (see the introduction to \cite{manuscript}).
The relation between simplicity of a PAC field and properties of its absolute Galois group
is a well-known result (see \cite[Fact 2.6.7]{kim1} --- a PAC field is simple if and only if its absolute Galois group is small as a profinite group).
However, it turns out that there is a similar link between NSOP$_1$ of a PAC field and more sophisticated properties of its absolute Galois group
(see \cite{chatzidakis2017}, \cite{NickThesis}), and the main goal of the following paper is to generalize this link to the level of arbitrary PAC substructures in the stable context.

PAC substructures were already studied in the case of a strongly minimal ambient monster model (\cite{manuscript}) and also in the case of a stable ambient monster model (\cite{PilPol}). An interesting result is provided in \cite{Polkowska}, where the author proves that theory of \emph{bounded} PAC structures must be simple. Bounded means that the absolute Galois group (automorphisms of the algebraic closure considered in the stable ambient monster model) is a small profinite group. Therefore it was reasonable to suspect that, similarly to PAC fields, PAC structures are controlled by their absolute Galois groups. The main result of \cite{DHL1} is the so-called ``Elementary Equivalence Theorem for Structures" --- a counterpart of the aforementioned ``Elementary Equivalence Theorem" covering the case of PAC structures
(in short: two PAC structures have the same first order theory provided they have isomorphic absolute Galois groups). In the case of fields, the ``Elementary Equivalence Theorem" was elaborated in \cite{cherlindriesmacintyre} to a version involving the ``co-logic" (see \cite[Proposition 33]{cherlindriesmacintyre}), which was helpful in later studies on PAC fields in model theory (especially in the current studies in neostability: \cite{chatzidakis2017} and \cite{NickThesis}).
Therefore we are developing here a version of the ``co-logic" for arbitrary structures, expressing afresh the ``Elementary Equivalence Theorem for Structures" and then use it to show results related to Kim-independence in PAC substructures.
Our generalization of the ``co-logic" is thought to achieve the following goals:
\begin{itemize}
\item to describe absolute Galois groups in a way that they can be interpreted in a monster model (many-sorted case): Section \ref{sec:sorted}
\item  to refine the ``Elementary Equivalence Theorem for structures" and provide a description of types in PAC structures: Section \ref{sec:coelementary}
\item to generalize a recent theorem of Chatzidakis (c.f. \cite[Theorem 3.1]{chatzidakis2017}): Section \ref{sec:zoe.theorem}
\item to achieve the ``Weak Independence Theorem" (Theorem \ref{thm:weak.ind.thm}): Section \ref{sec:wit.nsop}
\item to describe Kim-independence and conditions for NSOP$_n$ in PAC substructures: Section \ref{sec:wit.nsop}
\end{itemize}
The part related to Kim-independence was inspired by \cite{NickThesis}. Let us explain the context of these results. In \cite{chatzidakis2017}, Chatzidakis achieved her Theorem 3.1, which is a beautiful result connecting a notion of independence in a PAC field with its counterpart on the level of the absolute Galois group. Then Chatzidakis considered a notion of independence combining the forking independence in algebraically/separably closed monster field and the forking independence present in the ``co-logic" of the absolute Galois group of a given PAC subfield, and this led to results about NSOP$_n$ for $n>2$.
Ramsey has (in \cite{NickThesis}) a slightly different approach and he combines the notion of independence from the forking independence in algebraically/separably closed monster field and the Kim-independence on the level of the absolute Galois group of a given PAC subfield. By this, he obtains results concerning NSOP$_1$ and NSOP$_2$, and a characterization of the Kim-independence in a PAC field.
All this was achieved using Chatzidakis' theorem (\cite[Theorem 3.1]{chatzidakis2017}), therefore the central part of this paper is a generalization of \cite[Theorem 3.1]{chatzidakis2017} to the case of substructures of a stable monster model which satisfies the property $B(3)$ (\cite{Kim16}, \cite{GKK}), Proposition \ref{thm:Zoe}. After achieving Proposition \ref{thm:Zoe}, we start to assume nfcp (the \emph{no finite cover property}), mainly because our interpretation of absolute Galois groups is given in pairs of structures, and a theory of pairs of structures is more tame if the bigger structure has nfcp.

In Section \ref{sec:wit.nsop}, we provide the so-called ``Weak Independence Theorem", Theorem \ref{thm:weak.ind.thm}, which is the main ingredient in our results related to NSOP$_1$ and the Kim-independence. The Weak Independence Theorem says that if the Independence Theorem (over a model) holds in the ``co-logic", it also holds in a PAC substructure. We hope that Theorem \ref{thm:weak.ind.thm} will serve in a better understanding of the nature of Kim-independence in the ongoing research on neostability. 
A prospective use of Theorem \ref{thm:weak.ind.thm} might involve
fields with operators (to work with a monster model which has nfcp and the property $B(3)$),
$G$-actions (to get control over the absolute Galois group, as in \cite{Hoff4})
and results about the logical structure of profinite groups (as in \cite{zoeIP}, e.g. if a profinite groups enjoys the \emph{Iwasawa Property}, then its theory --- in the language of complete systems --- is stable).

In Section \ref{sec:acl}, we provide a description of the algebraic closure in PAC substructures.
This part is independent from the previous sections and generalizes similar results given in \cite{ChaPil}. In the case of $\omega$-stable monster models, we obtain a precise description of the algebraic closure operator in PAC structures. 

We end the paper with Section \ref{sec:applications}, where we show that the previously obtained results might be applied to the theory of DCF$_0$.

Now, let us provide conditions assumed in this paper.
We fix a theory $T_0$ in a language $\mathcal{L}_0$, and we set $T:=(T_0^{\eq})^m$ which is a theory in language $\mathcal{L}:=(\mathcal{L}_0^{\eq})^m$ (we add imaginary sorts and then do the Morleyisation). Note that $T$ has quantifier elimination and elimination of imaginaries
(even uniform elimination of imaginaries in the sense of point b) from \cite[Lemma 8.4.7]{tentzieg}, which will be used in Subsection  \ref{sec:encoding}).
Moreover
\begin{itemize}
\item if $T_0$ is stable, then $T$ is stable,
\item if $T_0$ has nfcp (\emph{the no finite cover property}, 
Definition 4.1 and Theorem 4.2 in \cite[Chapter II]{ShelahModels}), then $T$ has nfcp
\item if $T_0^{\eq}$ is stable and has the property $B(3)$ (see Definition \ref{def:boundary_property}), then $T$ has the property $B(3)$.
\end{itemize}
Let us enumerate all sorts of $\mathcal{L}$ by $\mathcal{J}:=(S_i)_{i\in I}$.
Moreover, we fix a monster model $\mathfrak{C}\models T$ and assume that $T=\theo(\mathfrak{C})$
(in other words: we assume that $T$ is complete).

\section{Preliminaries}
Here, we provide definitions of several notions important for the rest of this paper. The paper continues studies from \cite{DHL1}, hence instead of copying large parts of the text of \cite{DHL1}, we decided to include only definitions of some basic notions which are used in formulations of forthcoming results.
\begin{definition}
For any subsets $A\subseteq B$ of $\mathfrak{C}$ and tuple $b\in\mathfrak{C}$, we define
$$\CG(B/A):=\aut\big(\dcl(AB)/\dcl(A)\big),\qquad \CG(A):=\aut\big(\acl(A)/\dcl(A)\big),$$
$$\CG(b/A):=\CG(\dcl(Ab)/A),\qquad [B:A]:=|\CG(B/A)|.$$
\end{definition}

\begin{definition}\label{galois.ext.def}
\begin{enumerate}
\item Assume that $A\subseteq B$ are $\mathcal{L}$-substructures of $\mathfrak{C}$. We say that $B$ is \emph{normal over $A$} (or we say that $A\subseteq B$ is a \emph{normal extension}) if $\CG(\mathfrak{C}/A)\cdot B\subseteq B$.
(Note that if $B$ is small and $A\subseteq B$ is normal, then it must be $B\subseteq\acl(A)$.)

\item Assume that $A\subseteq B\subseteq\acl(A)$ are small $\mathcal{L}$-substructures of $\mathfrak{C}$ such that $A=\dcl(A)$, $B=\dcl(B)$ and $B$ is normal over $A$. In this situation we say that $A\subseteq B$ is a \emph{Galois extension}.
\end{enumerate}
\end{definition}

\begin{definition}\label{regular.def}
Let $E\subseteq A$ be small subsets of $\mathfrak{C}$. We say that $E\subseteq A$ is \emph{$\mathcal{L}$-regular} (or just \emph{regular}) if
$$\dcl(A)\cap\acl(E)=\dcl(E).$$
\end{definition}

\begin{definition}\label{def:PAC}
Assume that $M\preceq\mathfrak{C}$ and $P$ is a substructure of $M$. 
We say that \emph{$P$ is PAC in $M$} if for every regular extension $N$ of $P$ in $M$ (i.e. $N\subseteq M$ and $N$ is regular over $P$), the structure $P$ is existentially closed in $N$.
If $P$ is PAC in $\mathfrak{C}$, then we say that $P$ is PAC.
\end{definition}

For a more detailed exposition of the notions of regularity and PAC structures, the reader may consult \cite[Section 3.1]{Hoff3}.

\begin{fact}\label{fact:PAC_elemsubst}
Let $P$ be PAC, and let $P_0\preceq P$. Then $P_0$ is PAC.
\end{fact}

\begin{proof}
Let $N\subseteq \FC$ be a regular extension of $P_0$. Suppose $N\models \exists x\,\varphi(x,e)$ for a quantifier free $\varphi(x,y)$ and $e\in P_0^{|y|}$, say $\FC\models\varphi(n,e)$ for an $n\in N$.
We may find $n'\in\mathfrak{C}$ such that $\mathfrak{C}\models\varphi(n',e)$ and $P\ind_{P_0}n'$ (in $\mathfrak{C}$).	

Since $P\ind_{P_0}n'$, by \cite[Lemma 3.39]{Hoff3}, $\dcl(P,n')$ is a regular extension of $P$.
Because $P$ is PAC, we have that $P\models \exists x\,\varphi(x,e)$, which immediately implies that $P_0\models \exists x\,\varphi(x,e)$.
\end{proof}

The notion of a \emph{sorted isomorphism} of absolute Galois groups was introduced in \cite{DHL1}, and the main theorem of \cite{DHL1} is also proven after using this notion in the proof of \cite[Lemma 5.5]{DHL1}. 
Let us collect here important notions and facts from \cite{DHL1}, which will be used in the rest of this paper.

\begin{notation}
Let $I$ be a set. 
\begin{enumerate}
	\item Let $I^{<\omega}$ be the set of finite tuples of elements in $I$.
	\item For $J,J'\in I^{<\omega}$, we write $J\leqslant J'$ if $J$ is a subtuple of $J'$
	(i.e. if $J=(j_1,\ldots,j_n)$, then any $J=(j_{s_1},\ldots,j_{s_{n'}})$, where $n'\leqslant n$ and $1\leqslant s_1<\ldots<s_{n'}\leqslant n$, is a subtuple of $J'$).
	\item For $J,J'\in I^{<\omega}$, we write $J^\smallfrown J'$ for the concatenation of $J$ and $J'$.
	\item For $J=(j_1,\ldots,j_n)\in I^{<\omega}$, set $|J|=n$.
	\item For $J=(j_1,\ldots,j_n)\in I^{<\omega}$ and a permutation $\sigma\in \Sym(n)$, $\sigma(J)=(j_{\sigma(1)},\ldots,j_{\sigma(n)})$.
	\item For $J=(j_1,\ldots,j_n)\in I^{<\omega}$, we write $S_J=S_{j_1}\times\cdots\times S_{j_n}$.

\item $\aut_{J}(L/K)$ is the image of the restriction map $\CG(L/K)\to\aut(S_J(L)/S_J(K))$, where $K\subseteq L$ is a Galois extension of small substructures of $\FC$.
\end{enumerate}
\end{notation}

\begin{definition}
An element $a\in\FC$ is a \emph{primitive element} of a Galois extension $K\subseteq L$ if $L=\dcl(K,a)$. By $e(L/K)$ we denote the subset of $\FC$ of all primitive elements of the Galois extension $K\subseteq L$.
\end{definition}

\begin{fact}[Primitive Element Theorem]\cite[Proposition 4.3]{DHL1}
If $|\CG(L/K)|<\omega$ for a Galois extension $K\subseteq L$, then $e(L/K)\neq\emptyset$.
\end{fact}

For a topological group $G$, we define $\CN(G)$ as the family of all open normal subgroups of $G$.
For small substructure $K$ of $\FC$ and $N\in\CN(\CG(K))$, we put 
$$\mathcal{PE}_K(N):=\Big\{J\in \mathcal{J}^{<\omega}\;|\;\big(\exists a\in S_J(\FC)\big)\Big(a\in e\big(\acl(K)^N/K\big)\Big)\Big\},$$
$$\CF_K(N):=\{J\in \mathcal{J}^{<\omega}\;|\;	|\aut_J(\acl(K)^{N}/K)|=|\CG(\acl(K)^N/K)| \}.$$
If there is no confusion, we skip the subscript and write $\mathcal{PE}$ and $\CF$ for $\mathcal{PE}_K$ and $\CF_K$ respectively.

\begin{definition}\label{def:sorted_map1}
Assume that $F$ and $E$ are small substructures of $\mathfrak{C}$ and $\pi:\CG(F)\to\CG(E)$ is a continuous epimorphism. 
We say that $\pi$ is \emph{sorted} if 
for each $N\in\CN(\CG(E))$ we have $\CF_E(N)\subseteq\CF_F(\pi^{-1}[N])$.

We say that $\pi$ is a sorted isomorphism if $\pi$ is an isomorphism of profinite groups such that $\pi$ and $\pi^{-1}$ are sorted.
\end{definition}

The following fact is an improvement of results obtained in \cite[Proposition A.11]{DHL1}.

\begin{fact}\label{fact:regular_to_sorted}
Assume that $E\subseteq F$ is an extension of small substructures of $\FC$ such that $E=\dcl(E)$. 
Let $\pi$ be the restriction map $\pi:\CG(F)\to \CG(E)$.
The following conditions are equivalent.
\begin{enumerate}
\item $E\subseteq F$ is regular.
\item $\pi$ is a sorted epimorphism.
\item $\pi$ is an epimorphism.
\end{enumerate}
Moreover, if any of the above listed condition holds, then
for each $N\in\CN(\CG(E))$, if $a\in e(\acl(E)^N/E)$, then $a\in e(\acl(F)^{\pi^{-1}[N]}/F)$.
\end{fact}

\begin{proof}
1) $\Rightarrow$ 2)\hspace{3mm} We start with showing that $\pi$ is onto.
It is enough to show that the restriction map induces an epimorphism $\CG(F)$ to $\CG(\acl(E)^N/E)$ for each open normal subgroup $N$ of $\CG(E)$.

Fix an open normal subgroup $N$ of $\CG(E)$.
Let $b$ be a primitive element of $\acl(E)^N$ over $E$.
\
\\
\textbf{Claim}
$\CG(E)\cdot b=\CG(F)\cdot b$
\
\\ Proof of the claim: 
It is clear that $\CG(F)\cdot b\subseteq \CG(E)\cdot b$.
Let $c$ be the code of $\CG(F)\cdot b$. It is algebraic over $E$ so $c$ is contained in $\acl(E)$.
Also, $c$ is contained in the definable closure of $F$.
Therefore, $c\in\dcl(F)\cap\acl(E)=\dcl(E)=E$ (by regularity of $F$ over $E$).
So, we conclude that $\CG(E)\cdot b$ is a subset of $\CG(F)\cdot b$.
Here ends the proof of Claim.

Choose $\sigma\in \CG(\acl(E)^N/E)$ arbitrary and define $b'$ to be $\sigma(b)$. Since $b'\in \CG(E)\cdot b=\CG(F)\cdot b$, there exists $\hat{\sigma}\in \CG(F)$ such that $\hat{\sigma}(b)=b'$. Because $b$ is a primitive element of $\acl(E)^N$ over $E$, we have that $\sigma=\hat{\sigma}|_{\acl(E)^N}$.

To show that $\pi$ is sorted it is enough to follow the proof of \cite[Proposition A.11]{DHL1} (the assumption about stability in the proof of \cite[Proposition A.11]{DHL1} was used to show that $\pi$ is onto, since we already achieved this, we may use the rest of the proof).

\noindent
2) $\Rightarrow$ 3) \hspace{3mm} Obvious. \\
\noindent
3) $\Rightarrow$ 1) \hspace{3mm} It is standard, but let us give the proof.
Suppose that there is an element $a\in \dcl(F)\cap\acl(E)\setminus E$. There exists $\sigma\in \CG(E)$ such that $\sigma(a)\neq a$ and there is also $\hat{\sigma}\in \CG(F)$ such that $\hat{\sigma}|_{\acl(E)}=\sigma$. Then $a\neq \sigma(a)=\hat{\sigma}(a)=a$ (since $a\in \dcl(F)$).

The moreover part also is covered by the proof of \cite[Proposition A.11]{DHL1}
\end{proof}

\begin{cor}\label{cor2033}\cite[Corollary 3.7]{DHL1}
Assume that $T$ is stable. If $E$ and $F$ are $\kappa$-PAC substructures of $\mathfrak{C}$ ($\kappa\geqslant|T|^+$,
please consult \cite[Definition 3.1]{PilPol}
for the notion of \emph{$\kappa$-PAC substructure}),
and for some definably closed $L\subseteq F\cap E$ of size strictly smaller than $\kappa$ there exists a continuous isomorphism $\Phi:\CG(F)\to\CG(E)$ such that
$$\xymatrix{\CG(F) \ar[dr]_{\res} \ar[rr]^{\Phi}& & \CG(E)\ar[dl]^{\res} \\
& \CG(L) &
}$$
, then $E\equiv_L F$.
\end{cor}

The following definition is a modification of \cite[Definition 3.3]{PilPol}, which was already used in the main result of \cite{DHL1} (see Theorem \ref{thm:elementary_invariance}).

\begin{definition}\label{def:PAC.first.order}
We say that \emph{PAC is a first order property in $T$} if there exists a set $\Sigma$ of $\mathcal{L}$-sentences such that for any $M\models T$ and $P\subseteq M$
$$P\models \Sigma\qquad\iff\qquad P\text{ is PAC}.$$
\end{definition}

\begin{theorem}[Elementary Equivalence Theorem for Structures]\label{thm:elementary_invariance}\cite[Theorem 5.11]{DHL1}
Assume that $T$ is stable.
Suppose PAC is a first order property.
Assume that
\begin{itemize}
\item $K$, $L$, $M$, $E$, $F$ are small definably closed substructures of $\mathfrak{C}$,
\item $K\subseteq L\subseteq E$, $K\subseteq M\subseteq F$,
\item $F$ and $E$ are PAC,
\item $\phi_0\in\aut(\mathfrak{C}/K)$ is such that $\phi_0[L]=M$,
\item $\Phi:\mathcal{G}(F)\to\mathcal{G}(E)$ is a \textbf{sorted isomorphism} such that
$$\xymatrix{
\mathcal{G}(F)\ar[r]^{\Phi} \ar[d]_{\res} & \mathcal{G}(E) \ar[d]^{\res}\\
\mathcal{G}(M)\ar[r]_{\Phi_0} & \mathcal{G}(L)
}$$
where $\Phi_0(\sigma):=\phi_0^{-1}\circ \sigma\circ\phi_0$.
\end{itemize}
Then $E\equiv_K F$.
\end{theorem}

\begin{cor}\label{final_cor}
Let $T$ be stable.
Suppose PAC is a first order property.
If the restriction map $\res:\CG(F)\to\CG(E)$, where $E\subseteq F$ are PAC structures, is a  sorted isomorphism, then $E\preceq F$. 
\end{cor}

In the case of $T$ being a stable theory,
we see that for a regular extension of PAC structures $E\subseteq F$ the restriction map $\res:\CG(F)\to\CG(E)$ is sorted and if it is a sorted isomorphism, then the embedding $E\subseteq F$ is elementary. We want to develop a first order language for profinite groups (similarly as in \cite{cherlindriesmacintyre}) which will encode ``being a sorted map" and which will distinguish maps corresponding to elementary embeddings. One of our goals is to find an appropriate property in the place of ``?" in the following picture, where $E\subseteq F$ is an extension of PAC structures (in the case of stable $T$):
$$\xymatrix{
E\subseteq F \text{ is regular} \ar@{<~>}[r] & \res:\CG(F)\to\CG(E) \text{ is sorted} \\
E\subseteq F \text{ is an elementary embedding} \ar@{<~>}[r] & ?
}$$

\section{Sorted groups and systems}\label{sec:sorted}
\subsection{Sorted profinite groups}

In this subsection, we equip profinite groups with ``sorting data", i.e. a family of sets of finite tuples of sorts, 
which should recognize where (i.e. on which sorts) primitive elements of finite Galois extensions live
(if the given profinite group is an absolute Galois group). Because we model ``sorted profinite groups" on absolute Galois groups 
which encode presence of primitive elements, let us first note a property which holds in such absolute Galois groups. This property occurs in Definition \ref{def:sorted_profinite_gp} and is related to the ``modular lattice axioms" from Subsection \ref{sec:sorted_systems}.

\begin{remark}
Suppose that $\dcl(K)=K$ is a small substructure of $\FC$.
There exist functions $J^*_{\subset}:\omega\times\mathcal{J}^{<\omega} \to\mathcal{J}^{<\omega}$
and $J^*_{\cap}:(\mathcal{J}^{<\omega})^{\times 2}\to\mathcal{J}^{<\omega}$ satisfying the following points.
\begin{enumerate}
\item If $K\subseteq L$ is a Galois extension and $J\in\mathcal{J}^{<\omega}$ is such that 
 $|\aut_{J}(L/K)|=|\CG(L/K)|\leqslant k$, then for any 
 Galois extension $K\subseteq D$ such that $D\subseteq L$, we have that
 $|\aut_{J^*_{\subset}(k,J)}(D/K)|=|\CG(D/K)|\leqslant k$.
 
\item If $K\subseteq L_1$ and $K\subseteq L_2$ are Galois extensions and $J_1,J_2\in\mathcal{J}^{<\omega}$ are such that 
$|\aut_{J_i}(L_i/K)|=|\CG(L_i/K)|\leqslant k_i$ for $i=1,2$, then
$$|\aut_{J^*_{\cap}(J_1,J_2)}(\dcl(L_1,L_2)/K)|=|\CG(\dcl(L_1,L_2)/K)|\leqslant k_1k_2.$$
\end{enumerate}
\end{remark}

\begin{proof}
Note that 
if $|\aut_{J}(L/K)|=|\CG(L/K)|=n<\omega$, then $L$ has a primitive element in $J^{\smallfrown n}$ (e.g. see the last part of the proof of \cite[Proposition A.10]{DHL1}).

Assume that $L=\dcl(K,a)$ is a finite Galois extension of $K$ with a primitive element $a\in J^{\smallfrown [L:K]}$ and that $[L:K]\leqslant k$.
Let $K\subseteq D$ be any Galois extension such that $D\subseteq L$ and let $J_n'$ be the sort corresponding to the sort of codes for sets of $n$-many elements from $J$. The element $\ulcorner\CG(L/D)\cdot a\urcorner$ is a primitive element of $D$ over $K$ and belongs to some $J_n'$, where $n\leqslant [L:K]$.
Then for $J^*_{\subset}(k,J):=J_1'^{\smallfrown}J_2'^{\smallfrown}\ldots^{\smallfrown} J_{k}'$ we have that $|\CG(D/K)|=|\aut_{J_{\subset}^*(k,J)}(D/K)|$.

Let $L_1=\dcl(a,K)$ and $L_2=\dcl(b,K)$ be Galois extensions of $K$ with primitive elements $a\in J_1^{\smallfrown [L_1:K]}$ and $b\in J_2^{\smallfrown [L_2:K]}$ respectively.
Then $(a,b)\in J_1^{\smallfrown [L_1:K]}\hspace{0mm}^{\smallfrown}J_2^{\smallfrown [L_2:K]}$ is a primitive element of $\dcl(L_1,L_2)=\dcl(K,a,b)$ (over $K$). Set $J^*_{\cap}(J_1, J_2):= J_1^{\smallfrown}J_2$.
\end{proof}
\noindent
If $T$ is one-sorted, then instead of working in $T^{\eq}$ one may consider working in $T$, since the above remark trivially holds for one-sorted theories even without assuming any variant of elimination of imaginaries.

\begin{notation}
For $J=(S_1,\ldots,S_n)\in \CJ^{<\omega}$, let $||J||:=\{S_1,\ldots,S_n\}$. We define 
$$\sqrt{J}:=\{J'\in \CJ^m:\ ||J'||\supseteq||J||,\ m\geqslant 1\}.$$
\end{notation}

\begin{definition}\label{def:sorted_profinite_gp}
Let $G$ be a profinite group and 
let $\bar{\CF}=\{\CF(N)\subseteq\mathcal{J}^{<\omega}\;|\; N\in \CN(G)\}$
(for some choice of $\CF(N)$'s).
We say that $(G,\bar{\CF})$ is a \emph{sorted profinite group} if
for $N, N_1,N_2\in \CN(G)$,
\begin{enumerate}
\item $J\in \CF(N)\Leftrightarrow \sqrt{J}\subseteq \CF(N)$.
\item for $J:=(S_{j_1},\ldots,S_{j_n})\in\CF(N)$ and $\sigma\in \Sym(n)$ we have that 
$\sigma(J)\in\CF(N)$.

\item if $N_1\subseteq N_2$, $[G:N_1]\leqslant k$ and $J\in\CF(N_1)$, then $J^*_{\subset}(k,J)\in\CF(N_2)$.

\item  for any $J_1\in \CF(N_1)$ and $J_2\in \CF(N_2)$ we have that
	$$J^*_{\cap}(J_1,J_2)\in \CF(N_1\cap N_2).$$
	
\item for $N_1,N_2\in \CN(G)$ and $g\in G$, if $g^{-1}N_1g=N_2$, that is, $\Phi_g[N_1]=N_2$ for the inner automorphism $\Phi_g:x\mapsto g^{-1}xg$ on $G$, then $$\CF(N_1)=\CF(N_2).$$
\end{enumerate}
\end{definition}

\begin{example}
Let $\CL=\{+,\,-,\,\cdot,\,0,\,1\}$ be the language of rings with one sort $S_{=}$ so $\CJ=\{S_{=}\}$. Let $T=ACF_p$ be the complete theory of algebraically closed fields of characteristic $p$. 
Assume that $\FC$ is a saturated algebraically closed field of characteristic $p$. 
Take a perfect subfield $K\subseteq \FC$ (so it is definably closed) and set $G:=\CG(K)$. We define
$\CF(N):=\mathcal J^{<\omega}$ for each $N\in \CN(G)$. 
Then $(G,\bar{\CF})$ forms a sorted profinite group.
\end{example}

In accordance with Definition \ref{def:sorted_map1}, we introduce the following notion.

\begin{definition}\label{def:sorted_map2}
Assume that $(G_1,\bar{\CF}_1)$ and $(G_2,\bar{\CF}_2)$ are sorted profinite groups. A \emph{morphism of sorted profinite groups} $\pi: (G_1,\bar{\CF}_1)\to(G_2,\bar{\CF}_2)$ is a continuous epimorphism $\pi:G_1\to G_2$ such that
for each $N\in\CN(G_2)$ we have $\CF_2(N)\subseteq\CF_1(\pi^{-1}[N])$.
\end{definition}

Then, the family of sorted profinite groups with morphisms of sorted profinite groups forms a category. Note that being a sorted isomorphism as in Definition \ref{def:sorted_map1} corresponds to being an isomorphism of sorted profinite groups. Now, we will define a functor taking a category of regular extensions of small substructures of $\FC$ into the category of sorted profinite groups. For a small definably closed substructure $F$ of $\FC$, recall that for every $N\in\CN(\CG(F))$ (i.e. open normal subgroup):
$$\CF_F(N):=\{J\in \mathcal{J}^{<\omega}\;|\;	|\aut_J(\acl(F)^{N}/F)|=|\CG(\acl(F)^N/F)| \}.$$ Define $\bar{\CF}(F):=\big(\CF_F(N)\big)_{N\in\CN(\CG(F))}$.


\begin{remark}	
We consider a category of small definably closed substructures of $\mathfrak{C}$ whose morphism between $E$ and $F$ is an $\CL$-embedding $\phi:\acl(E)\rightarrow \acl(F)$ such that $F$ is a regular extension of $\phi[E]$, which is an elementary map by quantifier elimination of $T$. Then, we can define a functor $\CG$ from this category into the category of sorted profinite groups. The functor $\CG$ is given by 
$$F\mapsto (\CG(F),\bar{\CF}(F)),$$
$$\phi:E\to F\quad\mapsto\quad \CG(\phi):(\CG(F),\bar{\CF}(F))\to(\CG(E),\bar{\CF}(E))$$
where $\CG(\phi)$ is the composition of the restriction map $\res:\CG(F)\to\CG(\phi[E])$ and the isomorphism $\Phi:\CG(\phi[E])\to\CG(E),\sigma\mapsto \phi^{-1}\circ\sigma\circ \phi$ (which is a morphism of sorted profinite groups by Fact \ref{fact:regular_to_sorted}). The map $\CG(\phi)$ is called the {\em dual} of $\phi$.

\end{remark}

\subsection{Sorted complete systems}\label{sec:sorted_systems}
There is a standard way to study profinite groups in model theory (e.g \cite{cherlindriesmacintyre}, 
\cite{zoeIP}). The point is to avoid arguments based on ``infinite topology", by formulating everything in terms of finite quotients (from which this topology arises) of a given profinite group.
The same scheme works for sorted profinite groups, although we need to consider a different collection of sorts on which we set our first-order structure corresponding to a sorted profinite group.

We introduce language $\mathcal{L}_G(\mathcal{J})$ over sorts $m(k,J)$ where $k<\omega$ and $J\in\mathcal{J}^{<\omega}$ as follows. The language $\mathcal{L}_G(\mathcal{J})$ consists:
\begin{itemize}
\item a family of binary relations $\leqslant_{k,k',J,J'}$, $C_{k,k',J,J'}$ ``evaluated" on elements of $m(k,J)\times m(k',J')$,

\item a family of ternary relations $P_{k,J}$ ``evaluated" on elements of $m(k,J)^{\times 3}$.
\end{itemize}

Usually, if there is no confusion,
we will skip the subscripts and 
write only ``$\leqslant$", ``$C$" and ``$P$". 
The same with elements of a $\mathcal{L}_G(\mathcal{J})$-structure: we will use ``$a$" and ``$(a,k,J)$" to denote the same element $a\in m(k,J)$.

\begin{definition}\label{def:sorted complete system}
We call an $\mathcal{L}_G(\mathcal{J})$-structure $(S,\leqslant,C,P)$ a \emph{sorted complete system} if
the following (first order) axioms and axiom schemes are satisfied:
\begin{enumerate}
\item \begin{itemize}
\item (order): $\leqslant$ is reflexive and transitive on $S$.
\item (maximal elements 1): $|m(1,J)|=1$, where $J\in\mathcal{J}^{<\omega}$.
\item (maximal elements 2): $(\forall\, x\in m(1,J),\, y\in m(k',J'))\,(\,y\leqslant x\,)$, where $J,J'\in\mathcal{J}^{<\omega}$ and $0<k'<\omega$.
\end{itemize}

\item
Define $x\sim y$ as $x\leqslant y\,\wedge\,y\leqslant x$.
Denote the $\sim$-class of $a$ by $[a]$ for $a\in m(k,J)$ and set $[a]_{k,J}:=[a]\cap m(k,J)$ (which is definable).
\begin{itemize}
\item (extending tuples): $(\forall\,a\in m(k,J))\,(\exists\,a'\in m(k',J'))\,(a~\sim a')$, where $k\leqslant k'$ and $J\leqslant J'$.

\item (permutations): $(\forall\,a\in m(k,J))\,(\exists\,a'\in m(k,\sigma(J))\,(a\sim a')$, where $k<\omega$, $J\in\mathcal{J}^{<\omega}$ and $\sigma$ is a permutation on the tuple $J$.

\item (finiteness): $(\forall\,a\in m(k,J))(\,|[a]_{k,J}|\leqslant k\,)$

\item (reducing degree): $(\forall\,a\in m(k,J))\,(\,|[a]_{k,J}|\leqslant n\,\rightarrow\,(\exists\,a'\in m(n,J))\,(a\sim a'))$, where $n\leqslant k<\omega$ and $J\in\mathcal{J}^{<\omega}$.
\end{itemize}

\item \begin{itemize}
\item (intersection $H\cap H'$): 
$$\big(\forall\,x\in m(k,J),\,y\in m(k',J'),\,z\in m(k'',J'')\big)\,\big(z\leqslant x\,\wedge\,z\leqslant y\,\rightarrow$$
$$(\exists\,w\in m(kk',J^*_{\cap}(J,J')))\,(z\leqslant w\,\wedge\,w\leqslant x\,\wedge\,w\leqslant y)\big)$$

\item (subgroup $H\subseteq H'$):
$$\big(\forall\,x\in m(k,J),\,y\in m(k',J')\big)\,\big(x\leqslant y\,\rightarrow$$
$$(\exists\,y'\in m(k,J^*_{\subset}(J,k)))\,(y\sim y')\big)$$

\item (inf):
Suppose that $a\in m(k,J)$ and $b\in m(k',J')$, we define an $\mathcal{L}_G(\mathcal{J})$-formula
$\varphi^{\inf}_{a,b}(x)$ as follows
$$x\leqslant a\,\wedge\,x\leqslant b\,\wedge$$
$$\big(\forall\,y\in m(kk',J^*_{\cap}(J,J'))\big)\,(y\leqslant a\,\wedge\,y\leqslant b\,\rightarrow\,y\leqslant x).$$
We require that the following holds in $S$
$$\big(\forall\,a\in m(k,J),\,b\in m(k',J'))\,(\exists\,x\in m(kk',J^*_{\cap}(J,J'))\big)\, \big(\,\varphi^{\inf}_{a,b}(x)\,\big).$$

\item (sup):
Suppose that $a\in m(k,J)$ and $b\in m(k',J')$, we define an $\mathcal{L}_G(\mathcal{J})$-formula
$\varphi^{\sup}_{a,b}(x)$ as follows
$$a\leqslant x\,\wedge\,b\leqslant x\,\wedge$$
$$\big(\forall y\in m(kk',J^*_{\subset}(J,k))\big)\,(a\leqslant y\,\wedge\,b\leqslant y\,\rightarrow\,x\leqslant y)).$$
We require that the following holds in $S$
$$\big(\forall\,a\in m(k,J),\,b\in m(k',J'))\,(\exists\,x\in m(kk',J^*_{\subset}(J,k))\big)\, \big(\,\varphi^{\sup}_{a,b}(x)\,\big).$$
\end{itemize}

\item
For each $a\in m(k,J)$ and $b\in m(k',J')$ 
we define 
$[a]\wedge[b]:=[c]$, where $c\in m(kk',J^*_{\cap}(J,J')$ is such that $\varphi^{\inf}_{a,b}(c)$ holds,
and
$[a]\vee[b]:=[d]$, where $d\in m(kk',J^*_{\subset}(J,k))$ is such that $\varphi^{\sup}_{a,b}(d)$ holds .
\begin{itemize}
\item (lattice): Note that $(S/\sim,\leqslant,\vee,\wedge)$ forms a lattice.

\item (modular law): We require that $[a]\leqslant[b]$ implies that $[a]\vee ([c]\wedge[b])=([a]\vee[c])\wedge{b}$ which can be expressed as a first order axiom scheme.
\end{itemize}

\item (group structure): $P\subseteq \bigcup\limits_{k,J}\bigcup\limits_{a\in m(k,J)}\big(\,[a]_{k,J}\,\big)^{\times 3}$ and $P$ is the graph of a binary operation making $[a]_{k,J}$ into a finite group of order at most $k$.

\item
\begin{itemize}
\item $C(x,y)\;\rightarrow\;x\leqslant y$
\item (projections): For all $a\in m(k,J)$ and $b\in m (k',J')$, if $a\leqslant b$ then
$C\cap ([a]_{k,J}\times [b]_{k',J'})$ is the graph of a group epimorphism $\pi_{a,b}:[a]_{k,J}\to[b]_{k',J'}$.
\item (compatible system 1): $\pi_{a,a}=\id_{[a]_{k,J}}$ for all $a\in m(k,J)$, and all $k<\omega$ and $J\in\mathcal{J}^{<\omega}$.
\item (compatible system 2): If $a\leqslant b\leqslant c$ then $\pi_{b,c}\circ\pi_{a,b}=\pi_{a,c}$.
\end{itemize}

\item (normal subgroups): $(\forall\, a\in m(k,J))\,(\forall\; N\trianglelefteq [a]_{k,J})\,(\exists !\,b\in m(k,J^*_{\subset}(J,k)))$
$(C(a,b)\,\wedge\,N=\{a^{-1}c\;|\;c\in[a]_{k,J}\,\wedge\,C(c,b)\})$.

\item (hidden axiom): $(\forall\,a\in m(k,J),\,b\in m(k',J'),\,c\in m(k'',J''))\,(a\leqslant b\,\wedge\,a\leqslant c\,\wedge\,\ker\pi_{a,b}=\ker\pi_{a,c}\,\rightarrow\,b\sim c)$.

\item (invariant under the inner automorphisms) For $a\in m(k,J),b\in m(k',J')$ and for $c\in m(kk',J_{\cap}^*((k,J),(k',J'))$ with $[a]\wedge [b]=[c]$, if there is $d\in [c]\cap m(kk',J_{\cap}^*((k,J),(k',J'))$ such that $\Phi_d[N_a^c]=N_b^c$ for the inner automorphism $\Phi_d$ on $[c]_{kk',J_{\cap}^*((k,J),(k',J')}$, where $N_a^c$ is the kernel of $\pi_{c,a}:[c]_{kk',J_{\cap}^*((k,J),(k',J')}\rightarrow [a]_{k,J}$, then for any $(k'',J'')\in \omega\times \CJ^{<\omega}$, $$[a]\cap m(k'',J'')\neq \emptyset\Leftrightarrow [b]\cap m(k'',J'')\neq \emptyset.$$
\end{enumerate}
The set of consequences of the above axioms and axiom schemes will be denoted by $\scs$ (i.e. the theory of Sorted Complete Systems).
\end{definition}

The axiom scheme 8. in the above definition is needed also in the case corresponding to (one-sorted) fields, but (to our knowledge) it was not stated explicitly up to this point, hence we call it the ``hidden axiom". Example \ref{ex:hidden.axiom} shows that axiom scheme from point 8. does not follow from the previous axioms.

\begin{example}\label{ex:hidden.axiom}
Put $X_0:=\{x_0^0,x_1^0,x_2^0,x_3^0\}$, $X_1:=\{x_0^1,x_1^1\}$, $X_2:=\{x_0^2,x_1^2\}$, and $X_3:=\{x_0^3\}$.

\begin{enumerate}
	\item Define a binary relation $\le_{i,j}'\subseteq X_i\times X_j$ for $i,j\le 3$ as follows: Put $\le_{i,j}'=\emptyset$ if $i>j$ or $\{i,j\}=\{1,2\}$, and put $\le_{i,j}'=X_i\times X_j$ otherwise. Set $\le':=\bigcup_{i,j\le 3} \le_{i,j}'$.

	\item Define a binary relation $C_{i,j}'\subseteq X_i\times X_j$ for $i,j\le 3$ as follows: Put $C_{i,j}'=\emptyset$ if $i>j$ or $\{i,j\}=\{1,2\}$, $C_{0,j}'=\{(x^0_p,x^i_q)|\ p=q (\mod 2)\}$ if $j\in \{1,2\}$, and put $C_{i,j}'=X_i\times X_j$ otherwise. Set $C':=\bigcup_{i,j\le 3}C'_{i,j}$.

	\item Define a ternary relation $P'_i\subseteq X_i\times X_i\times X_i$ for $i\le 3$ as follows: $P'_0:=\{(x^0_p,x^0_q,x^0_r)|\ p+q=r(\mod 4)\}$, $P'_i:=\{(x^i_p,x^i_q,x^i_r)|\ p+q=r(\mod 2)\}$ if $i\in \{1,2\}$, and $P_3':=X_3\times X_3\times X_3$. Set $P':=\bigcup_{i\le 3}P'_i$.
\end{enumerate}

For a set $X$, define $\epsilon_k: X\rightarrow X^k,x\mapsto (x,\ldots,x)$, and $\delta_k:X^k\rightarrow X,(x_0,\ldots, x_{k-1})\mapsto x_0$, where $k\ge 1$. Note that the functions $\epsilon_k$'s and $\delta_k$'s are serving a superficial coding role rather than a substantive role in the argument (to make $X_i$'s disjoint to fit into the formalism). If there is no risk for a confusion, we skip $k$. Set $m(1):=\epsilon_1(X_3)$, $m(2):=\epsilon_2(X_1)\sqcup \epsilon_2(X_2)\sqcup \epsilon_2(X_3)$, $m(3):=\epsilon_3(X_1)\sqcup \epsilon_3(X_2)\sqcup \epsilon_3(X_3)$, and $m(k):=\sqcup_{i\ge 0} \epsilon_k(X_i)$ for $k\ge 4$.

\begin{enumerate}
	\item For $k_1,k_2\ge 1$ and for $\alpha\in m(k_1)$ and $\beta \in m(k_2)$, define $\le_{k_1,k_2}$ and $C_{k_1,k_2}$ as follows:
\begin{itemize}
	\item $\le_{k_1,k_2}(\alpha,\beta)$ if and only if $\le'(\delta(\alpha),\delta(\beta))$, and
	\item $C_{k_1,k_2}(\alpha,\beta)$ if and only if $C'(\delta(\alpha),\delta(\beta))$.
\end{itemize}

	\item For $k\ge 1$ and for $\alpha,\beta,\gamma \in m(k)$, define $P_k$ as follows: $P_k(\alpha,\beta,\gamma)$ if and only if $P'(\delta(\alpha),\delta(\beta),\delta(\gamma))$.
\end{enumerate}

Now, we consider an $\mathcal{L}_G(\mathcal{J})$-structure (where $|\mathcal{J}|=1$) $S=(m(k),\le_{k_1,k_2},C_{k_1,k_2},P_k)$ and we can check that $S$ satisfies axioms on  \cite[page 979.]{chatzidakis2002}, but $S$ does not satisfy our new additional axiom scheme (``hidden axiom").
To see the last thing, note that:
\begin{enumerate}
	\item $[x^i]=X_i$ for each $x^i\in X_i$, and so $[x^1]\neq [x^2]$.
	\item $[x^0]\cong \BZ/4\BZ$, $[x^1]\cong \BZ/2\BZ\cong [x^2]$, and $[x^3]=0$. 
	\item $\pi_{0,i}:[x^0]\rightarrow [x^i]$, $x^0_p\mapsto x^i_q$ such that $p=q(\mod 2)$ for $i\in \{1,2\}$.
\end{enumerate}

\noindent
From the above $(1),(3)$, we have that $\ker(\pi_{0,1})=\ker(\pi_{0,2})=\{x^0_0,x^0_2\}$, but $x^1_i\not\sim x^2_j$ for any $i,j\in \{1,2\}$. 
\end{example}

Now, we will establish the following correspondence between categories:
$$\xymatrix{\Big\{\text{sorted profinite groups}\Big\} \ar@/^/[r]^{S}  &  \Big\{\text{sorted complete systems} \Big\} \ar@/^/[l]^{G}}$$
We only define desired maps and leave checking details to the reader, since precise arguments will significantly increase the number of pages of this paper and most of these arguments are just standard ``diagram chasing".

If we start with a sorted profinite group $(G,\bar{\CF})$, a functor $S$ attaches to $(G,\bar{\CF})$ the sorted complete system $S(G)$ defined in the following way:
\begin{itemize}
\item $m(k,J)(S(G)):=\{gH\;|\;g\in G, H\in\CN(G), [G:H]\leqslant k, J\in\CF(H)\}$.
\item if $gH\in m(k,J)(S(G))$ and $g'H'\in m(k',J')(S(G))$, then we set 
$$gH\,\leqslant_{k,k',J,J'}\, g'H'\qquad\iff\qquad H\subseteq H'.$$
\item similarly
$$C_{k,k',J,J'}(gH,g'H')\quad\iff\quad H\subseteq H'\text{ and }gH'=g'H'.$$
\item if $g_1H_1,g_2H_2,g_3H_3\in m(k,J)$ then we set
$$P_{k,J}(g_1H_1,g_2H_2,g_3H_3)\quad\iff\quad H_1=H_2=H_3\text{ and }g_1g_2H_1=g_3H_1.$$
\end{itemize}
Any morphism of sorted profinite groups $\pi:(G_1,\bar \CF_1)\to (G_2,\bar \CF_2)$ leads to an $\mathcal{L}_G(\mathcal{J})$-embedding $S(\pi):S(G_2)\to S(G_1)$ given by
$$S(\pi)(gH):=g'\pi^{-1}[H]\in m(k,J)(S(G_1))$$
where $gH\in m(k,J)(S(G_2))$ and $g'\in G_1$ is any element such that $\pi(g')=g$.

If we start with a sorted complete system $(S,\leqslant,C,P)$, then the
collection of $\pi_{a,b}:[a]_{k,J}\to [b]_{k',J'}$, where $a\in m(k,J)(S)$ and $b\in m(k',J')(S)$,
forms a projective system of finite groups. Therefore we define a functor $G$ on $S$ as
$$G(S):=\lim\limits_{\leftarrow}\;[a]_{k,J}.$$
From the axioms of a sorted complete system, it follows that for each open normal subgroup $N$ of $G(S)$ there is some $a\in m(k,J)(S)$ such that $N=N_{a,k,J}$, where $N_{a,k,J}$ is the kernel of the epimorphism $G(S)\to [a]_{k,J}$ coming from the definition of a projective limit. Therefore we can define $\CF(N)$ for $N\in\CN(G(S))$ in the following way:
$$\CF(N):=\{J\in\mathcal{J}^{<\omega}\;|\;N=N_{a,k,J} \text{ for some }a\in m(k,J)(S)\}.$$
If $f:S\to S'$ is an $\mathcal{L}_G(\mathcal{J})$-embedding between sorted complete systems, then,
since $f:[a]_{k,J}\cong [f(a)]_{k,J}$, the embedding $f$ induces an epimorphism $G(f):G(S')\to G(S)$
(it is not necessarily an isomorphism, since projective systems corresponding to $G(S)$ and $G(S')$ are indexed by different sets of elements, we use here e.g. \cite[Lemma 1.1.5]{ribzal}).

Let us now describe the canonical isomorphisms $\beta: S\to SG(S):=S(G(S))$ and $\alpha:G\to GS(G):=G(S(G))$ needed to obtain the aforementioned equivalence of categories.
Suppose that $S$ is a sorted complete system, and $a\in m(k,J)(S)$. We define 
$$\beta(a):=gN_{a,k,J}\in m(k,J)(SG(S)),$$
where $N_{a,k,J}:=\ker\big(\pi_{a,k,J}:G(S)\to [a]_{k,J}\big)$ and $\pi_{a,k,J}(g)=a$ (it does not depend on the choice of such element $g$).

We treat $G(S)$ as a subset of $\prod\limits_{a\in m(k,J)(S)} [a]_{k,J}$ containing compatible sequences. Assume now that $(G,\bar{\CF})$ is a sorted profinite group and $g\in G$. We define $\alpha:G\to GS(G)$
by
$$\alpha(g):=(gN_{a,k,J})_{a\in m(k,J)(S(G)}.$$
\begin{remark}\label{rem.alpha_beta}
Take a small definably closed substructure $F$ of $\FC$.
Consider $\alpha:\CG(F)\to GS\CG(F)$, $\beta:S\CG(F)\to SGS\CG(F)$ and $S(\alpha):SGS\CG(F)\to S\CG(F)$.
Then we have that $S(\alpha)\circ\beta=\id_{S\CG(F)}$, which will be useful at the end of the proof of Theorem \ref{thm:Zoe}.
\end{remark}

\begin{definition}\label{def:dual_double_dual}
Let $\Phi:(G_1,\bar{\mathcal{F}}_1)\to (G_1,\bar{\mathcal{F}}_1)$ be a morphism between sorted profinite groups. We call $S\Phi:=S(\Phi)$ and $\Phi$ the {\em dual} of $\Phi$ and $S\Phi$ respectively. Let $\phi:\acl(E)\rightarrow \acl(F)$ be an embedding such that $F$ is a regular extension of $\phi[E]$. We call the dual $S\CG(\phi)$ of $\CG(\phi)$ the {\em double dual} of $\phi$.
\end{definition}

\begin{example}
Let us come back for a moment to Example \ref{ex:hidden.axiom}
to show the actual purpose for introducing the ``hidden axiom".
Assume that $S$ is the $\mathcal{L}_G(\mathcal{J})$-structure (where $|\mathcal{J}|=1$) given in Example \ref{ex:hidden.axiom}.
There is no embedding from $S$ to $SG(S)$.
To see this, note that by $(2),(3)$ from the end of Example \ref{ex:hidden.axiom}, 
we have that $G(S)\cong \BZ/4\BZ$ and so
$$SG(S)/\sim\cong \{\BZ/4\BZ,\;2\BZ/4\BZ,\;4\BZ/4\BZ\},$$
and we have 
$$S/\sim\cong \{X_0,X_1,X_2,X_3\}.$$
So we have that $|SG(S)/\sim|=3\neq 4=|S/\sim|$.
\end{example}

\subsection{Encoding Galois groups}\label{sec:encoding}
Let us recall that we are working with
a complete theory $T$ which has the uniform elimination of imaginaries (in the sense of point b) from \cite[Lemma 8.4.7]{tentzieg}) in the language $\CL$ with sorts $\CJ$. 
Moreover, $\mathfrak{C}\models T$ is a monster model of $T$.
Consider a small definably closed  subset $K$ of some $M\preceq\FC$.

\begin{definition}\label{def:conjugation}
Let $n\ge 1$, $J\in\CJ^{<\omega}$ and $a_1,\ldots, a_n\in S_J(M)$. We say that $a_1,\ldots,a_n$ are {\em conjugated over $K$} if 
\begin{itemize}
	\item $\bigwedge\limits_{i\neq j}a_i\neq a_j$,
	\item $\ulcorner \{a_1,\ldots, a_n\} \urcorner\in K$, and
	\item $\ulcorner A \urcorner\not\in K$ for any proper nonempty subset $A$ of $\{a_1,\ldots, a_n\}$.
\end{itemize}
We write $\conj_{J,K,M}^n(a_1,\ldots,a_n)$ to indicate that 
$a_1,\ldots,a_n$ are conjugated over $K$ (in $M$).
If $J$, $n$, and $K$ or $M$ are obvious, we skip them. 
\end{definition}

Note that $\conj_{J,K}^n(a_1,\ldots,a_n)$ if and only if $a_1\in \acl(K)$, $\CG(K)a_1=\{a_1,\ldots,a_n\}$, and $|\CG(K)a_1|=n$. Hence ``being conjugated" does not depend on the choice of $M$.

\begin{remark}
Because we assume that $T$ has the uniform elimination of imaginaries, 
conditions from Definition \ref{def:conjugation} can be written down as a formula in the language $\mathcal{L}_P$, where $P$ is a predicate corresponding to $K$ (i.e. we consider the $\mathcal{L}_P$-structure $(M,K)$), for example:
$$(M,K)\models\conj(a_1,\ldots,a_n).$$
It is the only place in this subsection, where we require the uniform elimination of imaginaries.
Moreover, it is even enough to assume that $T$ has the uniform elimination of imaginaries only for finite sets.
\end{remark}

\begin{definition}\label{def:n-primitive element}
Let $n\ge 1$ and let $J\in\CJ^{<\omega}$. We say that $a\in S_J(M)$ is {\em an $n$-primitive element of $S_J(M)$ over $K$} if there are $a_2,\ldots,a_n\in S_J(M)$ such that
\begin{enumerate}
	\item $(M,K)\models\conj(a,a_2,\ldots,a_n)$, and
	\item $(M,K)\models\conj(\alpha,\alpha_2,\ldots,\alpha_n)$ for
	$\alpha:=(a,a_2,\ldots,a_n)$ and some $\alpha_2,\ldots,\alpha_n\in S_J^n(M)$.
\end{enumerate} 
We write $\pr_{J,K,M}^n$ ($\subseteq S_J(M)$) for the set of all $n$-primitive elements of $S_J(M)$ over $K$.
\end{definition}

\begin{remark}\label{rmk:example_n-primitivelement}
\begin{enumerate}
	\item The set $\pr_{J,K,M}^n$ is a $\emptyset$-definable set in the language $\CL_P$, that is, there is a formula $\varphi(x)\in \CL_P$ such that 
	for each $M'\preceq\FC$ with $K\subseteq M'$, we have $\varphi(M')= \pr_{J,K,M'}^n$ (here, we consider the $\mathcal{L}_P$-structure $(M',K)$).
	
	\item Let $L$ be a Galois extension of $K$ such that $[L:K]=n$. Any primitive element of $L$ (i.e. an element $a\in L$ such that $\dcl(K,a)=L$) is an $n$-primitive element of $S_J(M)$ over $K$ for an appropriate $J\in\mathcal{J}^{<\omega}$.
	
	\item Let $a\in S_J(M)$ be an $n$-primitive element over $K$. Then, $L=\dcl(a,K)$ is a Galois extension of $K$ with $[L:K]=n$. 
\end{enumerate}
\end{remark}

\begin{proof}
Proofs of points $(1)$ and $(2)$ are clear. We proceed to the proof of point $(3)$. Let $\CG(K)a:=\{a_1(:=a),a_2,\ldots,a_n\}$ and let $\alpha:=(a_1,\ldots,a_n)$.
Let $L':=\dcl(K,\alpha)$, which is a Galois extension of $K$ with $[L':K]=n$. 
It is enough to show
$L'= L$.

Suppose that $\{\sigma_1(a_1),\ldots,\sigma_n(a_1)\}=\CG(K)\cdot a_1$ for some $\sigma_1,\ldots,\sigma_n\in \CG(K)$, say $\sigma_1=\id$. We have that $\sigma_i(\alpha)\neq\sigma_j(\alpha)$ for all $i\neq j$, hence
$\{\sigma_1(\alpha),\ldots,\sigma_n(\alpha)\}=\CG(K)\cdot\alpha$ and $\sigma_1(\alpha)=\alpha$.
Take $\sigma\in \CG(L'/L)$, since $\sigma(\alpha)\in \CG(K)\cdot\alpha$ and $\sigma(a_1)=a_1$ it must be $\sigma(\alpha)=\sigma_1(\alpha)=\alpha$. The last thing implies that $\sigma=\id_{L'}$. By the Galois correspondence, $\CG(L'/L)=\{\id_{L'}\}$ turns into $L=L'$.
\end{proof}

An $\mathcal{L}_P$-formula from the first point of Remark \ref{rmk:example_n-primitivelement} will be denoted by $\pr^n_{J,K}$ (or $\pr^n_J$ when the choice of $K$ is obvious), so $\pr^n_{J,K}(M)=\pr^n_{J,K,M}$.

\begin{lemma}\label{lem:definability_dcl}
Let $n\ge 1$, $J_1,J_2\in\CJ^{<\omega}$,
and $a\in S_{J_1}(M)$ and $b\in S_{J_2}(M)$.
Suppose that $a\in \pr^n_{J_1,K}(M)$. The following are equivalent:
\begin{enumerate}
	\item $b\in \dcl(a,K)$.
	\item $(a,b)\in\pr^n_{J_1^{\smallfrown}J_2,K}(M)$.
\end{enumerate}
\end{lemma}
\begin{proof}
Let $L:=\dcl(a,K)$ and let $\alpha:=(a,b)\in S_{J_1^{\smallfrown}J_2}(M)$.
\
\\
$(1)\Rightarrow (2)$ 
By Remark \ref{rmk:example_n-primitivelement}.(3), $L=\dcl(K,\alpha)$ is a Galois extension of $K$ with $[L:K]=n$, hence by Remark \ref{rmk:example_n-primitivelement}.(2), $\alpha$ is an $n$-primitive element over $K$.
\
\\
$(2)\Rightarrow (1)$ 
By Remark \ref{rmk:example_n-primitivelement}.(3) for elements $a$ and $\alpha$, we have that both
$L$ and $L':=\dcl(K,\alpha)$ are Galois extensions of $K$ such that
$[L:K]=n=[L':K]$. Since $L\subseteq L'$ and $\res:\CG(L'/K)\to \CG(L/K)$ is a bijection, the Galois correspondence implies that $L'=L$.
\end{proof}

\begin{cor}\label{cor:definability_dcl}
Let $n\ge 1$ and let $J_1,J_2\in\CJ^{<\omega}$. 
There exists an $\mathcal{L}_P$-formula $\varphi(x,y)$, where $x\in S_{J_1}$ and $y\in S_{J_2}$, such that
for any $a\in\pr^n_{J_1,K}(M)$
$$\dcl(K,a)\cap S_{J_2}(M)=\varphi(a,M).$$
In other words: $\dcl(K,a)\cap S_{J_2}(M)$ is uniformly definable over $a$ in $\CL_P$.
\end{cor}


\subsection{Interpretability of $SG(K)$ in $\CL_P$}
We are still working with a small definably closed $K$ contained in some $M\preceq\FC$.

\begin{definition}\label{def:interpret_galoisgroup}
For $(k,J)\in \omega\times \CJ^{<\omega}$  
define $\U_{J}^k(M)$ as the 
set of pairs $(a,b)\in (S_{J}(M))^2$ such that
\begin{itemize}
	\item $a,b\in\pr^k_{J,K}(M)$, and
	\item $a$ and $b$ are conjugated over $K$ (i.e. there exists $c_3,\ldots,c_k\in S_J(M)$ such that $(M,K)\models\conj(a,b,c_3,\ldots,c_k)\;$).
\end{itemize}
\end{definition}

Note that for $a,b\in S_J(M)$, $(a,b)\in U_J^k(M)$ if and only if $L:=\dcl(K,a)$ is a Galois extension of $K$ such that 
$[L:k]=k$ and $\sigma(a)=b$ for some $\sigma\in \CG(L/K)$. Note also that the set $U^k_J(M)$ is definable by a formula in the language $\mathcal{L}_P$, which will be denoted by ``$U^k_J$".

Define an equivalence relation $\approx$ on $\U_J^k(M)$ as follows: for $(a_1,b_1)$ and $(a_2,b_2)$ in $\U_J^k(M)$, $(a_1,b_1)\approx (a_2,b_2)$ if and only if
\begin{IEEEeqnarray*}{rCl}
(M,K) &\models & \pr^n_{J^{\smallfrown}J,K}(a_1,a_2) \\
(M,K) &\models & \conj(c,d,e_3,\ldots,e_n)
\end{IEEEeqnarray*}
for some $e_3,\ldots,e_n\in (S_{J}(M))^2$, where $c=(a_1,a_2)$, $d=(b_1,b_2)$.

Suppose that $L=\dcl(K,a)$ is a Galois extension of $K$ with $[L:K]=k$ and $a\in S_J(M)$.
Consider map $U:\CG(L/K)\to U^k_J(M)/\approx$ given by $\sigma\mapsto (a,\sigma(a))/\approx$.
The map $U$ is injective. More generally:


\begin{remark}\label{rmk:unique_approx}
Let $(k,J)\in\omega\times \CJ^{<\omega}$ and $(a,b),(a,b')\in U_J^k(M)$. 
If $(a,b)\approx (a,b')$, then $b=b'$.
\end{remark}

\begin{proof}
Suppose that $(a,b)\approx (a,b')$. Then $(a,a)$ and $(b,b')$ are conjugated over $K$: there is $\sigma\in \CG(K)$ such that $\sigma((a,a))=(b,b')$. Therefore $b=\sigma(a)=b'$.
\end{proof}

\begin{definition}\label{def:basicoperations_on_SG}
Let $(k,J),(k_1,J_1),(k_2,J_2)\in \omega\times \CJ^{<\omega}$.
\begin{enumerate}
	\item Define a binary relation $\leqslant'_{k_1,k_2,J_1,J_2}$ on $\U_{J_1}^{k_1}(M)\times \U_{J_2}^{k_2}(M)$ as follows: for $\alpha_i:=(a_i,b_i)\in \U_{J_i}^{k_i}(M)$, where $i=1,2$, we have $			\le'_{k_1,k_2,J_1,J_2}(\alpha_1,\alpha_2)$ if
	\begin{itemize}
		\item $k_1\geqslant k_2$, and
		\item $a_2\in\dcl(K,a_1)$.
	\end{itemize}

	\item Define a binary relation $C'_{k_1,k_2,J_1,J_2}$ on $\U_{J_1}^{k_1}(M)\times \U_{J_2}^{k_2}		(M)$ as follows: for $\alpha_i:=(a_i,b_i)\in \U_{J_i}^{k_i}(M)$, where $i=1,2$, we have 				$C'_{k_1,k_2,J_1,J_2}(\alpha_1,\alpha_2)$ if
	\begin{itemize}
		\item $k_1\geqslant k_2$, and
		\item $(\alpha,\beta)\in U^{k_1}_{J_1 ^{\smallfrown} J_2}(M)$ for
		$\alpha:=(a_1,a_2)$ and $\beta:=(b_1,b_2)$.
	\end{itemize}
	[i.e. for $\sigma_i\in \CG(a_i/K)$ such that $\sigma_i(a_i)=b_i$, where $i=1,2$, we have 
	$\sigma_1(a_2)=\sigma_2(a_2)$ and $a_2\in \dcl(a_1,K)$]

	\item Define a ternary relation $P'_{k,J}$ on $U_J^k(M)$ as follows: for $\alpha_i:=(a_i,b_i)\in U_J^k(M)$, where $i=1,2,3$, we have $P'_{k,J}(\alpha_1,\alpha_2,\alpha_3)$ if
	\begin{itemize}
		\item $a_2,a_3\in\dcl(K,a_1)$, and
		\item there is $c\in S_{J}(M)$ (which is unique by Remark \ref{rmk:unique_approx}) such that $(a_1,b_1)\approx (b_2,c)$ and $(a_3,b_3)\approx (a_2,c)$.
	\end{itemize}
	[i.e. for $\sigma_i\in \CG(a_i/K)$ corresponding to $\alpha_i$, where $i=1,2,3$, 
	we have $\sigma_1\sigma_2(a_2)=c=\sigma_3(a_2)$ and $a_2,a_3\in\dcl(K,a_1)$, hence $\sigma_1\sigma_2=\sigma_3$]
\end{enumerate}
If there is no confusion, we write $C'$, $\le'$, and $P'$ for $C'_{k_1,k_2,J_1,J_2}$, $\le'_{k_1,k_2,J_1,J_2}$, and $P'_{k,J}$ respectively.
\end{definition}

\begin{remark}\label{rmk:basicoperations_on_galoisgroup}
Let $(k,J),(k_1,J_1),(k_2,J_2)\in \omega\times \CJ^{<\omega}$.
\begin{enumerate}
	\item Assume that $\alpha_i:=(a_i,b_i),\, \alpha_i':=(a_i',b_i') \in U_{J_i}^{k_i}(M)$, where $i=1,2$. If $\alpha_i\approx\alpha_i'$, then
	$$\leqslant'(\alpha_1,\alpha_2)\quad\Leftrightarrow \quad\leqslant'(\alpha_1',\alpha_2').$$

	\item Assume that $\alpha_i:=(a_i,b_i),\, \alpha_i':=(a_i',b_i') \in U_{J_i}^{k_i}(M)$, where $i=1,2$. If $\alpha_i\approx\alpha_i'$, then
	$$C'(\alpha_1,\alpha_2)\quad\Leftrightarrow\quad C'(\alpha_1',\alpha_2').$$

	\item Assume that $\alpha_i:=(a_i,b_i),\, \alpha_i':=(a_i',b_i') \in U_J^k(M)$, where $i=1,2,3$.
	If $\alpha_i\approx\alpha_i'$, then
	$$P'(\alpha_1,\alpha_2,\alpha_3)\quad\Leftrightarrow\quad P'(\alpha_1',\alpha_2',\alpha_3').$$
\end{enumerate}
\end{remark}

\begin{proof}
It is enough to use the equivalent formulations provided in square brackets in Definition \ref{def:basicoperations_on_SG} and we leave the proof to the reader.
\end{proof}

Therefore $\leqslant'$, $C'$ and $P'$ induce well-defined relations (also denoted by $\leqslant'$, $C'$ and $P'$) on the classes of the relation $\approx$.

Before reaching the main theorem of this subsection (Theorem \ref{thm:interpretability_SG(K)}), we provide a result interesting on its own, namely Proposition \ref{prop:interpretability_galoisaction}. 
We use a standard definition of the notion of \emph{$A$-interpretability} coming from Definition 1.1 in \cite[Chapter 3]{anandgeometric}.
Although, let us start with auxiliary lemmas.

We fix a finite Galois extension $L$ of $K$. 

\begin{lemma}\label{lemma:GLK.interpretability01}
Suppose that $L$ is given by some $a\in\pr^n_J(M)$, i.e. $L=\dcl(K,a)$.
Then the group $\CG(L/K)$ is $\{a\}$-interpretable in $(M,K)$.
\end{lemma}

\begin{proof}
Consider a subset of $U^n_J(M)/\approx$ given by:
$$W_a:=\{(a,b)/\approx\;|\; b\in\pr^n_J(M)\text{ and }(a,b)\in U^n_J(M)\}.$$
One could write $W_a$ as the set $\{(a,\sigma(a))/\approx\;|\;\sigma\in \CG(L/K)\}$.
Note that $W_a$ is $\{a\}$-definable in $(M,K)^{\eq}$. Consider group structure on $W_a$ induced by the relation $P'$ (which is well defined by Remark \ref{rmk:basicoperations_on_galoisgroup}.(3)):
$$\alpha_1/\approx\;\cdot\;\alpha_2/\approx=\alpha_3/\approx\qquad\iff\qquad (M,K)\models P'(\alpha_1,\alpha_2,\alpha_3),$$
where $\alpha_1,\alpha_2,\alpha_3\in U^n_J(M)$. To finish the proof we need to find a group isomorphism between the group $\CG(L/K)$ and the set $W_a$ equipped with the above ``multiplication".

Consider $\Phi_a:\CG(L/K)\to W_a$, $\sigma\mapsto (a,\sigma(a))/\approx$. Since $L=\dcl(K,a)$, Remark \ref{rmk:unique_approx} implies that $\Phi_a$ is injective. By the note under Definition \ref{def:interpret_galoisgroup}, it is clear that $\Phi_a$ is onto. To see that it preserves the ``multiplication" it is enough to combine Remark \ref{rmk:basicoperations_on_galoisgroup} with the explanation provided in the square brackets in Definition \ref{def:basicoperations_on_SG}.
\end{proof}

\begin{lemma}\label{lemma:GLK.interpretability02}
Suppose that $L$ is given by some $a\in\pr^n_J(M)$, i.e. $L=\dcl(K,a)$ and
that $a'\in\pr^n_J(M)$ is such that $\dcl(K,a)=\dcl(K,a')$ (i.e. $(M,K)\models\pr^n_{J^{\smallfrown}J}(a,a')$). Then (using the notation from the previous proof)
$W_a=W_{a'}$ and $\Phi_a=\Phi_{a'}$.
\end{lemma}

\begin{proof}
Follows from $(a,\sigma(a))\approx(a',\sigma(a'))$ for each $\sigma\in \CG(L/K)$.
\end{proof}

\begin{cor}\label{cor:GLK.interpretability03}
The group $\CG(L/K)$ is $K$-interpretable in $(M,K)$.
\end{cor}

\begin{proof}
By \cite[Theorem 4.3]{DHL1} (The Primitive Element Theorem), there exists $J\in\mathcal{J}^{<\omega}$, $n\in\omega$ and $a\in\pr^n_J(M)$ such that $L=\dcl(K,a)$. Because $a\in\acl(K)$, there exists an $\mathcal{L}$-formula $\psi(y)$ with parameters from $K$ which isolates $\tp(a/K)$ (in the sense of $\FC$). Consider the following $\mathcal{L}_P$-formula $W$: $(\exists\,y)\,(\psi(y)\,\wedge\,W_y)$, where $W_y$ corresponds to the definable set introduced in the proof of Lemma \ref{lemma:GLK.interpretability01} for the case of $y=a$.
Note that, by Lemma \ref{lemma:GLK.interpretability02}, realizations of formula $W$ form exactly set $W_a$. Moreover, the definition of the relation $P'$ is parameter-free, hence our interpretation of group $\CG(L/K)$ involves only parameters which occur in the formula $\psi$.
\end{proof}

\begin{prop}\label{prop:interpretability_galoisaction}\cite[Proposition 5.5]{chatzidakis2002}
Assume that $L$ is a finite Galois extension of $K$ and and $J\in\CJ^{<\omega}$.
Then the group action $\cdot: \CG(L/K)\times S_J(L)\rightarrow S_J(L)$ is $K$-interpretable in $(M,K)$.
\end{prop}

\begin{proof}
By Corollary \ref{cor:GLK.interpretability03}, group $\CG(L/K)$ is $K$-interpretable in $(M,K)$.
By Lemma \ref{lem:definability_dcl} and a similar argument as in the proof of Corollary \ref{cor:GLK.interpretability03}, we see that the set $S_J(L)$ is also $K$-interpretable in $(M,K)$ (even $K$-definable in $(M,K)$). 

Suppose that $[L:K]=n$.
Consider $a\in \pr^n_{J_1}(M)$, such that $L=\dcl(K,a)$, and set $W_a$ (as in the proof of Lemma \ref{lemma:GLK.interpretability01}). If $\alpha=(a,\sigma(a))/\approx\in W_a$, where $\sigma\in \CG(L/K)$, and $c\in S_J(L)$, then set $\alpha\;\bullet\;c:=d$, where $d\in S_J(L)$ is the unique element which satisfies
$$(M,K)\models\conj((a,c),(\sigma(a),d),e_3,\ldots,e_{n})$$ 
for some $e_3,\ldots,e_{n}\in S_{J_1}(L)\times S_J(L)$.

We need to show that the group action is $K$-interpretable, in other words the bijections between $\CG(L/K)$, $S_J(L)$ and their interpretations in $(M,K)^{\eq}$ commute with group actions $\cdot$ and $\bullet$ (we do not show that $\bullet$ defines a group action, since it will follow from the fact that bijections commute with $\cdot$ and $\bullet$).

Suppose that $\sigma\,\cdot\,c=\sigma(c)=d$ for some $c,d\in S_J(L)$ and $\sigma\in \CG(L/K)$.
It means that $\sigma$ moves $(a,c)$ into $(\sigma(a),d)$, and so $\Phi_a(\sigma)\,\bullet\,c=d$ (here $\Phi_a$ is the bijection coming from the proof of Lemma \ref{lemma:GLK.interpretability01}).
Conversely, if $\Phi_a(\sigma)\,\bullet\,c=d$, then there exists $\tau\in \CG(L/K)$ such that
$\tau(a)=\sigma(a)$ and $\tau(c)=d$. Since $L=\dcl(K,a)$ and $\tau(a)=\sigma(a)$, we have that $\tau=\sigma$ and so $d=\sigma(c)$.
\end{proof}

\begin{theorem}\label{thm:interpretability_SG(K)}
The sorted complete system $S\CG(K)$ is interpretable (without parameters) in $(M,K)$.
\end{theorem}

\begin{proof}
Similarly as in Example \ref{ex:hidden.axiom}, we consider ``diagonal" map
$\epsilon_{X,k}:\ X\rightarrow X^k$, $x\mapsto (x,\ldots,x)$
and ``projection" map $\delta_{X}:\ X^k\rightarrow X$, $(x_1,\ldots,x_k)\mapsto x_1$, where
$X$ is a set and $k$ is a positive integer (usually we skip ``$X$" in ``$\epsilon_{k,X}$" and ``$\delta_X$").

First, for each sort in $S\CG(K)$ we need to provide a definable set in $(M,K)^{\eq}$.
Let $k\in\omega$ and $J\in\mathcal{J}^{<\omega}$.
If we would define the sort $m(k,J)(S\CG(K))$ as the set of cosets of open normal subgroups of index equal to $k$, then the corresponding sort of our interpretation will be the set $U^k_{J^{\smallfrown k}}(M)/\approx$. 
Since we defined $m(k,J)(S\CG(K))$ as the set of cosets of open normal subgroups of index at most $k$, and different sorts intersect trivially, it is not enough to consider $\bigcup_{i\le k} \U_{J^{\smallfrown i}}^i(M)/\approx$ but
the set 
$$W_{k,J}:=\bigcup_{i\le k} \epsilon_k[\U_{J^{\smallfrown i}}^i(M)/\approx].$$
Let us explain why sets of the form $W_{k,J}$ have something to do with sorts $m(k,J)(S\CG(K))$ and how we can define the desired bijection.

Suppose that $gH\in m(k,J)(S\CG(K))$. It means that $H\in\CN(\CG(K))$, $g\in \CG(K)$, $[\CG(K):H]=i\leqslant k$, and $|\aut_J(L/K)|=|\CG(L/K)|$ and $[L:K]=i\leqslant k$ for the Galois extension $L:=\acl(K)^{H}$.
There is a unique $\sigma\in \CG(L/K)$ corresponding to $gH$. Moreover, because $|\aut_J(L/K)|=|\CG(L/K)|$, we can (after repeating a standard argument from the proof of \cite[Fact A.10]{DHL1}) find an element $a\in S_{J^{\smallfrown i}}(M)$ such that $L=\dcl(K,a)$. 
By Remark \ref{rmk:example_n-primitivelement}, $a\in\pr^i_{J^{\smallfrown i}}(M)$. We define a map 
$F_{k,J}:gH\mapsto \epsilon_k[(a,\sigma(a))/\approx]\in W_{k,J}$. It is well defined, since $(a,\sigma(a))\approx(a',\sigma(a'))$ for any $a'\in\pr^i_{J^{\smallfrown i}}(M)$ such that $\dcl(K,a)=\dcl(K,a')$.

To show that $F_{k,J}$ is injective, suppose that $gH,g'H'\in m(k,J)(S\CG(K))$ and $F_{k,J}(gH)=F_{k,J}(g'H')$. Let $[\CG(K):H]=i$ and $[\CG(K):H']=i'$, $L:=\acl(K)^H$ and $L'=:\acl(K)^{H'}$, $a\in\pr^i_{J^{\smallfrown i}}(M)$, $a'\in\pr^{i'}_{J^{\smallfrown i'}}(M)$, and $L=\dcl(K,a)$ and $L'=\dcl(K,a')$. If $i\neq i'$, then for formal reasons $F_{k,J}(gH)\neq F_{k,J}(g'H')$, hence assume that $i=i'$. Since $F_{k,J}(gH)=F_{k,J}(g'H')$, we have that $(a,\sigma(a))\approx(a',\sigma'(a'))$ which gives us $\sigma=\sigma'$.
By a similar, straightforward, argument one can show that $F_{k,J}$ is onto.

After defining sorts of the universe of our interpretation, we need to define relations corresponding to symbols $\leqslant$, $C$, and $P$ from language $\mathcal{L}_G(\mathcal{J})$.
\begin{itemize}
\item For $\alpha_1 \in W_{k_1,J_1}$, $\alpha_2 \in W_{k_2,J_2}$ we set
$$\leqslant^W_{k_1,k_2,J_1,J_2}(\alpha_1,\alpha_2)\quad\iff\quad \leqslant^{\bigcup}_{k_1,k_2,J_1,J_2}(\delta(\alpha_1),\delta(\alpha_2)),$$ 
where
$$\leqslant^{\bigcup}_{k_1,k_2,J_1,J_2}
:=\bigcup\limits_{i\leqslant k_1,j\leqslant k_2}\leqslant_{i,j,J_1^{\smallfrown i},J_2^{\smallfrown j}}'$$

\item For $\alpha_1 \in W_{k_1,J_1}$, $\alpha_2 \in W_{k_2,J_2}$ we set
$$C^W_{k_1,k_2,J_1,J_2}(\alpha_1,\alpha_2)\quad\iff\quad C^{\bigcup}_{k_1,k_2, J_1,J_2}(\delta(\alpha_1),\delta(\alpha_2)),$$ 
where 
$$C^{\bigcup}_{k_1,k_2, J_1,J_2}:=\bigcup\limits_{i\leqslant k_1,j\leqslant k_2}C_{i,j,J_1^{\smallfrown i},J_2^{\smallfrown j}}'$$

\item For $\alpha_1,\alpha_2,\alpha_3\in W_{k,J}$ we set
$$P^W_{k,J}(\alpha_1,\alpha_2,\alpha_3)\quad\iff\quad 
P^{\bigcup}_{k,J}(\delta(\alpha_1),\delta(\alpha_2),\delta(\alpha_3))$$
where
$$P^{\bigcup}_{k,J}:=\bigcup\limits_{i\leqslant k}P'_{i,J^{\smallfrown i}}$$
\end{itemize}
Now, we need to show that the family of bijections $F_{k,J}$ translates $\leqslant$, $C$ and $P$ into $\leqslant^W$, $C^W$ and $P^W$ respectively, e.g.
$$C(gH, g'H')\qquad\iff\qquad C^W(F_{k,J}(gH),F_{k',J'}(g'H'))$$
for any $gH\in m(k,J)(S\CG(K))$ and $g'H'\in m(k',J')(S\CG(K))$.
Comments in the square brackets in Definition \ref{def:basicoperations_on_SG} are here a guideline and we leave this part of the proof to the reader.
\end{proof}

\begin{cor}\label{cor:SG.saturated}
If $(M,K)$ is $\kappa$-saturated, then $S\CG(K)$ is $\kappa$-saturated.
\end{cor}

The above corollary follows immediately by Theorem \ref{thm:interpretability_SG(K)}. It is not difficult to show that ``if $(M,E)\preceq (N,F)$, then $S\CG(E)\preceq S\CG(F)$", but we want to write it more precisely and introduce \emph{choice functions}, because such an approach produces a good way of translating formulas between structure $K$ and $S\CG(K)$ (and we will use this translation later).

\begin{remark}\label{rem:translation.B}
Take $k<\omega$ and $J\in\mathcal{J}^{<\omega}$ and consider the bijection between $S\CG(K)$ and $W_{k,J}\subseteq (M,K)^{\eq}$ given in the proof of Theorem \ref{thm:interpretability_SG(K)}, $F_{k,J}:S\CG(K)\to W_{k,J}(M)$. 
Suppose that $B=\dcl(B)\subseteq K$ is regular (as previously, $K$ is a small substructure of $M$, where $M\preceq\FC$).
Assume that $gH\in m(k,J)(S\CG(B))$, $L:=\acl(B)^H=\dcl(B,a)$ for some
$a\in\pr^i_{J^{\smallfrown i},B}(M)\subseteq\acl(B)$ and $gH$ corresponds to $\sigma\in \CG(L/B)$.
Suppose that for $gH\in m(k,J)(S\CG(B))$ we have chosen such a primitive element $a$ and an automorphism $\sigma$.
Consider the following \emph{choice function}
$$c_B:S\CG(B)\to\acl(B)^{\eq}\subseteq(M,K)^{\eq}$$
where $c_B(gH):=(a,\sigma(a))$ for $gH$ as above (``$^{\eq}$" in ``$\acl(B)^{eq}$" indicates only that we are dealing with tuples of elements from $\acl(B)$). 
Similarly we define a choice function for any other regular substructure in $K$, in particular $c_K$.

Assume that
the restriction map $\pi:\CG(K)\to \CG(B)$ is onto 
(e.g. if $B\subseteq K$ is regular)
and the corresponding dual map $S(\pi)$ is an embedding.
Usually we identify $S\CG(B)$ with its image $S(\pi)(S\CG(B))$ in $S\CG(K)$ and
(by Fact \ref{fact:regular_to_sorted}) 
we have
$$F_{k,J}(S(\pi)(gH))=\epsilon_k[c_B(gH)/\approx],$$
hence $F_{k,J}(S(\pi)(gH))\in\acl(B)^{\eq}$ (here ``$^{\eq}$" really stands for imaginary elements in $(M,K)$).


Suppose that $\Theta(X,Y)$ is an $\mathcal{L}_G(\mathcal{J})$-formula,
$gH\in m(k,J)(S\CG(B))\subseteq S\CG(K)$ and $g'H'\in m(k',J')(S\CG(K))$ for appropriate $k,K',J,J'$ corresponding to variables $X$ and $Y$.
We have that $S\CG(K)\models\Theta(gH,g'H')$ if and only if $(M,K)^{\eq}\models\Theta'(F^K_{k,J}(gH),F^K_{k',J'}(g'H'))$, where $\Theta'$ corresponds to the interpretation of $\Theta$ in $(M,K)$.
On the other hand, the $(\mathcal{L}_P)^{\eq}$-formula $\Theta'$ is equivalent to an $\mathcal{L}_P$-formula $\theta$ (e.g. Lemma 1.4.(iii) in \cite[Chapter 1]{anandgeometric}) and we have
$$S\CG(K)\models\Theta(gH,g'H')\quad\iff\quad (M,K)\models\theta(c_B(gH),c_K(g'H')).$$
\end{remark}

\begin{cor}\label{cor:SG.elementary1}
If $(M,E)\preceq (N,F)$ for some $M\preceq N\preceq\FC$, then $S\CG(E)\preceq S\CG(F)$ (after embedding of $S\CG(E)$ into $S\CG(F)$).
\end{cor}


\section{Elementary vs co-elementary}\label{sec:coelementary}
\begin{lemma}\label{lemma:CZ.preceq}
Assume that $T$ has nfcp.
Let $E\preceq F$ be some small substructures of $\FC$ 
and let $M\preceq N\preceq\FC$ be such that $E\subseteq M$, $F\subseteq N$, $M$ is $|E|^+$-saturated and $N$ is $|F|^+$-saturated, and $M\ind_E F$.
Then $(M,E)\preceq (N,F)$.
\end{lemma}

\begin{proof}
By \cite[Proposition 2.1]{CasanovasZiegler} and the proof of \cite[Corollary 2.2]{CasanovasZiegler}, each $\mathcal{L}_P$-formula $\Phi(\bar{x})$ is equivalent in $(M,E)$ and in $(N,F)$ to an $\mathcal{L}_P$-formula of the form
$$Q\bar{\alpha}\in P\;\varphi(\bar{x},\bar{\alpha}),$$
where $\varphi(\bar{x},\bar{\alpha})$ is an $\mathcal{L}$-formula and $Q$ is a tuple of quantifiers.
Since $T$ has quantifier elimination, we may assume that $\varphi(\bar{x},\bar{\alpha})$ is quantifier free.

Suppose that $(M,E)\models \Phi(m)$ for some finite tuple $m$ from $M$. By the above lines, it means that $(M,E)\models Q\bar{\alpha}\in P \;\varphi(m,\bar{\alpha})$. 

We want to code $\varphi(m,\bar{\alpha})$ by some $\mathcal{L}$-formula without ``$m$" and to do this we will use a definition of the $\varphi$-type $\tp_{\varphi}(m/F)$. However, we need to show that our definition works also for the $\varphi$-type $\tp_{\varphi}(m/E)$.

Because $m\ind_E F$ and $E\subseteq F$ is regular, \cite[Corollary 3.38]{Hoff3} implies that $\tp(m/F)$ is the unique non-forking extension of $\tp(m/E)$. Therefore the set of all non-forking global extensions of $\tp(m/E)$ and $\tp(m/F)$ coincide.

(The following paragraph is based on an argument pointed to us by Martin Ziegler.)
By \cite[Theorem 8.5.6.(1)]{tentzieg}, all these global extensions conjugate over $E$ (and over $F$). There are only finitely many different $\varphi$-parts of these global extensions, say $p_1(\bar{x}),\ldots,\p_n(\bar{x})$, and let $\theta_1(\bar{y}),\ldots,\theta_n(\bar{y})$ be their definitions
(over some parameters from $\FC$).
If $\bar{b}\in F$ then $\varphi(\bar{x},\bar{b})\in tp(m/F)$ if and only if $\varphi(\bar{x},\bar{b})\in p_1(\bar{x}) \cap\ldots\cap p_n(\bar{x})$, which holds if and only if
$\models \theta_1(\bar{b})\;\wedge\ldots\wedge\;\theta_n(\bar{b})$.
Set $\psi(\bar{y}):=\theta_1(\bar{y})\;\wedge\ldots\wedge\;\theta_n(\bar{y})$ and note that
$\psi(\bar{y})$ is $E$-invariant, so we may assume that $\psi(\bar{y})$ is quantifier free and is over $E$, and note that
$\psi(\bar{y})$ defines $\tp_{\varphi}(m/F)$ and $\tp_{\varphi}(m/E)$.

We have that
$$(M,E)\models \forall\bar{\alpha}\in P\;(\varphi(m,\bar{\alpha})\,\leftrightarrow\,\psi(\bar{\alpha})\,),$$
and
$$(N,F)\models \forall\bar{\alpha}\in P\;(\varphi(m,\bar{\alpha})\,\leftrightarrow\,\psi(\bar{\alpha})\,).$$

Since $(M,E)\models Q\bar{\alpha}\in P \;\varphi(m,\bar{\alpha})$, we have $(M,E)\models Q\bar{\alpha}\in P\; \psi(\bar{\alpha})$ hence $E\models Q\bar{\alpha}\;\psi(\bar{\alpha})$.
Because $E\preceq F$, we obtain that $F\models Q\bar{\alpha}\;\psi(\bar{\alpha})$.
The last item gives us $(N,F)\models Q\bar{\alpha}\in P\;\psi(\bar{\alpha})$, hence we have
$(N,F)\models Q\bar{\alpha}\in P\;\varphi(m,\bar{\alpha})$ which ends the proof.
\end{proof}

\begin{prop}\label{prop:SG.elementary2}
Assume that $T$ has nfcp (hence $T$ is stable). If $E\preceq F$ are small substructures of $\FC$, then $S\CG(E)\preceq S\CG(F)$.
\end{prop}

\begin{proof}
By combining Corollary \ref{cor:SG.elementary1} with Lemma \ref{lemma:CZ.preceq}.
\end{proof}

It turns out that for PAC substructures the converse is also true, i.e.
``if $S\CG(E)\preceq S\CG(F)$ then $E\preceq F$".
Let us proceed to this fact.

\begin{theorem}\label{thm:ultraproduct_sorted_completesystem}
Let $\CU$ be an ultrafilter on an infinite index set $I$. Let $K_i\subseteq \FC$ be a definably closed substructure for each $i$.
Then, we have that
$$\prod_i S\CG(K_i)/\CU\cong S\CG(\prod_i(K_i)/\CU).$$
\end{theorem}

\begin{proof}
It comes from the uniform interpretation of the sorted complete system in Theorem \ref{thm:interpretability_SG(K)}.
\end{proof}

The following result generalizes \cite[Theorem 5.3]{chatzidakis2002}, and \cite[Theorem 2.6, Theorem 2.7]{chatzidakis2017}.

\begin{theorem}\label{thm:PAC_type_by_completesystem}
Suppose that $T$ has nfcp and PAC is a first order property (in the sense of \cite[Definition 2.6]{DHL1}). Let $K_1$ and $K_2$ are PAC and let $E\subseteq K_1\cap K_2$ be definably closed. Let $a$ and $b$ be tuples (of possibly infinite length) of $K_1$ and $K_2$ respectively. The following are equivalent:
\begin{enumerate}
	\item $\tp_{K_1}(a/E)=\tp_{K_2}(b/E)$ (so $K_1\equiv K_2$).
	\item There are $K_1 \preceq K_1^*$ and $K_2\preceq K_2^*$, and there are $A\subseteq K_1^*$ and $B\subseteq K_2^*$ being regular extensions of $E$ which contain $a$ and $b$ respectively, and which are definably closed in $\FC$, and there is an $\CL(E)$-embedding $\phi:\acl(A)\rightarrow \acl(B)$ such that
	\begin{enumerate}
		\item $K_1^*$ and $K_2^*$ are regular extensions of $A$ and $B$ respectively,
		\item $\phi(a)=b$ and $\phi[A]=B$, and
		\item $S\Phi:S\CG(A)\rightarrow S\CG(B)$ is a partial elementary map from $S\CG(K_1^*)$ to $S\CG(K_2^*)$, where $S\Phi:=S\CG(\phi)$.
	\end{enumerate}	  
\end{enumerate}
\end{theorem}
\begin{proof}
Fix $M\preceq \FC$ containing $K_1$ and $K_2$.\\

$(1)\Rightarrow (2)$ Since $\tp_{K_1}(a/E)=\tp_{K_2}(b/E)$, we have that $(K_1,a)\equiv_E (K_2,b)$. By the Keisler-Shelah theorem, there is an ultrafilter $\CU$ and an $E$-isomorphism $\psi:K_1^{\CU}\rightarrow K_2^{\CU}$ with $\psi(a^{\CU})=b^{\CU}$. Set $K_i^*:=K_i^{\CU}$ for $i=1,2$, and set $A:= K_1^{\CU}\cap \acl(aE)$ and $B:= K_2^{\CU}\cap \acl(bE)$. Note that 
\begin{itemize}
	\item $a^{\CU}=a$ and $b^{\CU}=b$, and
	\item $\psi[A]=B$
\end{itemize}
We extend $\psi$ to an $E$-isomorphism from $\acl(K_1^{\CU})$ to $\acl(K_2^{\CU})$, still denoted by $\psi$. Let $\phi:=\psi\restriction_{\acl(A)}:\acl(A)\rightarrow \acl(B)$. Let $S\Phi:S\CG(A)\rightarrow S\CG(B)$ be the dual of $\phi$.
\
\\
\textbf{Claim}
The dual map $S\Phi$ gives a partial elementary map from $S\CG(K_1)$ to $S\CG(K_2)$.
\
\\ Proof of the claim: 
For $i=1,2$, we have that $S\CG(K_i)\preceq S\CG(K_i^{\CU})$ because $(M,K_1,K_2)\preceq (M^{\CU},K_1^{\CU},K_2^{\CU})$. Also we have that $S\CG(A)\subseteq S\CG(K_1)$ and $S\CG(B)\subseteq S\CG(K_2)$ because $K_1$ and $K_2$ are regular extensions of $A$ and $B$ respectively. Let $S\Psi:S\CG(K_1^{\CU})\rightarrow S\CG(K_2^{\CU})$ be the double dual of the isomorphism $\psi$. The restriction of $S\Psi$ to $S\CG(A)$ is exactly same as $S\Phi$ so $(S\CG(K_1^{\CU}),S\CG(A))\equiv (S\CG(K_2^{\CU}),S\CG(B))$. Thus we have that $(S\CG(K_1),S\CG(A))\equiv (S\CG(K_2),S\CG(B))$, and we are done.
Here ends the proof of Claim.

$(2)\Rightarrow (1)$ Since $\tp_{K_1^*}(a/E)=\tp_{K_1}(a/E)$ and $\tp_{K_2^*}(b/E)=\tp_{K_2}(b/E)$, we may assume that $K_i^*=K_i$ by replacing $K_i$ by $K_i^*$ for $i=1,2$. Since $S\Phi:S\CG(A)\rightarrow S\CG(B)$ is a partial elementary map, by \cite[Theorem 10.3, Theorem 10.5]{ultraKeisler}, there is a non-principal ultrafilter $\CU$ and an isomorphism $S F:S\CG(K_1)^{\CU}\rightarrow S\CG(K_2)^{\CU}$ such that
\begin{itemize}
	\item $S\CG(K_1)^{\CU}\equiv_{S\CG(A)} S\CG(K_1)$ and $S\CG(K_2)^{\CU}\equiv_{S\CG(B)}S\CG(K_2)$, and
	\item $S F[S\CG(A)]=S\CG(B)$,
	\item $K_1^{\CU}$ and $K_2^{\CU}$ are $\kappa$-saturated for some infinite cardinal $\kappa\geqslant (\max\{|A|,|B|,|T|\})^+$.
\end{itemize}
Using Theorem \ref{thm:ultraproduct_sorted_completesystem} and dualising $S F$, we have a group homeomorphism $F:\CG(K_1^{\CU})\rightarrow \CG(K_2^{\CU})$ such that for $\sigma \in \CG(K_1^{\CU})$, $F(\sigma)\restriction_{\acl(B)}=\Phi(\sigma\restriction_{\acl(A)})$. From the proof of \cite[Proposition 3.6]{DHL1}, we have an isomorphism $\phi:K_1'\rightarrow K_2'$ extending $\phi\restriction_A:A\rightarrow B$ for some $K_i'\preceq K_i^{\CU}$ with $A\subseteq K_1'$ and $B\subseteq K_2'$. We conclude that
$$\tp_{K_1^{\CU}}(A/E)=\tp_{K_1'}(A/E)=\tp_{K_2'}(B/E)=\tp_{K_2^{\CU}}(B/E).$$ Since $K_i\preceq K_i^{\CU}$, we have that $\tp_{K_i^{\CU}}(A/E)=\tp_{K_i}(A/E)$ for $i=1,2$. Therefore we have that $\tp_{K_1}(A/E)=\tp_{K_2}(B/E)$.
\end{proof}

\begin{cor}\label{cor:sortedcompletesystem_elementaryequiv}
Suppose that $T$ has nfcp and
PAC is a first order property.
Assume that $E\subseteq F$ are PAC.
We have that
$S\CG(E)\preceq S\CG(F)$ if and only if $E\preceq F$.
\end{cor}

\section{Generalization of Chatzidakis' Theorem}\label{sec:zoe.theorem}
From now on we assume that $T$ is \textbf{stable}.
We write $A,B,C,\ldots$ for small subsets of $\mathfrak{C}$ and $a,b,c,\ldots$ for tuples of $\mathfrak{C}$ of bounded length. We write $a\in A$ if $a$ is a tuple consisting of elements of $A$. For $A,B\subseteq \mathfrak{C}$, we write $AB$ for $A\cup B$. For a subset $A$, we denote $\ov A:=\acl(A)$.
We recall a notion of the boundary property.
The original definition (\cite[Definition 3.1]{GKK}) has a typo and therefore we refer to 
\cite[Remark 2.3]{Kim16}.

\begin{notation}
Let $\{a_0,\ldots, a_{n-1}\}$ be an $A$-independent set. For $u\subseteq \{0,\ldots,n-1\}$, we write $\a_u:=\ov{\{a_i:i\in u\}A}$. And we write $\{0,\ldots, \hat i,\ldots, n-1\}:=\{0,\ldots,i-1,i+1,\ldots,n-1\}$.
\end{notation}

\begin{definition}\label{def:boundary_property}\cite[Remark 2.3]{Kim16}
Let $n\ge 2$.
\begin{enumerate}
	\item Let $A$ be a small subset of $\FC$. We say that {\em the property $B(n)$ holds over $A$} if for every $A$-independent set $\{a_0,\ldots, a_{n-1}\}$, $$\dcl(\a_{\{\hat 0,\ldots,n-1\}},\ldots,\a_{\{0,\ldots,\widehat{n-2},n-1\}})\cap \a_{\{0,\ldots,n-2\}}=\dcl(\a_{\{\hat 0,\ldots,n-2\}},\ldots,\a_{\{0,\ldots,\widehat{n-2}\}}).$$
	\item We say that {\em $B(n)$ holds for $T$} if $B(n)$ holds over every small subset of $\mathfrak{C}$.
\end{enumerate}
\end{definition}

\begin{fact}\label{fact:equivalent_conditions_B(n)}\cite[Lemma 3.3]{GKK}
For any set $A$ and any $n\ge 3$, the following are equivalent:
\begin{enumerate}
	\item $T$ has $B(n)$ over $A$.
	\item For any $A$-independent set $\{a_0,\ldots,a_{n-1}\}$ and any $c\in \a_{\{0,\ldots,n-2\}}$, $$\tp(c/\a_{\{\hat 0,\ldots, n-2\}}\cdots \a_{ \{0,\ldots,\widehat {n-2} \} })\models \tp(c/\a_{\{\hat 0,\ldots,n-1\}} \cdots \a_{\{0,\ldots,\widehat{n-2},n-1\}} ).$$
	\item For any $A$-independent set $\{a_0,\ldots,a_{n-1}\}$ and any map $\sigma$ such that $$\sigma\in \aut( \a_{\{0,\ldots, n-2\}}/\a_{\{\hat 0,\ldots, n-2\}}\cdots \a_{\{0,\ldots,\widehat{n-2}\}} ),$$ $\sigma$ can be extended to $\ov \sigma \in \aut(\FC)$ which fixes $$\a_{\{\hat 0,\ldots, n-1\}}\cdots \a_{\{0,\ldots,\widehat{n-2},n-1\}}$$ pointwise.
\end{enumerate}
\end{fact}

\begin{lemma}\label{lemma:galois_times}\cite[Lemma 2.14]{chatzidakis2017}
Assume that $B(3)$ holds for $T$.
Let $E\subseteq A,B,C$ be definably closed. Suppose that $\{A,B,C\}$ is an $E$-independent set and
consider the map 
$$\rho:\CG(\dcl(\ov{AB},\ov{AC},\ov{BC})/ABC)\rightarrow\CG(AB)\times \CG(AC)\times \CG(BC)$$
defined by $\sigma\mapsto(\sigma \restriction_{\ov{AB}}, \sigma \restriction_{\ov{AC}}, \sigma \restriction_{\ov{BC}})$. Then we have that 
$$\im(\rho)=\{(\sigma_1,\sigma_2,\sigma_3)\;|\; \sigma_1\restriction_{\ov A}=\sigma_2\restriction_{\ov A},\sigma_1\restriction_{\ov B}=\sigma_3\restriction_{\ov B}, \sigma_2\restriction_{\ov C}=\sigma_3\restriction_{\ov C}\}.$$
\end{lemma}

\begin{proof}
Consider the following property $(\ast)$:
$$\sigma_1\restriction_{\ov A}=\sigma_2\restriction_{\ov A},\quad
\sigma_1\restriction_{\ov B}=\sigma_3\restriction_{\ov B}, \quad
\sigma_2\restriction_{\ov C}=\sigma_3\restriction_{\ov C}$$
given for any triple $(\sigma_1,\sigma_2,\sigma_3)$ of mappings with proper domains.
Consider also the following diagram
$$\xymatrix{
\CG(\ov A\,\ov B\,\ov C/ABC) \ar[r]^-{\rho_1} & \big[\CG(\ov A\,\ov B/AB)\times \CG(\ov A\,\ov C/AC)\times \CG(\ov B\,\ov C/BC)\big]^*\\
\CG( \ov{AB}\,\ov{AC}\,\ov{BC}/ABC) \ar[u]^{\res} \ar[r]^-{\rho}& \big[\CG(AB)\times \CG(AC)\times \CG(BC)\big]^*
\ar[u]_{\res\times\res\times\res}\\
\CG( \ov{AB}\,\ov{AC}\,\ov{BC}/\ov A\,\ov B\,\ov C) \ar[u]^{\subseteq} \ar[r]^-{\rho_0=\rho}& \CG(\ov A\,\ov B)\times \CG(\ov A\,\ov C)\times \CG(\ov B\,\ov C) \ar[u]_{\subseteq}
}$$
where $\big[\ldots\big]^*$ denotes a adequate subgroup consisting exactly triplets satisfying $(\ast)$,
and 
$$\rho_1:\sigma\mapsto (\sigma\restriction_{\dcl(\ov A,\ov B)},\, \sigma\restriction_{\dcl(\ov A,\ov C)},\, \sigma\restriction_{\dcl(\ov B,\ov C)}).$$
It commutes, the columns form short exact sequences, and maps $\rho_0$, $\rho$ and $\rho_1$ are monomorphisms.
\
\\
\textbf{Claim 1}
The map $\rho_0$ is onto.
\
\\ Proof of the claim: 
Let $(\sigma_1,\sigma_2,\sigma_3)\in \CG(\ov A\,\ov B)\times \CG(\ov A\,\ov C)\times \CG(\ov B\,\ov C)$, we need to find a common extension to an element $\sigma\in\CG( \ov{AB}\,\ov{AC}\,\ov{BC}/\ov A\,\ov B\,\ov C)$.
By Fact \ref{fact:equivalent_conditions_B(n)}.(3), we extend $\sigma_1$ to $\tilde{\sigma}_1\in\aut(\FC/\ov{AC}\,\ov{BC})$. Similarly for $\sigma_2$ and $\sigma_3$:
\begin{table}[h]\centering
\begin{tabular}{r|c|c|c}
  & $\ov{AB}$ & $\ov{AC}$ & $\ov{BC}$ \\ \hline
 $\tilde{\sigma}_1$ & $\sigma_1$ & $\id$ & $\id$ \\
 $\tilde{\sigma}_2$ & $\id$ & $\sigma_2$ & $\id$ \\
 $\tilde{\sigma}_3$ & $\id$ & $\id$ & $\sigma_3$
\end{tabular}
\end{table}
\\
Set $\sigma:=\tilde{\sigma}_1\circ\tilde{\sigma}_2\circ\tilde{\sigma}_3\restriction_{\dcl(\ov{AB},\,\ov{AC},\,\ov{BC})}$.
Here ends the proof of the first claim.
\
\\
\textbf{Claim 2}
The map $\rho_1$ is onto.
\
\\ Proof of the claim: 
Let $(\sigma_1,\sigma_2,\sigma_3)\in \big[\CG(\ov A\,\ov B/AB)\times \CG(\ov A\,\ov C/AC)\times \CG(\ov B\,\ov C/BC)\big]^*$. Again, our goal is to find a common extension 
$\sigma\in \CG(\ov A\,\ov B\,\ov C/ABC)$. Since $\{A,B,C\}$ is $E$-independent, we have that
$\ov A\,\ov B \ind_{\ov E} \ov C$. Of course $\sigma_1$ and $\sigma_2\restriction_{\ov C}$ agree on $\ov E$ and $\ov E\subseteq \ov C$ is regular, therefore, by \cite[Corollary 3.38]{Hoff3}, there exists $\tau\in\aut(\FC)$ such that $\tau\restriction_{\ov A\,\ov B}=\sigma_1\restriction_{\ov A\,\ov B}$ and $\tau\restriction_{\ov C}=\sigma_2\restriction_{\ov C}$. By the property $(\ast)$, we have also
$\tau\restriction_{\ov A\,\ov C}=\sigma_2\restriction_{\ov A\,\ov C}$
and 
$\tau\restriction_{\ov B\,\ov C}=\sigma_3\restriction_{\ov B\,\ov C}$, hence $\sigma:=\tau\restriction_{\dcl(\ov A,\,\ov B,\, \ov C)}$
does the job.
Here ends the proof of the second claim.

It follows that $\rho_1$ and $\rho_0$ are isomorphisms, hence, by the Short Five lemma,
also $\rho$ is an isomorphism.
\end{proof}

We consider a relative algebraic closure $\acl^r_F(A):=\acl(A)\cap F$ for every small subsets $A$ and $F$ of $\FC$.

\begin{fact}\label{fact:model_wikipedia}
\begin{enumerate}
\item Let $M_0\preceq M_1\preceq M_2\preceq\ldots$ be an elementary chain of some structures, of length $\omega$. If each $M_{n+1}$ is $|M_n|^+$-saturated, then every partial elementary map $f:A\to B$, where $A,B\subseteq M_0$, extends to an automorphism of $\bigcup_nM_n$.

\item Let $M_0\preceq M_1\preceq M_2\preceq\ldots$ be an elementary chain of some structures, of length $\kappa^+$. 
If each $M_{i+1}$ is $\max\{|M_i|^+,\kappa\}$-saturated, then every partial elementary map $f:A\to B$, where $A,B\subseteq M_0$, extends to an automorphism of $\bigcup_{i<\kappa^+}M_i$ and
$\bigcup_{i<\kappa^+}M_i$ is $\kappa$-saturated.
\end{enumerate}
\end{fact}

\begin{proof}
The first part is standard. The second part follows by repeating argument from the first part for all limit ordinals below $\kappa^+$.
\end{proof}

\begin{prop}[Generalization of Chatzidakis' Theorem]\label{thm:Zoe}
Suppose that $B(3)$ holds for $T$.
Fix a (very) saturated $\mathcal{L}_P$-structure $(M^\ast,F^\ast)$ such that 
$F^\ast\subseteq M^\ast\preceq\FC$ and $F^\ast$ is PAC.

Assume that $E\subseteq A,B,C_1,C_2\subseteq F^\ast$ are regular extensions of small definably closed subsets such that
$$A\ind_E B,\quad C_1\ind_E A,\quad C_2\ind_E B.$$
Assume that
$\phi\in\aut(\FC/\bar{E})$ satisfies 
$\phi[C_1]=C_2$.
If
there exists $S\subseteq S\CG(F^\ast)$ such that $S\models\tp_{S\CG(F^\ast)}(S\CG(C_1)/S\CG(A))\cup\tp_{S\CG(F^\ast)}(S\CG(C_2)/S\CG(B))$ (where variables are identified via $S\CG(\phi)$), then
there
exists  $C'\subseteq F^\ast$ such that $C'\ind_E AB$, $C'\models \tp_{F^\ast}(C_1/A)\cup\tp_{F^\ast}(C_2/B)$ (the variables for $\tp_F(C_1/A)$ and
$\tp_F(C_2/B)$ are identified via $\phi$) and $S\CG(C')\equiv_{S\CG(\acl_{F}^r(AB))}S$.
\end{prop}

\begin{proof} 
The proof is a generalization of the proof of \cite[Theorem 3.1]{chatzidakis2017}.
In our proof, we are using tools related to stationarity and forking independence (i.e. \cite[Corollary 3.38]{Hoff3}) instead of linear disjointness. We include a detailed proof, since we found some gaps in the exposition of the original proof
(e.g. in the proof of \cite[Theorem 3.1]{chatzidakis2017}, $F^*$ does not need to be an elementary extension of $F$: it is elementarily equivalent to $F$ and there is an automorphism $\sigma$ fixing $(AB)^s$ such that $F^*\cap \sigma((AB)^s(AC)^s(BC)^s)=D$, but $\sigma$ does not necessarily fix $F$)
and we would like to provide a more transparent exposition of this very nice argument. To be in accordance with the proof of \cite[Theorem 3.1]{chatzidakis2017}, we preserve its notations.

Let $S_0^\ast\models\tp_{S\CG(F^\ast)}(S\CG(C_1)/S\CG(A))\cup\tp_{S\CG(F^\ast)}(S\CG(C_2)/S\CG(B))$ (where variables are identified via $S\Phi(:=S\CG(\phi))$), then there exist partial elementary maps (in $S\CG(F^\ast)$) $S\Psi_1:S\CG(C_1)\to S_0^\ast$ (over $S\CG(A)$) and $S\Psi_2:S\CG(C_2)\to S_0^\ast$ (over $S\CG(B)$) such that
$$S\Psi_2\circ S\Phi\restriction_{S\CG(C_1)}=S\Psi_1\restriction_{S\CG(C_1)}.$$
Let $S_1^*$ realize the pushforward of the type $\tp_{S\CG(F^\ast)}\big(S\CG(\acl^r_{F^\ast}(AC_1))/S\CG(A)S\CG(C_1)\big)$ along $S\Psi_1$ (there is such a realization in $S\CG(F^\ast)$ by Corollary \ref{cor:SG.saturated}). 
Similarly, let
$S_2^*$ be a realization of the pushforward of the type $\tp_{S\CG(F^\ast)}\big(S\CG(\acl^r_{F^\ast}(BC_2))/S\CG(B)S\CG(C_1)\big)$ along $S\Psi_2$. 
Without loss of generality we may denote partial elementary maps extending $S\Psi_1$ and $S\Psi_2$ again by $S\Psi_1$ and $S\Psi_2$:
$$S\Psi_1:S\CG(\acl^r_F(AC_1))\xrightarrow[S\CG(A)]{\equiv} S_1^\ast,\qquad S\Psi_2:S\CG(\acl^r_F(BC_2))\xrightarrow[S\CG(B)]{\equiv} S_2^\ast$$

We do not know whether $S\CG(F^\ast)$ is homogeneous and therefore, in a moment, we will construct an auxiliary PAC substructure $F$ such that $S\CG(F)$ is properly homogeneous.

Take a cardinal $\kappa> |T|^+, |A|^+,\,|B|^+,\,|C_1|^+,|\,S_1^\ast|^+,\,|S_2^\ast|^+$ and a chain of elementary extensions $(M_{i},F_i)$ (in $(M^\ast,F^\ast)$) of length $\kappa^+$ such that $A,B,C_1,C_2\subseteq F_0$ and each $(M_{i+1},F_{i+1})$ is $|M_i|^+$-saturated. 
Set $(M,F):=\bigcup_{i<\kappa^+}(M_i,F_i)$. Since $F\preceq F^\ast$, we obtain that $F$ is PAC (by Fact \ref{fact:PAC_elemsubst}). Note that $F$ is $\kappa$-saturated, hence also a $\kappa$-PAC substructure in $\FC$ (by \cite[Proposition 3.9]{Hoff3}).

Because $S\CG(F)$ is interpretable in $(M,F)$ and $(M,F)$ is $\kappa$-saturated, we deduce that also $S\CG(F)$ is $\kappa$-saturated (see Corollary \ref{cor:SG.saturated}).
Let $S_0S_1S_2\subseteq S\CG(F)$ be such that $S_0S_1S_2\equiv_{S\CG(\acl_{F}^r(AB))}S_0^\ast S_1^\ast S_2^\ast$. Say that $S\Gamma$ is a partial elementary map (over $S\CG(\acl_{F}^r(AB))$) such
that $S\Gamma(S_0^\ast S_1^\ast S_2^\ast)=S_0S_1S_2$. 
By Corollary \ref{cor:SG.elementary1}, we have that
$S\CG(F_i)\preceq S\CG(F_{i+1})$ and each $S\CG(F_{i+1})$ is $|S\CG(F_i)|^+$-saturated, so we may use Fact \ref{fact:model_wikipedia} and find $S\Psi_1',S\Psi_2'\in\aut(S\CG(F))$ which extend 
$S\Gamma\circ S\Psi_1$ and $S\Gamma\circ S\Psi_2$ respectively.

The following diagram illustrates our situation on the level of the sorted complete system $S\CG(F)$:
$$\xymatrix{
& S\CG(F) \ar[r]^{S\Psi_1'} & S\CG(F) & S\CG(F) \ar[l]_{S\Psi_2'} & \\
S\CG(\acl^r_F(AC_1)) \ar@[blue][r]_-{S\Psi_1'} \ar@{^{(}->}[ur] & S_1 \ar@{^{(}->}[ur] & & S_2 \ar@{_{(}->}[ul] & S\CG(\acl^r_F(BC_2)) \ar@[blue][l]^-{S\Psi_2'} \ar@{_{(}->}[ul]\\
& S\CG(C_1) \ar@/_2pc/@[yellow][rr]_{S\Phi} \ar@[blue][r]_-{S\Psi_1'} \ar@{_{(}->}[ul]& S_0 \ar@{_{(}->}[ul] \ar@{^{(}->}[ur]& S\CG(C_2) \ar@[blue][l]^-{S\Psi_2'} \ar@{^{(}->}[ur] &}$$
Recall that we have the canonical isomorphisms $\alpha:G\to GS(G)$ and $\beta:S\to SG(S)$ for a sorted profinite group $G$ and a sorted complete system $S$. For $i=1,2,3$, consider the kernel $N_i$ of the following map $\CG(F)\xrightarrow[]{\cong} GS\CG(F)\to GS_i$, which is the composition of the dual of $S_i\to S\CG(F)$ and the canonical isomorphism $\CG(F)\cong GS\CG(F)$. Let $L_i=\ov{F}^{N_i}$. We transfer the previous diagram by the functor $G$, add canonical isomorphisms (wavy lines) and place $\CG(L_i/F)$ with proper restriction maps (red arrows):
\begin{center}
\resizebox{12.5cm}{!}{
$$\xymatrix{
& & \CG(F) \ar@{~>}[dd]_{\alpha} \ar@[green][dddll]_-{\res} & \CG(F) 
\ar@[green][l]_{\alpha^{-1}\circ GS\Psi_1'\circ\alpha}
\ar@[green][r]^{\alpha^{-1}\circ GS\Psi_2'\circ\alpha}
\ar@{~>}[dd]_{\alpha}  
\ar@[red][dddl]_-{\res} \ar@[red][dddr]^-{\res} 
& \CG(F) \ar@{~>}[dd]^{\alpha} \ar@[green][dddrr]^-{\res}& & \\
&&&&&& \\
& & GS\CG(F) \ar@{>>}[ddl]& GS\CG(F) \ar[l]^{GS\Psi_1'} \ar[r]_{GS\Psi_2'} \ar@{>>}[ddl] \ar@{>>}[ddr]   & GS\CG(F) \ar@{>>}[ddr] & & \\
\CG(\acl^r_F(AC_1)) \ar@[green][dd]_-{\res} \ar@{~}@[blue][dr] & & \CG(L_1/F) \ar@{-}@[blue][d]_-{\cong} \ar@[red][dr]^-{\res} & & \CG(L_2/F) \ar@{-}@[blue][d]^-{\cong} \ar@[red][dl]_-{\res} & & \CG(\acl^r_F(BC_2)) \ar@[green][dd]^-{\res}\ar@{~}@[blue][dl] \\
& GS\CG(\acl^r_F(AC_1)) \ar@{>>}[dr] & GS_1 \ar@{>>}@[blue][l]^-{GS\Psi_1'} \ar@{>>}[dr]& \CG(L_0/F) \ar@{-}@[blue][d]_-{\cong} & GS_2 \ar@{>>}@[blue][r]_-{GS\Psi_2'} \ar@{>>}[dl]& GS\CG(\acl^r_F(BC_2)) \ar@{>>}[dl] & \\
\CG(C_1) \ar@{~}@[blue][rr] & & GS\CG(C_1) & GS_0 \ar@{>>}@[blue][l] \ar@{>>}@[blue][r] & GS\CG(C_2) & & \CG(C_2) \ar@{~}@[blue][ll]
\ar@/^2pc/@[yellow][llllll]^-{\Phi}
}$$}
\end{center}
Let us denote the shortest path (in blue) between:
\begin{itemize}
\item $\CG(L_1/F)$ and $\CG(\acl^r_F(AC_1))$ by $\Psi_1'$,
\item $\CG(L_2/F)$ and $\CG(\acl^r_F(BC_2))$ by $\Psi_2'$,
\item $\CG(L_0/F)$ and $\CG(C_1)$ by $\Psi_{01}'$,
\item $\CG(L_0/F)$ and $\CG(C_2)$ by $\Psi_{02}'$.
\end{itemize}
	Take $\phi_0\in\aut(\FC/\overline{E})$ such that $C\ind_{\overline{E}}F$ for $C:=\phi_0^{-1}[C_1]$. Since $\overline{C}\ind_{\overline{E}} \overline{A}$, $\overline{C}_1\ind_{\overline{E}}\overline{A}$ and $\overline{E}\subseteq \overline{A}$ is regular, we may extend $\phi_0:\overline{C}\to \overline{C_1}$ and $\id_{\overline{A}}$ to an automorphism $\phi_1\in\aut(\FC/\overline{A})$ (by \cite[Corollary 3.38]{Hoff3}). Set $\phi_2:=\phi\circ \phi_1$ and $\phi_1':=\phi_1$.

Now, we refine $\phi_2$. Since $\ov{C}\ind_{\ov{E}}\ov{B}$, $\ov{C_2}\ind_{\ov{E}}\ov{B}$ and $\ov{E}\subseteq\ov{B}$ is regular, we may extend $\phi_2:\ov{C}\to\ov{C_2}$ and $\id_{\ov{B}}$ to an automorphism $\phi_2'\in\aut(\FC/\ov{B})$. Note that $\phi_1'[C]=C_1$, $\phi_2'[C]=C_2$,
$\phi_2'\restriction_{\ov{C}}=\phi\circ\phi_1'\restriction_{\ov{C}}$ and we have
$$\phi_1':\ov{AC}\xrightarrow[\ov{A}]{}\ov{AC_1},\qquad
\phi_2':\ov{BC}\xrightarrow[\ov{B}]{}\ov{BC_2}.$$
Set $D_1:=(\phi_1')^{-1}[\acl^r_F(AC_1)]$, $D_2:=(\phi_2')^{-1}[\acl^r_F(BC_2)]$,
$\Phi_1':=\CG(\phi_1'):\CG(\acl^r_F(AC_1))\to \CG(D_1)$ 
and $\Phi_2':=\CG(\phi_2'):\CG(\acl^r_F(BC_2))\to \CG(D_2)$.

The previous diagram simplifies to the following one extended by $\Phi_1'$ and $\Phi_2'$ (in orange):
$$
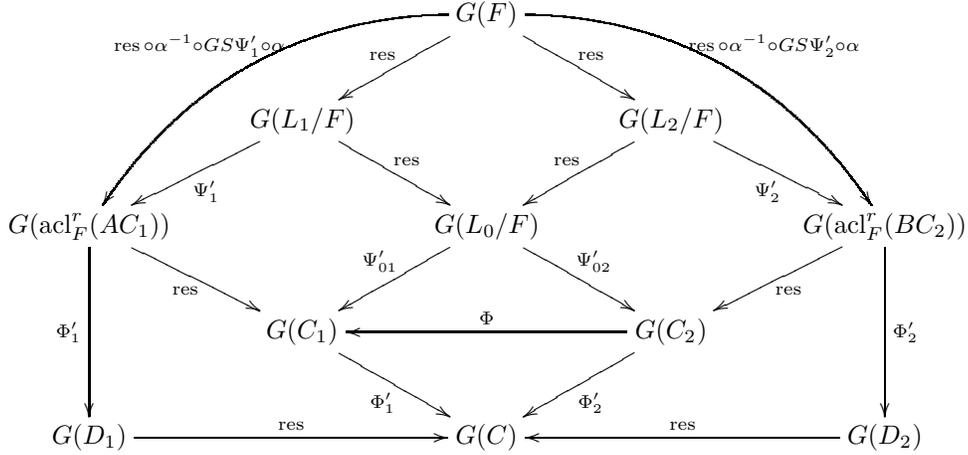
\xymatrix{
& & \CG(F) \ar@/_2.1pc/@[green][ddll]_-{\res\circ \alpha^{-1}\circ GS\Psi_1'\circ\alpha} \ar@/^2.1pc/@[green][ddrr]^-{\res\circ \alpha^{-1}\circ GS\Psi_2'\circ\alpha}
\ar@[red][dl]_-{\res} \ar@[red][dr]^-{\res}
& & \\
& \CG(L_1/F) \ar@[red][dr]^-{\res} \ar@[blue][dl]^-{\Psi_1'} & & \CG(L_2/F) \ar@[red][dl]_-{\res} \ar@[blue][dr]_-{\Psi_2'}& \\
\CG(\acl^r_F(AC_1)) \ar@[orange][dd]_-{\Phi_1'} \ar@[green][dr]_-{\res} & & \CG(L_0/F) \ar@[blue][dl]_-{\Psi_{01}'} \ar@[blue][dr]^-{\Psi_{02}'} & & \CG(\acl^r_F(BC_2)) \ar@[orange][dd]^-{\Phi_2'} \ar@[green][dl]^-{\res}\\
& \CG(C_1) \ar[dr]_-{\Phi_1'} & & \CG(C_2) \ar[dl]^-{\Phi_2'} \ar@[yellow][ll]_-{\Phi}& \\
\CG(D_1)  \ar[rr]^-{\res} && \CG(C)  && \CG(D_2)  \ar[ll]_-{\res}
}$$
Set $\Theta_1:=\Phi_1'\Psi_1'$ and $\Theta_2:=\Phi_2'\Psi_2'$. 
Consider the following map
$$\Theta':\CG(\ov{AB}L_1L_2/F)\to \CG(AB)\times \CG(AC)\times \CG(BC)$$
given by $\Theta'(f):=(f\restriction_{\ov{AB}},\Theta_1(f\restriction_{L_1}),\Theta_2(f\restriction_{L_2}))$.

Note that for any $f\in \CG(F)$ we have
$(\alpha^{-1}\circ GS\Psi_1'\circ \alpha)(f)\restriction_{\ov{A}}=f\restriction_{\ov{A}}$ and so 
$\Theta_1(f\restriction_{L_1})\restriction_{\ov{A}}=f\restriction_{\ov{A}}$.
In the same manner we show that
$\Theta_2(f\restriction_{L_2})\restriction_{\ov{B}}=f\restriction_{\ov{B}}$ .
By the commutativity of the last diagram we see that also 
$\Theta_1(f\restriction_{L_1})\restriction_{\ov{C}}=\Theta_2(f\restriction_{L_2})\restriction_{\ov{C}}$
Therefore we can use Lemma \ref{lemma:galois_times} to conclude that the image of $\Theta'$ is contained in the image of $\rho$. So we extend $\Theta'$ to a map 
$\Theta:\CG(\ov{AB}L_1L_2/F)\to \CG(\ov{AB}\,\ov{AC}\,\ov{BC}/ABC)$ and define $D:=\dcl(\ov{AB}, \ov{AC}, \ov{BC})^{\im\Theta}$.

Note that the following diagram commutes
$$\xymatrix{
\CG(F) \ar[rr]^-{\res} \ar[drr]_-{\res} & & \CG(\ov{AB}L_1L_2/F) \ar[r]^-{\Theta} & \CG(\ov{AB}\,\ov{AC}\,\ov{BC}/D) \ar[dl]^-{\res}\\
&& \CG(\acl^r_F(AB))
}$$

By \cite[Lemma 3.3]{DHL1}, there exists $\sigma\in\aut(\FC/\ov{AB})$ such that $\sigma[D]\subseteq F$ and for each $f\in \CG(F)$ we have that $\Theta(f\restriction_{\dcl(\ov{AB},L_1,L_2)})=\sigma^{-1}f\sigma\restriction_{\dcl(\ov{AB}, \ov{AC}, \ov{BC})}$.

Because compositions $\CG(F)\xrightarrow{\res} \CG(\ov{\sigma[D_1]}F/F)\xrightarrow[\cong]{\res} \CG(\sigma[D_1])\xrightarrow[\cong]{\CG(\sigma)} \CG(D_1)$ and $\CG(F)\xrightarrow{\res} \CG(L_1/F)\xrightarrow[\cong]{\Theta_1} \CG(D_1)$ are equal,
we obtain that their kernels are also equal, hence, by Galois correspondence, $L_1=\dcl(\ov{\sigma[D_1]},F)$.
Similarly $L_2=\dcl(\ov{\sigma[D_2]},F)$.

Note that for any $f\in \CG(F)$ we have
\begin{IEEEeqnarray*}{rCl}
\CG(\sigma^{-1})\Theta_1(f\restriction_{L_1}) &=& \sigma\Theta_1(f\restriction_{L_1})\sigma^{-1} \\
&=& \sigma \Theta(f\restriction_{\dcl(\ov{AB},L_1,L_2)})\restriction_{\ov{AC}}\sigma^{-1} \\
&=& \sigma \Theta(f\restriction_{\dcl(\ov{AB},L_1,L_2)})\sigma^{-1}\restriction_{\ov{\sigma[D_1]}} \\
&=& \sigma\sigma^{-1}f\sigma\restriction_{\dcl(\ov{AB}, \ov{AC}, \ov{BC})}\sigma^{-1}\restriction_{\ov{\sigma[D_1]}} \\
&=& f\restriction_{\ov{\sigma[D_1]}}
\end{IEEEeqnarray*}
Therefore the following diagram commutes
$$\xymatrixcolsep{25mm}\xymatrix{
\CG(F) \ar[d]_-{\res} \ar[r]^-{\big(\alpha^{-1}\circ GS\Psi_1'\circ\alpha \big)^{-1}} & \CG(F) \ar[r]^-{\id} \ar[d]_-{\res} & \CG(F) \ar[d]^-{\res} \\
\CG(\acl_F^r(AC_1)) \ar[dr]_{\res}
 \ar@/_1pc/[rr]_{\CG(\phi_1'\sigma^{-1})} \ar[r]^-{(\Psi_1')^{-1}}& \CG(L_1/F) \ar[r]^-{\CG(\sigma^{-1})\Theta_1} & \CG(\sigma[D_1]) \ar[dl]^{\res}\\
& \CG(A) &
}$$
By the proof of \cite[Proposition 3.6]{DHL1} (i.e. by a back-and-forth construction of an isomorphism of small elementary substructures of $F$) and Fact \ref{fact:model_wikipedia},
$\phi_1'\sigma^{-1}\restriction_{\sigma[D_1]}$ extends to an automorphism $\delta_1\in\aut(F/A)$ sending $\sigma[C]$ to $C_1$, hence $\sigma[C]\models\tp_F(C_1/A)$. In a similar manner we show that $\sigma[C]\models\tp_F(C_2/B)$.

Because $C\ind_E AB$ and $\sigma\in\aut(\FC/\ov{AB})$ we have that $\sigma[C]\ind_E AB$. We put $C':=\sigma[C]$.
It remains to show the moreover part. We have the following commuting diagram
$$\xymatrix{
GS\CG(F) \ar[r]^-{\alpha^{-1}} \ar@{->>}[d] & \CG(F) \ar[rr]^-{\id} \ar@[red][d]_-{\res} & & \CG(F) \ar[d]^-{\res} &\\
GS_1 \ar@[blue]@{-}[r]^-{\cong} \ar@{->>}[d] & \CG(L_1/F) \ar@[red][d]_-{\res}
\ar[rr]^{\CG(\sigma^{-1})\Phi_1'\Psi_1'} \ar@[blue][dr]^-{\Psi_1'} & & \CG(\sigma[D_1]) \ar[dr]^-{\res} &\\
GS_0 \ar@[blue]@{-}[r]^-{\cong} \ar@{=}[d] & \CG(L_0/F) \ar@[blue]@{-}[dl]_-{\cong} \ar@[blue][dr]_-{\Psi_{01}'} & \CG(\acl_F^r(AC_1)) \ar@[green][d]^-{\res} \ar@[orange][r]^-{\Phi_1'} & \CG(D_1) \ar[d]^-{\res} \ar[u]_-{\CG(\sigma^{-1})} & \CG(\sigma[C]) \\
GS_0 \ar@[blue]@{->>}[r] & GS\CG(C_1) \ar@[blue]@{~}[r]& \CG(C_1) \ar@[orange][r]^-{\Phi_1'} & \CG(C) \ar[ur]_-{\CG(\sigma^{-1})} &
}$$
where the external frame can be presented as
$$\xymatrix{
GS\CG(F) \ar@{->>}[d] & \CG(F) \ar[l]_-{\alpha} \ar[d]^-{\res} \\
GS_0 & \CG(\sigma[C]) \ar[l]_-{\text{sorted}}^-{\text{iso.}}
}$$
After taking functor $S$ and extending the part related to the canonical isomorphism $\beta:S\CG(F)\to SGS\CG(F)$, we obtain
$$\xymatrix{
S\CG(F) \ar[r]^-{\beta} & SGS\CG(F) \ar[r]^-{S\alpha} & S\CG(F) \\
S_0 \ar@{^{(}->}[u] \ar[r]^-{\cong} & SGS_0 \ar@{^{(}->}[u] \ar[r]^-{\cong} & S\CG(\sigma[C]) \ar@{^{(}->}[u]
}$$
Since $S\alpha\circ\beta=\id_{S\CG(F)}$ (see Remark \ref{rem.alpha_beta}), we get that $S_0$ and $S\CG(\sigma[C])$ coincide after embedding into $S\CG(F)$, so $S\CG(\sigma[C])\equiv_{S\CG(\acl_{F}^r(AB))}S_0^\ast$.
\end{proof}

\section{Weak Independence Theorem and NSOP$_1$}\label{sec:wit.nsop}
In the following section we will study the relation between the Kim-independence in a PAC substructure $F$ and the Kim-independence in $S\CG(F)$ combined with the forking independence in $\FC$. Our results in this section depend on the previous results about the interpretability of $S\CG(F)$ in $(M,F)$ for some $M\preceq\FC$ containing $F$ (see Theorem \ref{thm:interpretability_SG(K)}). 
As we saw in Proposition \ref{prop:SG.elementary2},
if we want to deal with 
model-theoretic properties of the structure $F$ alone rather than properties of the pair $(M,F)$,
then it is better to assume that $T$ has nfcp. Moreover nfcp is related to the notion of \emph{saturation over $F$} (see \cite[Definition 3.1, Remark 3.6]{PilPol}), which will used in the proof of Lemma \ref{lemma:CZ.equiv_E}. 

In short, if we start with some saturated PAC structure $F^\ast$ and want to prove the independence theorem over a model for $F^\ast$, then we need to pick up some $F\preceq F^\ast$. However to use our methods we need also that $S\CG(F)\preceq S\CG(F^\ast)$ which would follow if we will be able to find $M\preceq M^\ast\preceq\FC$ such that $(M,F)\preceq (M^\ast,F^\ast)$.
We use exactly Lemma \ref{lemma:CZ.equiv_E} to get a proper $M$ in this set-up.

\begin{remark}\label{remark:nfcp.saturation}
Suppose that $T$ has nfcp.
Consider a small substructure $F$ of $\FC$ and any small $M\preceq\FC$ which contains $F$ and which is $|F|^+$-saturated. If we pass to $(M',F')\succeq(M,F)$, then, by \cite[Remark 3.6]{PilPol}, $M'$ will be saturated over $F'$, hence also 
small over $F'$ (for the definition check first lines of \cite{CasanovasZiegler}).
\end{remark}

\begin{fact}\label{fact:PAC_monstermodel}
Let $F\subseteq M\preceq \mathfrak{C}$. Suppose $F$ is PAC in $M$ and $M$ is saturated over $F$. Then $F$ is PAC in $\mathfrak{C}$.
\end{fact}

\begin{proof}
Let $N\subseteq \mathfrak{C}$ be a regular extension of $F$. Suppose $N\models \exists x\,\varphi(x,e)$ for a quantifier free $\varphi(x,y)$ and $e\in F^{|y|}$. Take $n\in N$ such that $\models \varphi(n,e)$. Since $N$ is a regular extension of $F$, the type $p(x):=\tp(n/F)$ is stationary, and it is realized in $M$ (by saturation over $F$). Let $m\in M$ be a realization of $p$. Since $\tp(m/F)=p$ is stationary, $\dcl(F,m)\subseteq M$ is a regular extension of $F$. Since $F$ is PAC and $\varphi(x,e)\in \tp(m/F)$, $\varphi(x,e)$ is realized in $F$.
\end{proof}

\noindent
Let $\bar{\kappa}$ be a cardinal bigger than the size of any interesting set (although still smaller than the saturation of $\FC$).
It is convenient to work with $\bar{\kappa}$-saturated $(M^\ast,F^\ast)$ such that 
$M^\ast\preceq\FC$ is saturated over $F^\ast$ and $F^\ast$ is PAC in $M^\ast$ (and so, by Fact \ref{fact:PAC_monstermodel}, also in $\FC$).
However to obtain such $(M^\ast,F^\ast)$ some assumptions are needed. Suppose that we start with $F$ which is a PAC substructure of $\FC$. Take small $M\preceq\FC$ which contains $F$ and which is $|F|^+$-saturated. In the next step we pick up some $\bar{\kappa}$-saturated $(M^\ast,F^\ast)\succeq (M,F)$.
If we assume that $T$ has nfcp, then it follows that $M^\ast$ is saturated over $F^\ast$ (by Remark \ref{remark:nfcp.saturation}). If we assume that \emph{PAC is a first order property} (\cite[Definition 2.6]{DHL1}), then we obtain that also $F^\ast$ is PAC. 

Therefore under assumptions that $T$ has nfcp and that PAC is a first order property we can obtain the desired structure $(M^\ast,F^\ast)$.
Furthermore,
these assumptions allow us to obtain even a little bit better variant of ``saturation" (preserved after passing to $S\CG(F)$):

\begin{definition}
We say that a structure $M$ is {\em weakly $\kappa$-special} (compare to
the definition of special structure given just before \cite[Proposition 5.1.6]{ChungKeisler})
if there exists ordered set $I$, of size $\kappa^+$, and an increasing chain of elementary extensions $(M_i)_{i\in I}$ such that $M_{i+1}$ is $|M_i|^+$-saturated, where $i\in I$, and $M=\bigcup\limits_{i\in I}M_i$.
\end{definition} 

A standard argument (e.g. Fact \ref{fact:model_wikipedia}) gives us the following remark.

\begin{remark}
\begin{enumerate}
\item If $M$ is weakly $\kappa$-special, then $M$ is $\kappa$-saturated.
\item If $M$ is weakly $\kappa$-special, then $M$ is strongly $\kappa$-homogeneous.
\item If $(M,F)$ is weakly $\kappa$-special, then $S\CG(F)$ is also weakly $\kappa$-special.
\end{enumerate}
\end{remark}

Instead of assuming now that $T$ has nfcp (however the assumption about the finite cover property will still appear in the upcoming results) and that PAC is a first order property,
which can be used to obtain a desired $(M^\ast,F^\ast)$,
we will only assume now that we work in the desired $(M^\ast,F^\ast)$.
More precisely,
from now on, we \textbf{assume} that
$(M^\ast,F^\ast)$ is weakly $\bar{\kappa}$-special (for some big $\bar{\kappa}$) such that $M^\ast\preceq\FC$, $M^\ast$ is saturated over $F^\ast$
and $F^\ast$ is PAC in $M^\ast$ (so also in $\FC$).

\begin{lemma}\label{lemma:CZ.equiv_E}
Suppose that $T$ has nfcp. If $E\preceq F^\ast$, then there exists small $M_0\preceq M^\ast$ such that $(M_0,E)\preceq (M^\ast,F^\ast)$.
\end{lemma}

\begin{proof}
Let $M\preceq M^\ast$ be $|E|^+$-saturated, but small in $M^\ast$, such that $E\subseteq M$.
\
\\
\textbf{Claim}
$(M,E)\equiv_E(M^\ast,F^\ast)$
\
\\ Proof of the claim: Similarly as in the proof of Lemma \ref{lemma:CZ.preceq}, we are using \cite[Proposition 2.1]{CasanovasZiegler} and the proof of \cite[Corollary 2.2]{CasanovasZiegler}, to state each $\mathcal{L}_P$-formula $\Phi(\bar{x})$ is equivalent in $(M,E)$ and in $(M^\ast,F^\ast)$ to $\mathcal{L}_P$-formula of the form
$$Q\bar{\alpha}\in P\;\varphi(\bar{x},\bar{\alpha}),$$
where $\varphi(\bar{x},\bar{\alpha})$ is an $\mathcal{L}$-formula and $Q$ is a tuple of quantifiers.
Let $\bar{a}\in E$. We have the following sequence of equivalences
\begin{IEEEeqnarray*}{rCl}
(M,E)\models\Phi(\bar{a}) &\Leftrightarrow & (M,E)\models Q\bar{\alpha}\in P\;\varphi(\bar{a},\bar{\alpha}) \\
&\Leftrightarrow & E\models Q\bar{\alpha}\;\varphi(\bar{a},\bar{\alpha}) \\
&\Leftrightarrow & F^\ast\models Q\bar{\alpha}\;\varphi(\bar{a},\bar{\alpha}) \\
(M^\ast,F^\ast)\models\Phi(\bar{a}) &\Leftrightarrow & (M^\ast,F^\ast)\models Q\bar{\alpha}\in P\;\varphi(\bar{a},\bar{\alpha})
\end{IEEEeqnarray*}
Here ends the proof of the claim.

We embed $(M,E)$ over $E$ into $(M^\ast,F^\ast)$ and so obtain $(M_0,E)$ as the image of this embedding.
\end{proof}

\begin{definition}\label{def:ind.thm}
Suppose that $\mathfrak{M}$ is a somewhat saturated $\mathcal{L}$-structure and $\ind^{\circ}$ is a ternary relation on all small subsets of $\mathfrak{M}$.
We say that $\ind^{\circ}$ satisfies \emph{the Independence Theorem over a model} if the following holds:
\begin{center}
For every small $M\preceq\mathfrak{M}$, small subsets $A,B\subseteq\mathfrak{M}$ and tuples $c_1,c_2\subseteq\mathfrak{M}$ \\ such that $A\ind^{\circ}_M B$, $c_1\ind^{\circ}_MA$, $c_2\ind^{\circ}_MB$ and $c_1\equiv_M c_2$, \\there exists
\\a tuple $c\subseteq\mathfrak{M}$ such that $c\equiv_{MA}c_1$, $c\equiv_{MB}c_2$ and $c\ind^{\circ}_MAB$.
\end{center}
\end{definition}

\begin{definition}\label{def:extensioin.over.model}
Suppose that $\mathfrak{M}$ is a somewhat saturated $\mathcal{L}$-structure and $\ind^{\circ}$ is a ternary relation on all small subsets of $\mathfrak{M}$.
We say that $\ind^{\circ}$ satisfies \emph{the extension over a model axiom} if the following holds:
\begin{center}
For every small $M\preceq\mathfrak{M}$, tuples $a,b,c\subseteq\mathfrak{M}$ such that $a\ind^{\circ}_M b$, \\ 
 there exists $a'\equiv_{Mb}a$ such that $a'\ind_M^{\circ}bc$.
\end{center}
\end{definition}

Assume that $F^\ast$ is a substructure of $\FC$ and $\ind^{S\CG}$ is a ternary relation on small subsets of $S\CG(F^\ast)$ (more precisely: we treat $S_1$ in $S_1\ind^{S\CG}_{S_0} S_2$ as a tuple) such that
$$S_1\ind_{S_0}^{S\CG} S_2\qquad\text{if and only if}\qquad S_1'\ind_{S_0}^{S\CG}S_2',$$
whenever $S_1S_2\equiv_{S_0}S_1'S_2'$. Define a ternary relation $\ind^{\Game}$ on small subsets of $F^\ast$ in the following way:
$$A\ind_B^\Game C\quad\text{if and only if}\quad A\ind_B C\;\text{ and }\;S\CG(\acl^r_{F^\ast}(AB))\ind_{S\CG(\acl^r_{F^\ast}(B))}^{S\CG}S\CG(\acl^r_{F^\ast}(BC))$$

\begin{theorem}[Weak Independence Theorem]\label{thm:weak.ind.thm}
Suppose that $T$ has nfcp and B(3) holds for $T$.
If $\ind^{S\CG}$ satisfies the Independence Theorem over a model axiom and the extension over a model axiom, then $\ind^\Game$ satisfies the Independence Theorem over a model (in $\theo(F^\ast)$).
\end{theorem}

\begin{proof}
Assume that $E\preceq F^\ast$, $A,B\subseteq F^\ast$, $c_1,c_2\subseteq F^\ast$, $\tp_{F^\ast}(c_1/E)=\tp_{F^\ast}(c_2/E)$,
$$A\ind^{\Game}_E B,\quad c_1\ind^{\Game}_EA,\quad c_2\ind^{\Game}_EB.$$
We want to use Proposition \ref{thm:Zoe}, let us start with some preparations. 

Without loss of generality we may assume that $A=\acl_{F^\ast}^r(AE)$ and $B=\acl_{F^\ast}^r(BE)$, and let us define $C_1:=\acl_{F^\ast}^r(Ec_1)$ and $C_2:=\acl_{F^\ast}^r(Ec_2)$. Note that we have
$$C_1\ind^{\Game}_EA,\quad C_2\ind^{\Game}_EB.$$
Since $\tp_{F^\ast}(c_1/E)=\tp_{F^\ast}(c_2/E)$, there exists an automorphism $\phi_0\in\aut(F^\ast/E)$ sending $c_1$ to $c_2$.
Because $E\subseteq F^\ast$ is regular, we may extend $\phi_0$, by \cite[Fact 3.33]{Hoff3}, to $\phi\in\aut(\FC/\bar{E})$ such that $\phi[F^\ast]=F^\ast$ and $\phi(c_1)=c_2$ and $\phi[C_1]=C_2$. Note that $S\CG(\phi)\in\aut(S\CG(F^\ast)/S\CG(E))$ and $S\CG(\phi)[S\CG(C_1)]=S\CG(C_2)$. Moreover, by Lemma \ref{lemma:CZ.equiv_E} we find $M_0\preceq\FC$ such that $(M_0,E)\preceq(M^\ast,F^\ast)$, hence (by Corollary \ref{cor:SG.elementary1}) we conclude that $S\CG(E)\preceq S\CG(F^\ast)$.

Since $\ind^{S\CG}$ satisfies the Independence Theorem over a model, we obtain
$S\subseteq S\CG(F^\ast)$ such that $S\models\tp_{S\CG(F^\ast)}(S\CG(C_1)/S\CG(A))\cup\tp_{S\CG(F^\ast)}(S\CG(C_2)/S\CG(B))$ (where variables are identified via $S\CG(\phi)$) and $S\ind_{S\CG(E)}^{S\CG}S\CG(A)S\CG(B)$. By the extension over a model axiom (and since $S\CG(A),S\CG(B)\subseteq S\CG(\acl^r_{F^\ast}(AB))$), we may assume that $S\ind_{S\CG(E)}^{S\CG}S\CG(\acl_{F^\ast}^r(AB))$.

By Proposition \ref{thm:Zoe}, there exists  $C'\subseteq F^\ast$ such that $C'\ind_E AB$, $C'\models \tp_{F^\ast}(C_1/A)\cup\tp_{F^\ast}(C_2/B)$ (the variables for $\tp_F(C_1/A)$ and $\tp_F(C_2/B)$ are identified via $\phi$) and $S\CG(C')\equiv_{S\CG(\acl_{F^\ast}^r(AB))}S$.
It follows that $S\CG(C')\ind_{S\CG(E)}^{S\CG}S\CG(\acl_{F^\ast}^r(AB))$ so also $C'\ind_E^\Game AB$.
\end{proof}

From now on and until the end of this section, we denote Kim-independence in $S\CG(F^\ast)$ (see \cite[Definition 3.13]{kaplanramsey2017}) by $\ind^{S\CG}$. 
Our goal is to show that if $S\CG(F^\ast)$ is NSOP$_1$, then $F^\ast$ is NSOP$_1$.
To show that $F^\ast$ is NSOP$_1$ we will use the criterion \cite[Theorem 5.7(2)]{ArtemNick}, where $A\ind^u_C B$ indicates that $\tp(A/BC)$ is finitely satisfiable in $C$.
This idea is different from the original idea from the proof of \cite[Theorem 7.2.6]{NickThesis}, since we noted some gap in the proof of \cite[Theorem 7.2.6]{NickThesis}. After communicating with Nick Ramsey about the gap he suggested to use \cite[Theorem 5.7]{ArtemNick}, what we do.

Suppose that $a,b\in F^\ast\succeq F$ and set $A:=\acl^r_{F^\ast}(Fa)$, $B:=\acl^r_{F^\ast}(Fb)$.

\begin{lemma}\label{lemma:u.LP}
If $a\ind^u_F b$, then $\tp_{\mathcal{L}_P}(\bar{A}/\bar{B})$ is finitely satisfiable in $\bar{F}$.
\end{lemma}

\begin{proof}
Assume that $\tp_F(a/Fb)$ is finitely satisfiable in $F$. Let $a'\in\acl(Fa)=\bar{A}$, $b'\in\acl(Fb)=\bar{B}$ and $(M^\ast,F^\ast)\models\varphi(a',b')$. Consider a quantifier-free $\mathcal{L}$-formula $\zeta$ and $f_1\in F$ such that $\zeta(f_1,a,z_1)\vdash\tp(a'/Fa)$ and $|\zeta(f_1,a,M^\ast)|=n<\omega$. Moreover, consider an $\mathcal{L}_P$-formula $\xi$ and $f_2\in F$ such that $\xi(f_2,b,z_2)\vdash\tp_{\mathcal{L}_P}(b'/Fb)$.
We have that
\begin{equation}\label{eq:CS.u.1}
(M^\ast,F^\ast)\models (\exists\,z_1,\,z_2)\,\big(\varphi(z_1,z_2)\;\wedge\;\zeta(f_1,a,z_1)\;\wedge
\end{equation}
$$\wedge\;(\exists^{=n} z)\,(\zeta(f_1,a,z)\;\wedge\;\xi(f_2,b,z_2) \big).$$
By \cite[Proposition 2.1]{CasanovasZiegler}, there exists a quantifier-free $\mathcal{L}$-formula
$\psi$ and tuple of quantifiers ``$Q$" such that (\ref{eq:CS.u.1}) is equivalent to
$$(M^\ast,F^\ast)\models Q\alpha\in P\;\psi(f_1,f_2,a,b).$$
Therefore $F^\ast\models Q\alpha\;\psi(f_1,f_2,a,b)$ and so $Q\alpha\;\psi(f_1,f_2,x,b)\in\tp_F(a/Fb)$. By finite satisfiability, there exists $f\in F$ such that $F^\ast\models Q\alpha\;\psi(f_1,f_2,f,b)$. Hence $(M^\ast,F^\ast)\models Q\alpha\in P\;\psi(f_1,f_2,f,b)$ and so
$$(M^\ast,F^\ast)\models (\exists\,z_1,\,z_2)\,\big(\varphi(z_1,z_2)\;\wedge\;\zeta(f_1,f,z_1)\;\wedge$$
$$\wedge\;(\exists^{=n} z)\,(\zeta(f_1,f,z)\;\wedge\;\xi(f_2,b,z_2) \big).$$
Let $m^0,b^0\in M^\ast$ be such that
$$(M^\ast,F^\ast)\models \varphi(m^0,b^0)\;\wedge\;\zeta(f_1,f,m^0)\;\wedge$$
$$\wedge\;(\exists^{=n} z)\,(\zeta(f_1,f,z)\;\wedge\;\xi(f_2,b,b^0).$$
Note that $m^0\in\bar{F}$ and there is $h\in\aut_{\mathcal{L}_P}(M^\ast/Fb)$ such that $h(b^0)=b'$.
After ``acting" by $h$ on $(M^\ast,F^\ast)\models \varphi(m^0,b^0)$, we obtain that
$(M^\ast,F^\ast)\models\varphi(h(m^0),b')$. To finish the proof, observe that $h(m^0)\in\bar{F}$.
\end{proof}

\begin{lemma}\label{lemma:u.game}
If $a\ind^u_F b$, then $a\ind^{\Game}_F b$.
\end{lemma}

\begin{proof}
Finite satisfiability of $\tp_F(a/Fb)$ in $F$ implies finite satisfiability of $\qftp(a/Fb)$ in $F$, hence, by quantifier elimination in $T$, finite satisfiability of $\tp(a/Fb)$ in $F$.
Therefore $a\ind_F b$. It remains to show that $S\CG(A)\ind_{S\CG(F)}^{S\CG}S\CG(B)$.

We are done if we show that $\tp_{S\CG(F^\ast)}(S\CG(A)/S\CG(F)S\CG(B))$ is finitely satisfiable in $S\CG(F)$.
By Lemma \ref{lemma:u.LP}, we know that $q:=\tp_{\mathcal{L}_P}(\bar{A}/\bar{B})$ is finitely satisfiable in $\bar{F}$.

Suppose $S\CG(F^\ast)\models\Theta(a,b,f)$ for some $a\in S\CG(A), b\in S\CG(B), f\in S\CG(F)$, 
where $a=(a_1,\ldots,a_n)\in m(k_1,J_1)\times\ldots\times m(k_n,J_n)$.
Let $\theta(x,y,z)$ be an $\mathcal{L}_P$-formula which is translation of $\Theta(X,Y,Z)$ and 
let $c_A,c_B,c_F$ be choice functions such that $\im(c_A)\subseteq\bar{A}$, $\im(c_B)\subseteq\bar{B}$ and $\im(c_F)\subseteq\bar{F}$ (see Remark \ref{rem:translation.B}). We have that $(M^\ast,F^\ast)\models\theta(c_A(a),c_B(b),c_F(f))$ and so $\theta(x',c_B(b),c_F(f))\;\wedge\;x'\in U^{k_1}_{J_1}\times\ldots\times U^{k_n}_{J_n}$ belongs to $q$.

There exists $d'=(d^1_1,d^2_1,\ldots,d^1_n,d^2_n)\in\bar{F}$, where $(d^1_i,d^2_i)\in U^{k_i}_{J_i}(M^\ast)$, such that $(M^\ast,F^\ast)\models \theta(d',c_B(b),c_F(f))$. 
Note that $d_i:=F_{k_i,J_i}^{-1}\big(\epsilon_{k_i}[(d^1_i,d^2_i)/\approx]\big)\in S\CG(F)$ and for $d:=(d_1,\ldots,d_n)\in S\CG(F)$ it follows that $S\CG(F^\ast)\models \Theta(d,b,f)$.
\end{proof}

\begin{theorem}\label{thm:SG.NSOP1}
(Suppose $T$ has nfcp and $B(3)$ holds for $T$.)
If $S\CG(F^\ast)$ is NSOP$_1$, then $F^\ast$ is NSOP$_1$.
\end{theorem}

\begin{proof}
We will use the criterion \cite[Theorem 5.7(2)]{ArtemNick}. In order to show that $T$ is NSOP$_1$, we must show that given $c_1,c_2,a,b$ and $F$ satisfying $$c_1a\equiv_F c_2b,\ b\ind^u_F a,\ a\ind^u_F c_1,\ b\ind^u_F c_2,$$ there is $d\in F^\ast$ such that $da\equiv_F db\equiv_F c_1a$.

By Lemma \ref{lemma:u.game} and since $\ind^{\Game}$ is symmetric, we have that
$$A\ind^{\Game}_F B,\quad c_1\ind^{\Game}_F A,\quad c_2\ind^{\Game}_F B.$$
Since $c_1\equiv_F c_2$, we can use Theorem \ref{thm:weak.ind.thm} to get $d\in F^\ast$ such that
$c_2\equiv_{FB}d\equiv_{FA}c_1$, hence $da\equiv_Fc_1a\equiv_F c_2b\equiv_F db$.
\end{proof}

\subsection{Description of independence}
Suppose that $T$ has nfcp and
assume that $S\CG(F^\ast)$ is NSOP$_1$ and that $F^\ast$ is NSOP$_1$.
Recall that $a,b\in F^\ast\succeq F$ and set $A:=\acl^r_{F^\ast}(Fa)$, $B:=\acl^r_{F^\ast}(Fb)$.
We want to show that $a\ind^K_F b$ if  and only if $a\ind_F^{\Game} b$, but is is enough to show that
$A\ind^K_F B$ if and only if $a\ind_F^{\Game} b$ (by symmetry, monotonicity and \cite[Corollary 5.17]{kaplanramsey2017} it follows that:
$a\ind^K_F b$ iff $A\ind_F^K B$).

By Lemma \ref{lemma:CZ.equiv_E}, there exists small $M\preceq M^\ast$ such that $F\subseteq M$ and $(M,F)\preceq (M^\ast,F^\ast)$. Let $(M^{\ast\ast},F^{\ast\ast})\succeq(M^\ast,F^\ast)$ be special and at least $|M^\ast|^+$-saturated (it will play the role of a ``monster monster" for global types in $M^\ast$, we do not require that $F^{\ast\ast}$ is PAC).

Consider $\bar{B}=acl(B)$ as tuple, where $\bar{B}=B^{\smallfrown}(\bar{B}\setminus B)$ (i.e. elements from $B$ occupy first positions in the tuple $\bar{B}$). Note that $\tp_{\mathcal{L}_P}(\bar{B}/M)$ is finitely satisfiable in $\bar{F}$ (hence also in $M$). We define $q$ as a maximal set of $\mathcal{L}_P(M^\ast)$-formulas in variables corresponding to $\bar{B}$ which is finitely satisfiable in $\bar{F}$ and which contains $\tp_{\mathcal{L}_P}(\bar{B}/M)$ (in other words: we take a coheir extension). Note that $q\in S_{\mathcal{L}_P}(M^\ast)$, $q$ is finitely satisfiable in $\bar{F}$ and $\bar{F}$-invariant (hence finitely satisfiable in $M$ and also $M$-invariant).
There exists $B^\ast\subseteq \Fs$ such that $q=\tp_{\mathcal{L}_P}(\bar{B^\ast}/M^\ast)$.
There also exists $(B_i)_{i<\omega}\subseteq F^\ast$ such that $\bar{B_i}\models q\restriction_{M\bar{B}_{<i}}$ and $\bar{B}_0=\bar{B}$ (as a tuple).

The type $q^F:=\tp_F(B^\ast/F^\ast)$ is finitely satisfiable in $F$ and hence also $F$-invariant.
We define $q^F_0$ as the quantifier-free part of $q^F$ and we choose $q^M_0$ to be a maximal set of quantifier free $\mathcal{L}(M^\ast)$-formulas in variables corresponding to $B^\ast$ which contains $q^F_0$ and which is finitely satisfiable in $F$. By the assumption that $T$ has quantifier elimination, $q_0^M$ determines the unique complete type $q^M$ in $S_{\mathcal{L}}(M^\ast)$, which is finitely satisfiable in $F$ and $F$-invariant.

\begin{remark}\label{remark:q^F.morley}
The sequence $(B_i)_{i<\omega}$ is a Morley sequence in $q^F=\tp_F(B^\ast/F^\ast)$ over $F$.
\end{remark}

\begin{proof}
For each $f\in\aut_{\mathcal{L}_P}(\Ms)$, we have that $f\restriction_{\Fs}\in\aut_{\mathcal{L}}(\Fs)$.
Because $\bar{B_i}\models\tp_{\mathcal{L}_P}(\bar{B^\ast}/M\bar{B_{<i}})$, there exists $f\in\aut_{\mathcal{L}_P}(\Ms/M\bar{B_{<i}})$ such that $f[\bar{B_i}]=\bar{B^\ast}$ (as tuples).
Hence $f[B_i]=B^\ast$ and $f\restriction_{\Fs}\in\aut_{\mathcal{L}}(\Fs/FB_{<i})$.
\end{proof}

\begin{remark}\label{remark:q^M.morley}
$(B_i)_{i<\omega}$ is a Morley sequence in $q^M$ over $F$.
\end{remark}

\begin{proof}
Because $q^F_0\restriction_{FB_{<i}}=q^M_0\restriction_{FB_{<i}}$ and $B_i\models q^F\restriction_{FB_{<i}}$, we obtain that
$B_i\models q^M_0\restriction_{FB_{<i}}$ and so $B_i\models q^M\restriction_{FB_{<i}}$. 
\end{proof}

Note that $\bar{B^\ast}\cap\Fs=B^\ast$, $\bar{B_i}\cap\Fs=B_i$ for each $i<\omega$ 
(all these tuples are $\mathcal{L}_P$-equivalent to $B=\bar{B}\cap\Fs$).
Consider
$$r:=\tp_{S\CG(\Fs)}\Big( S\CG(B^\ast)/S\CG(F^\ast)\Big).$$

\begin{remark}\label{remark:r.is.fs}
The type $r$ is finitely satisfiable in $S\CG(F)$ hence also $S\CG(F)$-invariant.
\end{remark}

\begin{proof}
Since $q=\tp_{\mathcal{L}_P}(\bar{B}^{\ast}/M^{\ast})$ is finitely satisfiable in $\bar{F}$, we can repeat the argument from the proof of Lemma \ref{lemma:u.game}.
\end{proof}

\begin{remark}\label{remark:r.morley}
$(S\CG(B_i))_{i<\omega}$ is a Morley sequence in $r$ over $S\CG(F)$.
\end{remark}

\begin{proof}
For each $i<\omega$, we will find an automorphism
$$F\in\aut\big(S\CG(\Fs)/S\CG(F)S\CG(B_0)\ldots S\CG(B_{i-1})\big)$$
such that $F(S\CG(B_i))=S\CG(B^\ast)$.
Because $\bar{B_i}\models q\restriction_{M\bar{B}_{<i}}$, there exists $f\in\aut(\Ms/M\bar{B}_{<i})$ such that
$f[\bar{B}_i]=\bar{B^\ast}$ with respect to orderings of tuples $\bar{B}_i$ and $\bar{B^\ast}$. We see that $f[B_i]=B^\ast$,
hence $S\CG(f)[S\CG(B_i)]
=S\CG(B^\ast)$. 
We set $F:=S\CG(f)$ and easily check the remaining properties.
\end{proof}

\begin{lemma}\label{lemma:LP.equiv}
Suppose that $C,D_0,D_1$ are small in $F^\ast$, $C\subseteq F^\ast$ is regular and $D_0\equiv_C D_1$ (in $\theo(F^\ast)$). Then $D_0\equiv_{\acl_{\mathcal{L}}(C)}^{\mathcal{L}_P}D_1$.
\end{lemma}

\begin{proof}
Since $D_0\equiv_C D_1$, there exists $g\in\aut(F^\ast/C)$ such that $g[D_0]=D_1$.
Because $C\subseteq F^\ast$ is regular, we can use \cite[Fact 3.33]{Hoff3} to conclude that
there exists $h\in\aut_{\mathcal{L}}(\Ms/\acl(C))$ such that $h\restriction_{F^\ast}=g$.

By \cite[Proposition 2.1]{CasanovasZiegler}, we can restrict our attention to \emph{bounded} formulas: let
$$(M^\ast,F^\ast)\models Q\alpha\in P\;\varphi(d_0,a,\alpha),$$
where $d_0\subseteq D_0$, $a\in\acl(C)$ and $\varphi$ is a quantifier-free $\mathcal{L}$-formula. Our goal is to show that $(M^\ast,F^\ast)\models Q\alpha\in P\;\varphi(d_1,a,\alpha)$, where $d_1:=g(d_0)$.

For all $\alpha\in F^\ast$,
$M^\ast\models\varphi(d_0,a,\alpha)$ is equivalent to $M^\ast\models\varphi(d_1,a,g(\alpha))$ (by passing to $\Ms$ and using $h$). Because $g$ is a bijection on $F^\ast$,
if $(M^\ast,F^\ast)\models Q\alpha\in P\;\varphi(d_0,a,\alpha)$, then
$(M^\ast,F^\ast)\models Q\alpha\in P\;\varphi(d_1,a,\alpha)$.
\end{proof}

The phrase ``respecting the enumeration of tuples $\bar{B}_i$'s" refers to the previously chosen enumeration of tuples $\bar{B}_i$'s as realizations of type $q$.

\begin{prop}\label{prop:K.game}
If $A\ind^K_F B$, then $a\ind^{\Game}_F b$.
\end{prop}

\begin{proof}
The proof reuses a nice argument from the proof of \cite[Theorem 7.2.6]{NickThesis} (the argument in the proof of \cite[Theorem 7.2.6]{NickThesis} is used to show something different, but it is enough general to be adapted to show that some independence relation holds).
Let $C_{n,0}:=\acl^r_{F^\ast}(B_0,\ldots,B_{2^n-1})$, $C_{n,1}:=\acl^r_{F^\ast}(B_{2^n},\ldots, B_{2^{n+1}-1})$. 
Note that $\bar{B}_0\ldots\bar{B}_{2^n-1}\equiv^{\mathcal{L}_P}_F\bar{B}_{2^n}\ldots\bar{B}_{2^{n+1}-1}$ (respecting the enumeration of tuples $\bar{B}_i$'s).
Note also that
$C_{n,1}\ind^u_F C_{n,0}$.

We will recursively construct a sequence $(A_n)_{n<\omega}$ such that
\begin{itemize}
\item $A_{n+1}\equiv_{C_{n+1,0}}A_n$,
\item $A_n\ind^K_F C_{n+1,0}$,
\item $A_n \bar{B}_0\ldots\bar{B}_{2^n-1}\equiv^{\mathcal{L}_P}_F A_n\bar{B}_{2^n}\ldots\bar{B}_{2^{n+1}-1}$ (respecting the enumeration of tuples $\bar{B}_i$'s).
\end{itemize}
\underline{Case of $A_0$:}
\
\\
Since $\bar{B}_0\equiv^{\mathcal{L}_P}_F\bar{B}_1$, there exists $A'\subseteq F^\ast$ such that
$A\bar{B}_0\equiv^{\mathcal{L}_P}_F A'\bar{B}_1$ (respecting the enumeration of tuples $\bar{B}_i$'s). It follows that $A\ind^K_F B_0$, $A'\ind^K_F B_1$ and $B_0\ind^K_F B_1$. By the Independence Theorem, there exists $A''\subseteq F^\ast$ such that $A''B_0\equiv_F AB_0$, $A''B_1\equiv_F A'B_1$ and $A''\ind^K_F B_0B_1$.
Set $A_0:=A''$. 
Because
$A_0\ind^K_F B_0B_1$, we have that $A_0\ind^K_F C_{1,0}$. 
By Lemma \ref{lemma:LP.equiv}, it follows that $A''\bar{B}_0\equiv^{\mathcal{L}_P}_F A\bar{B}_0$ and $A'\bar{B}_1\equiv^{\mathcal{L}_P}_F A''\bar{B}_1$,
hence $A_0\bar{B}_0\equiv^{\mathcal{L}_P}_F A_0\bar{B}_1$ (respecting enumeration of tuples $\bar{B}_i$'s).
\
\\
\underline{Recursion step:}
\
\\
Suppose that $n\geqslant 0$ and we have defined $A_n$ satisfying our demands.
Since $\bar{B}_0\ldots\bar{B}_{2^{n+1}-1}\equiv^{\mathcal{L}_P}_F\bar{B}_{2^{n+1}}\ldots\bar{B}_{2^{n+2}-1}$,
we can find $A'_n\subseteq F^\ast$ such that 
$$A_n\bar{B}_0\ldots\bar{B}_{2^{n+1}-1}\equiv^{\mathcal{L}_P}_F A'_n\bar{B}_{2^{n+1}}\ldots\bar{B}_{2^{n+2}-1}$$
(respecting enumeration of tuples $\bar{B}_i$'s).
We have that $A_n\ind^K_F B_0\ldots B_{2^{n+1}-1}$, hence also
$A'_n\ind^K_F B_{2^{n+1}}\ldots B_{2^{n+2}-1}$ and so
$A_n\ind^K_F C_{n+1,0}$, $A'_n\ind^K_F C_{n+1,1}$ and $C_{n+1,0}\ind^K_F C_{n+1,1}$.

By the Independence Theorem, there exists $A_{n+1}\models\tp_F(A_n/C_{n+1,0})\cup\tp_F(A'_n/C_{n+1,1})$ such that $A_{n+1}\ind^K_F C_{n+1,0}C_{n+1,1}$.
Therefore $A_{n+1}\equiv_{C_{n+1,0}} A_n$, $A_{n+1}\ind^K_F C_{n+2,0}$.
Using Lemma \ref{lemma:LP.equiv} for $A_{n+1}\equiv_{C_{n+1,0}} A_n$ and $A_{n+1}\equiv_{C_{n+1,1}} A'_n$, we obtain
$$A_{n+1}\bar{B}_0\ldots\bar{B}_{2^{n+1}-1}\equiv^{\mathcal{L}_P}_F
A_n\bar{B}_0\ldots\bar{B}_{2^{n+1}-1},$$
$$ A_{n+1}\bar{B}_{2^{n+1}}\ldots\bar{B}_{2^{n+2}-1}\equiv^{\mathcal{L}_P}_F A'_n\bar{B}_{2^{n+1}}\ldots\bar{B}_{2^{n+2}-1}.$$
Therefore
$$A_{n+1}\bar{B}_0\ldots\bar{B}_{2^{n+1}-1}\equiv^{\mathcal{L}_P}_F
A_{n+1}\bar{B}_{2^{n+1}}\ldots\bar{B}_{2^{n+2}-1}$$
(respecting enumeration of tuples $\bar{B}_i$'s). Here ends our recursive construction.

Note that $A_{n+1}\equiv_{C_{n+1,0}} A_n$ leads to $A_{n+1}\equiv^{\mathcal{L}_P}_{\acl_{\mathcal{L}}(C_{n+1,0})} A_n$. Suppose that $b_0\subseteq\bar{B}_0$, by $b_i$ we denote element of $\bar{B}_i$ corresponding to variables given by $b_0\subseteq\bar{B}_0$.
\
\\
\\
\textbf{Claim}
If $\varphi(x',b_0)\in\tp_{\mathcal{L}_P}(\bar{A}/\bar{B}_0)$, then there exists an infinite $I\subseteq\mathbb{N}$ such that $\{\varphi(x',b_i)\;|\;i\in I\}$ is consistent.
\
\\ Proof of the claim: 
We start with $(M^\ast,F^\ast)\models\varphi(a',b_0)$, where $a'\in\bar{A}$.
Let $\zeta$ be a quantifier-free $\mathcal{L}$-formula satisfying $\zeta(x',a)\vdash\tp(a'/A)$ for some $a\in A$,
say $|\zeta(M^\ast,a)|=n$.
Consider $\psi(a,b_0)$ given as
$$(\exists\,x')\,\big(\varphi(x',b_0)\;\wedge\;\zeta(x',a)\;\wedge\;(\exists^{=n}y')\,(\zeta(y',a))\big).$$
We have $\psi(x,b_0)\in\tp_{\mathcal{L}_P}(A/\bar{B}_0)=\tp_{\mathcal{L}_P}(A_0/\bar{B}_0)$,
so $(M^\ast,F^\ast)\models \varphi(a_0,b_0)$ for some $a_0\in A_0$.
Using $A_n$'s, we can show that the set $\{\psi(x,b_i)\;|\;i<\omega\}$ is consistent.
Let $a^\ast\models\{\psi(x,b_i)\;|\;i<\omega\}$ for some $a^\ast\in M^\ast$.
It means that for each $i<\omega$ we have
$$(M^\ast,F^\ast)\models 
(\exists\,x')\,\big(\varphi(x',b_i)\;\wedge\;\zeta(x',a^\ast)\;
\wedge\;(\exists^{=n}y')\,(\zeta(y',a^\ast))\big).$$
Because $|\zeta(M^\ast,a^\ast)|=n$,
there is $m\in\zeta(M^\ast,a^\ast)$ such that for infinitely many $i<\omega$ we have
$(M^\ast,F^\ast)\models\varphi(m,b_i)$.
Here ends the proof of the claim.
\
\\
\\
By Kim's lemma (\cite[Proposition 2.2.6]{kim1}), Remark \ref{remark:q^M.morley} and Claim (subsequence of a Morley sequence is a Morley sequence), we obtain $A\ind_F B$.

Suppose that $S\CG(A)\nind_{S\CG(F)}^{S\CG}S\CG(B)$. By a generalization of Kim's lemma (\cite[Theorem 3.16.(3)]{kaplanramsey2017}), there exists $\Psi(X,b')\in\tp_{S\CG(\Fs)}(S\CG(A)/S\CG(B))$,
where $b'\in S\CG(B)$,
which $q'$-divides (c.f.  \cite[Definition 3.11]{kaplanramsey2017}) for every global $S\CG(F)$-invariant $q'\supseteq\tp_{S\CG(\Fs)}(b'/S\CG(F))$.

Because $\bar{B_i}\models q\restriction_{M\bar{B}_{<i}}$ for each $i<\omega$, there exist $(f_i)_{i<\omega}\subseteq\aut_{\mathcal{L}_P}(M^\ast/M)$ such that $\bar{B}_i=f_i[\bar{B}_0]$.
Moreover $f_i\restriction_{F^\ast}\in\aut(F^\ast/F)$ for each $i<\omega$.

Let $S\CG(\Fs)\models \Psi(a',b')$ for some $a'\in S\CG(A)$, where $a'=(a'_1,\ldots,a'_n)\in m(k_1,J_1)\times\ldots\times m(k_n,J_n)$. Let $\psi(x,y)$ be an $\mathcal{L}_P$-formula which is a translation of
$\Psi(X,Y)$ and let $c_A$ and $c_B$ be choice functions such that $\im(c_A)\subseteq\bar{A}$ and
$\im(c_B)\subseteq\bar{B}$ (see Remark \ref{rem:translation.B}).
We have that $(M^\ast,F^\ast)\models \psi(c_A(a'),c_B(b'))$
and so $\psi(x,b_0)\;\wedge\;x\in U^{k_1}_{J_1}\times\ldots\times U^{k_n}_{J_n}$, where $b_0:=c_B(b')$, belongs to $\tp_{\mathcal{L}_P}(\bar{A}/\bar{B}_0)$.

By Claim, there exists an infinite $I\subseteq\mathbb{N}$ and $a^\ast\in U^{k_1}_{J_1}(M^\ast)\times\ldots\times U^{k_n}_{J_n}(M^\ast)$ such that for each $i\in I$ we have $(M^\ast,F^\ast)\models \psi(a^\ast,b_i)$ (note that $b_i=f_i(b_0)$).

Consider $b'_i:=S\CG(f_i)(b')$ for each $i<\omega$. By Remark \ref{remark:r.morley}, $(b'_i)_{i<\omega}$ is a Morley sequence over $S\CG(F)$, so also $(b'_i)_{i\in I}$ is a Morley sequence over $S\CG(F)$.
We will finish the proof of the Proposition if we show that the set $\{\Psi(X,b'_i)\;|\;i\in I\}$ is consistent.
This can be done by translating $\psi(a^\ast,b_i)$ into $\Psi(a^{\ast\prime},b_i')$, for one $a^{\ast\prime}\in S\CG(F^\ast)$ proper for each $i\in I$.
\end{proof}

\begin{prop}\label{prop:game.K}
If $a\ind^{\Game}_F b$, then $A\ind^K_F B$.
\end{prop}

\begin{proof}
We follow here the proof of \cite[Theorem 7.2.6]{NickThesis}, but using our generalizations of all necessary facts.

We assume that $a\ind^{\Game}_F b$, so, by definition, $A\ind^{\Game}_F B$.
Similarly as in the proof of Proposition \ref{prop:K.game}, we will construct a sequence $(A_n)_{n<\omega}\subseteq F^\ast$ such that
\begin{itemize}
\item $A_{n+1}\equiv_{C_{n+1,0}} A_n$,
\item $A_n\ind^{\Game}_F C_{n+1,0}$,
\item $A_n B_0\ldots B_{2^n-1}\equiv_F A_n B_{2^n}\ldots B_{2^{n+1}-1}$
(respecting the enumeration of tuples $\bar{B}_i$'s).
\end{itemize}
Instead of repeating the whole proof of \cite[Theorem 7.2.6]{NickThesis}, which is similar to a part of the proof of Proposition \ref{prop:K.game}, we only sketch how to find $A_0$ (which is actually missing in the proof of \cite[Theorem 7.2.6]{NickThesis}).

Since $B_0\equiv_F B_1$, there is $A'\subseteq F^\ast$ such that $AB_0\equiv_F A'B_1$.
We see that $A\ind^{\Game}_F B_0$ implies that $A'\ind^{\Game}_F B_1$. Moreover, $B_0\ind^u_F B_1$ 
leads, by Lemma \ref{lemma:u.game}, to $B_0\ind^{\Game}_F B_1$. Because $S\CG(F^\ast)$ is NSOP$_1$, we 
can use the Independence Theorem for $\ind^{S\CG}$ to obtain $S_0\subseteq S\CG(F^\ast)$ such that 
$S_0\models\tp_{S\CG(F^\ast)}(S\CG(A)/S\CG(B_0))\cup\tp_{S\CG(F^\ast)}(S\CG(A')/S\CG(B_1))$ and $S_0\ind^{S\CG}_{S\CG(F)}S\CG(B_0)S\CG(B_1)$. By extension over a model, there exists $S\equiv_{S\CG(B_0)S\CG(B_1)}S_0$
such that $S\ind^{S\CG}_{S\CG(F)}S\CG(C_{1,0})$. Now, we use Proposition \ref{thm:Zoe} to get 
$A_0\subseteq F^\ast$ such that $S\CG(A_0)\equiv_{S\CG(C_{1,0})}S$ (so $S\CG(A_0)\ind^{S\CG}_{S\CG(F)}S\CG(C_{1,0})$), $A_0\ind_F B_0B_1$ (so $A_0\ind^{\Game}_F C_{1,0}$), and $A_0B_0\equiv_F AB_0$
and $A_0B_1\equiv_F A'B_1$ (so $A_0B_0\equiv_F A_0B_1$ --- respecting the enumeration of tuples $\bar{B}_i$'s).

By \cite[Theorem 3.16]{kaplanramsey2017}, it is enough to show that for each $\varphi(x,b_0)\in\tp_F(A/B_0)$ the set $\{\varphi(x,b_i)\:|\:i<\omega\}$ is consistent (where $b_i$ is an element of $B_i$ corresponding to variables given by $b_0\in B_0$). We have that $\varphi(x,b_0)\in\tp_F(A_0/B_0)$, say $F^\ast\models \varphi(a_0,b_0)$ for some $a_0\in A_0$. By our construction for each $n<\omega$ there is $k<\omega$ such that
$$F^\ast\models \bigwedge\limits_{i\leqslant n}\varphi(a_k,b_i),$$
so the proof ends.
\end{proof}

\begin{cor}\label{cor:Kim.ind.description}
The following are equivalent
\begin{enumerate}
\item $a\ind^K_F b$,
\item $A\ind^K_F B$,
\item $a\ind^{\Game}_F b$,
\item $A\ind^{\Game}_F B$.
\end{enumerate}
\end{cor}

\begin{proof}
The equivalence (2)$\iff$(3) follows by Propositions \ref{prop:K.game} and \ref{prop:game.K}. Other  equivalences are clear.
\end{proof}

\begin{remark}\label{remark:polkowska}
Suppose that $T$ has nfcp and the property $B(3)$.
As previously, we assume
that
$(M^\ast,F^\ast)$ is weakly $\bar{\kappa}$-special (for some big $\bar{\kappa}$) such that $M^\ast\preceq\FC$, $M^\ast$ is saturated over $F^\ast$
and $F^\ast$ is PAC in $M^\ast$ (so also in $\FC$).

We recall here a result of Polkowska from \cite{Polkowska}.
Polkowska shows that bounded PAC substructures --- under the additional assumption that PAC is a first order property --- are simple. Her additional assumption is needed in her paper to fix some saturated PAC structure, which in our case is $F^\ast$ and therefore we do not assume that PAC is a first order property. We have the following:
\begin{center}
if $F^\ast$ is bounded, then $F^\ast$ is simple.
\end{center}
Let us describe another way of getting her result.
Suppose that $F^\ast$ is bounded.
The structure $F^\ast$ is bounded exactly when $\CG(F^\ast)$ is a small profinite group, i.e. it has only finitely many open normal subgroups (of finite index), hence $S\CG(F^\ast)$ is a structure such that each sort is finite. Hence $S\CG(F^\ast)$ is stable
and every element is in the algebraic closure of the empty set.
Thus the forking independence in $S\CG(F^\ast)$ is the trivial independence relation, that is, any triple belongs to the forking independence ternary relation.
In this situation the previously defined ternary relation $\ind^{\Game}$ is given in the following way
$$a\ind^{\Game}_A b\qquad\iff\qquad a\ind_A b$$
for any $a,b\in F^\ast$ and small $A\subseteq F^\ast$ ($\ind$ is the forking independence in $M^\ast$).
By Theorem \ref{thm:weak.ind.thm}, $\ind^{\Game}$ satisfies the Independence Theorem (over a model).
By the proof of \cite[Proposition 3.19]{Polkowska}, the above ternary relation satisfies the Extension axiom --- over sets containing some fixed $F\preceq F^\ast$, and since it is given by the forking independence in $M^\ast$ it also satisfies all the missing conditions from \cite[Theorem 3.3.1]{kim1}, so we can state that $F^\ast$ is simple in the language $\mathcal{L}$ expanded by adding parameters for some fixed $F\preceq F^\ast$ (and so $F^\ast$ is simple as an $\mathcal{L}$-structure) and the above $\ind^{\Game}$ coincides with the forking independence in $F^\ast$ (over sets containing some fixed $F\preceq F^\ast$, compare with \cite[Proposition 3.19]{Polkowska}).
\end{remark}

\subsection{NSOP$_n$ for $n>1$}
In this subsection, we generalize NSOP$_n$ criteria (for $n>1$) of PAC fields (\cite[Theorem 3.9]{chatzidakis2017} and \cite[Proposition 7.2.8]{NickThesis}) to PAC structures. We basically follow the proof schemes of \cite[Theorem 3.9]{chatzidakis2017} and \cite[Proposition 7.2.8]{NickThesis}. Also we give a detailed proof of a Galois theoretic result(\cite[Lemma 2.15]{chatzidakis2017}), used in the proof of \cite[Theorem 3.9]{chatzidakis2017}, in a general model theoretic setting (see Lemma \ref{lem:independenct_realization_sortedcompletesystem}). 
We need to do this detailed work because in \cite{chatzidakis2017}, 
there are minor unclear things. For example, in the proof of \cite[Lemma 2.15]{chatzidakis2017}, $L$ is not necessarily a regular extension of $F$
(we only know that $L\cap M=F$ and not that $L\cap \bar F=F$, see Example \ref{extra_example}). 


Through this subsection, we say a structure is {\em weakly special} if it is weakly $\bar\kappa$-special for a big enough cardinal $\bar \kappa$.

\begin{prop}\label{prop:galois_gps_under_conjugation}
Let $E,F\subseteq \FC$ and let $L$ be a Galois extension of $E$. Let $\phi\in \aut(\FC)$ be such that $\phi[E]\subseteq F$ and $\hat \phi :\CG(F)\rightarrow \CG(L/E)$, given by $\sigma\mapsto (\phi^{-1}\circ \sigma\circ \phi)\restriction_L$, is onto. Let $N=\ker \hat \phi$. Then we have that
\begin{enumerate}
	\item $\bar F^N=\dcl(\phi[L], F)$, and
	\item $\phi[L]\cap F=\phi[E]$. 
\end{enumerate} 
\end{prop}
\begin{proof}
$(1)$ $(\supseteq)$ Take $\sigma\in N$. Then, we have that $\phi^{-1}\circ \sigma\circ \phi(x)=x$ for all $x\in L$ so $\sigma(\phi(x))=\phi(x)$ for all $x\in L$.  Thus, $\sigma(y)=y$ for all $y\in \phi[L]$, and $\phi[L]\subseteq \bar F^N$.

$(\subseteq)$ Take $\sigma\in \CG(\bar{F}/\dcl(\phi[L],F))$. 
Then, for all $x\in L$, $\phi^{-1}\circ \sigma\circ \phi(x)=\phi^{-1}\circ \phi(x)=x$, that is, $\sigma\in N$ and so $\CG(\bar{F}/\dcl(\phi[L],F))\subseteq N=\CG(\bar{F}/\bar{F}^N)$. By Galois Theory, $\bar F^N\subseteq \dcl(\phi[L],F)$.

$(2)$ It is clear that $\phi[E]\subseteq \phi[L]\cap F$. Suppose $\phi(x)\in (\phi[L]\cap F)\setminus \phi[E]$ for some $x\in L$. Then, $x\notin E$ and there is $\tau\in \gal(L/E)$ such that $\tau(x)\neq x$. 
Take $\sigma\in \gal(F)$ such that $\hat \phi(\sigma)=\tau$. Then $\sigma(\phi(x))=\phi(\tau(x))\neq \phi(x)$, and so $\phi(x)\notin F$, which is a contradiction. Thus, $\phi[L]\cap F\subseteq \phi[E]$.
\end{proof}

Then the next example shows that in Proposition \ref{prop:galois_gps_under_conjugation}, $F$ does not need to be a regular extension of $\phi[E]$ (what was present in the proof of \cite[Lemma 2.15]{chatzidakis2017}).

\begin{example}\label{extra_example}
Let $E:=\mathbf{Q}$. Let $F_0$ be any field elementary equivalent to $E$ and $F:=F_0(\sqrt{2})$. Consider a Galois extension $L:=\mathbf{Q}(\sqrt{3})$ of $E$. Let $\phi:\bar E\rightarrow \bar F$ be an arbitrary embedding. Then, the embedding $\phi$ induces an epimorphism $\hat \phi:\CG(F)\rightarrow \CG(L/E)$ but $F$ is not a regular extension of $E$. Indeed, $F\cap \bar E=\mathbf{Q}(\sqrt{2})$.
\end{example}

\noindent We first prove the following lemma generalizing \cite[Lemma 2.15]{chatzidakis2017}, which with Theorem \ref{thm:Zoe} plays a crucial role in the proof of Theorem \ref{thm:NSOP_n_nge3}. In the proof of Lemma \ref{lem:independenct_realization_sortedcompletesystem}, if we work with $H$ defined in the proof of \cite[Lemma 2.15]{chatzidakis2017}, we do not know whether $L$ is a regular extension of $F$ --- by Example \ref{extra_example}. The new choice of $H$ was suggested to us by Zoé Chatzidakis after communicating with her on the proof of \cite[Lemma 2.15]{chatzidakis2017}.

\begin{lemma}\label{lem:independenct_realization_sortedcompletesystem}\cite[Lemma 2.15]{chatzidakis2017}
Let $F^*$ be a monster PAC structure. Let $F\preceq F^*$ be a small elementary substructure. Let $E\subseteq A$ and $E_1$ be definably closed substructures of $F$ such that $E\subseteq A$ and $E_1\subseteq F$ both be regular extensions, and let $S\subseteq S\CG(F)$ be a sorted complete system (see Definition \ref{def:sorted complete system}). Suppose that
\begin{itemize}
	\item there is an isomorphism $\phi_0:\bar E\rightarrow \bar E_1$ such that $\phi_0\restriction E:E\rightarrow E_1$ is an elementary map in $F$, and
	\item there is a partial $\CL_{G}(\CJ)$-elementary isomorphism $S\Psi:S\CG(A)\rightarrow S$ extending the double dual $S\Phi_0 : S\CG(E)\rightarrow S\CG(E_1)$ of $\phi_0$ (i.e. $S\Phi_0=S\CG(\phi_0)$).
\end{itemize}
Then there is a subset $B\subseteq F^*$ and an isomorphism $\phi:\bar A\rightarrow \bar B$ sending $A$ to $B$, which extends $\phi_0$, such that
\begin{itemize}
	\item $B\ind_{E_1} F$, and
	\item $S\CG(B)=S$ and the double dual $S\Phi=S\CG(\phi)$ of $\phi$ is equal to $S\Psi$. 
\end{itemize}
Moreover, $(B,E_1)$ realizes $\tp_F(A,E)$.
\end{lemma}

\begin{proof}
Let $N_0$ be the kernel of the ``restriction" map $\CG(F)\cong GS\CG(F)\to G(S)$ and let $M=\bar F^{N_0}$. Take $A_1\subseteq \FC$ such that $\bar A_1 A_1\equiv_{\bar E_1} \phi_0[\bar A] \phi_0[A]$ and $A_1\ind_{E_1} F$. Let $\phi_1:\bar A\rightarrow \bar A_1$ be an elementary map sending $A$ to $A_1$ and extending $\phi_0$.

Then, we have that the double dual $S\Phi_1=S\CG(\phi_1)$ of $\phi_1$ extends $S\Phi_0$, and the dual $\Phi_1=\CG(\phi_1)$ of $\phi_1$ defines an isomorphism from $\CG(A_1)$ to $\CG(A)$ which induces the dual $\Phi_0=\CG(\phi_0)$ of $\phi_0$.
Also the dual $\Psi$ of $S\Psi$ defines an isomorphism from $\CG(M/F)$ to $\CG(A)$, which induces $\Phi_0$. So we have the following diagram:
$$
\begin{tikzcd}
\CG(A_1) \ar[r,"\Phi_1"] \ar[d,"\res"]& \CG(A) \ar[r, "\Psi^{-1}"] \ar[d, "\res"]& \CG(M/F) \ar[d, "\res"]\\
\CG(E_1) \ar[r, "\Phi_0"]& \CG(E) \ar[r, "\Phi_0^{-1}"]& \CG(E_1)
\end{tikzcd}
$$
Note that $\CG(\dcl(\bar A_1\bar F)/\dcl(A_1 F))\cong \CG(A_1)\times_{\CG(E_1)} \CG(F)$, via $\sigma\mapsto {({\sigma\restriction_{\bar A_1},} \sigma \restriction_{\bar F})}$, by \cite[Corollary 3.38]{Hoff3}.
Consider the following profinite group 
$$H=\{\big((\Phi_1^{-1}\circ \Psi)(\sigma\restriction_M), \sigma\big):\sigma\in \CG(F)\}$$ which can be identified with a closed subgroup of $\CG(\dcl(\bar A_1\bar F)/\dcl(A_1 F))$. Let $L=\dcl(\bar A_1 \bar F)^H$ and so $\CG(\dcl(\bar A_1 \bar F)/L)=H$. Note that $L$ contains $\dcl(A F)$. Since $H$ projects onto $\CG(A_1)$ and $\CG(F)$, which are isomorphisms, two restriction maps from $\CG(\dcl(\bar A_1\bar F)/L)$ to $\CG(A_1)$ and $\CG(F)$ are onto, and the restriction to $\CG(F)$ is an isomorphism. So we have that $L$ is a regular extension of $A_1$ and $F$.

Consider the following diagram $(*)$:
$$
\begin{tikzcd}
\CG(F^*) \ar[r, "\Phi_2"] \ar[d, "\res"']& \CG(\dcl(\bar A_1 \bar F)/L) \ar[d, "\res"]\\
\CG(F) \ar[ur, dashrightarrow, "\res^{-1}", "\cong"'] \ar[r, "\id"']& \CG(F)
\end{tikzcd}
$$
\noindent Applying the Embedding Lemma \cite[Lemma 3.5]{DHL1} to the diagram $(*)$, we have an automorphism $\phi_2\in \aut(\FC)$ such that
\begin{itemize}
	\item $\phi_2[L]\subseteq F^*$,
	\item $\phi_2\restriction_{\bar F}=\id_{\bar F}$,
	\item $\Phi_2:\CG(F^*)\rightarrow \CG(\dcl(\bar A_1 \bar F)/L), \sigma\mapsto \phi_2^{-1} \circ \sigma\restriction_{\phi_2[\dcl(\bar A_1 \bar F)]}\circ (\phi_2\restriction_{\dcl(\bar A_1 \bar F)})$.
\end{itemize}
If there is no confusion, we write $\phi_2$ for $\phi_2\restriction_{\dcl(\bar A_1 \bar F)}$. Moreover, by Proposition \ref{prop:galois_gps_under_conjugation} (if we set $F^*$, $L$ and $\dcl(\bar{A}_1\bar F)$ in the place of structures $F$, $E$ and $L$ from Proposition \ref{prop:galois_gps_under_conjugation} respectively), 
for $N:=\ker \Phi_2$ we have that
\begin{align*}
\bar F^{*N}&= \dcl(F^* \phi_2[\dcl(\bar A_1 \bar F)])\\
&= \dcl(F^*\bar F \phi_2[\bar A_1])\\
&=\dcl(F^* \bar F).
\end{align*}

Set $B:=\phi_2[A_1]\subseteq \bar F^*$ and $\phi:=\phi_2\circ \phi_1$. It is clear that $\phi$ extends $\phi_0$ and $B\ind_{E_1} F$. It remains to show that $S\CG(B)=S$ and the double dual $S\Phi$ of $\phi$ is equal to $S\Psi$.
\
\\
\textbf{Claim}
$\dcl(F^* M)=\dcl(F^* B)$.
\
\\ Proof of the claim: 
Take $\sigma \in \CG(\bar F^*/\dcl(F^* M))$. Then, we have that
\begin{align*}
\Phi_2(\sigma)&=\res^{-1}(\sigma\restriction_{\bar F})\\
&=((\Phi_1^{-1}\circ \Psi)(\sigma\restriction_{M}),\sigma\restriction_{\bar F})\\
&=((\Phi_1^{-1}\circ \Psi)(\id_M),\sigma\restriction_{\bar F})\\
&=(\id_{\bar A_1}, \sigma\restriction_{\bar F})(\in H).
\end{align*}

\noindent Take $b\in B$ and $a\in A_1$ with $\phi_2(a)=b$. Thus, we have that
\begin{align*}
\sigma(b)&=\sigma\circ \phi_2(a)\\
&=\phi_2\circ (\phi_2^{-1}\circ \sigma\circ \phi_2)(a)\\
&=\phi_2\circ \Phi_2(\sigma)(a)\\
&=\phi_2(a)=b,
\end{align*}  
and so $B\subseteq \dcl(F^* M)$. Take $\tau\in \CG(\bar F^*/\dcl(F^* B))$. Then, by the similar way, we have that $\Phi_2(\tau)=(\id_{\bar A_1},\tau\restriction_{\bar F})\in H$. Since $\id_{\bar A_1}=(\Phi_1^{-1}\circ \Psi)(\tau_M))$, $\tau_M=\id_M$, which implies that $M\subseteq \dcl(F^* B)$.

\
\\
\textbf{Claim}
$S\CG(B)=S$ and $S\Phi=S\Psi$.
\
\\ Proof of the claim: 
Consider the following diagram $(\dagger)$:
$$
\begin{tikzcd}
\CG(\dcl(F^* B)/F^*) \ar[rrr, "\Theta"] \ar[d,"\res"]& & & \CG(\dcl( F^* M)/F^*)(=\CG(\dcl(F^* B)/F^*)) \ar[d, "\res"]\\
\CG(B) \ar[r, "\Phi_2"]& \CG(A_1) \ar[r, "\Phi_1"]& \CG(A) \ar[r, "\Psi^{-1}"]& \CG(M/F)
\end{tikzcd}
$$
where $\Theta=\res^{-1}\circ\Psi^{-1}\circ \Phi_1\circ \Phi_2\circ \res$.
We will show that $\Theta = \id$.

Let $\sigma,\tau \in \CG(F^*)$ be such that $\Theta(\sigma\restriction_{\dcl(F^* M)})=\tau\restriction_{\dcl(F^* M)}$. Then, we have $\Psi(\tau\restriction_M)=\Phi_1(\Phi_2(\sigma \restriction_{\bar B}))$ and so $(\Phi_2(\sigma\restriction_{\bar B}), \tau\restriction_{\bar F})\in H$. Note that $$\CG(\dcl(F^* \bar F/F^*)\cong \CG(\dcl(\bar A_1 \bar F)/L):\sigma_0\mapsto \res^{-1}(\sigma_0\restriction_{\bar F}).$$ Take $\gamma \in \CG(\dcl(F^* \bar F)/F^*)$ such that $\res^{-1}(\gamma\restriction_{\bar F})=(\Phi_2(\sigma\restriction_{\bar B}), \tau\restriction_{\bar F})$. It follows $\res^{-1}(\gamma\restriction_{\bar F})=\phi_2^{-1}\circ \gamma\restriction_{\phi_2[\dcl(\bar A_1 \bar F)]}\circ \phi_2$ and so
\begin{align*}
(\phi_2^{-1}\circ \gamma\restriction_{\phi_2[\dcl(\bar A_1 \bar F)]}\circ \phi_2)\restriction_{\bar A_1}&=(\phi_2^{-1}\circ \gamma\restriction_{\phi_2[\bar A_1]}\circ \phi_2)\restriction_{\bar A_1}\\
&=\Phi_2(\gamma\restriction_{\bar B})
\end{align*}
hence $\Phi_2(\sigma\restriction_{\bar B})=\Phi_2(\gamma\restriction_{\bar B})$. Thus, we have $\sigma\restriction_{\bar B}=\gamma\restriction_{\bar B}$ and $\sigma\restriction_{\dcl(F^* B)}=\gamma\restriction_{\dcl(F^* B)}$. Moreover,
\begin{align*}
(\phi_2^{-1}\circ \gamma\restriction_{\phi_2[\dcl(\bar A_1 \bar F)]}\circ \phi_2)\restriction_{\bar F}&=(\phi_2^{-1}\circ \gamma\restriction_{\bar F}\circ \phi_2)\restriction_{\bar F}\\
&=\gamma\restriction_{\bar F}
\end{align*}
and so $\tau\restriction_{\bar F}=\gamma\restriction_{\bar F}$. Thus, we have that $\tau\restriction_{M}=\gamma \restriction_{M}$ and $\tau\restriction_{\dcl(F^*M)}=\gamma\restriction_{\dcl(F^*M)}$.
Therefore, $\sigma\restriction_{\dcl(F^*M)}=\gamma\restriction_{\dcl(F^*M)}=\tau\restriction_{\dcl(F^*M)}$ and $\Theta$ is the identity map.

Since $\Theta$ is the identity map, the dual $S\Theta : S\CG(\dcl(F^*B)/F^*)\rightarrow S\CG(\dcl(F^*B)/F^*)$ is the identity hence $S\Phi_2\circ S\Phi_1\circ S\Psi^{-1}:S\rightarrow S\CG(B)$ is also the identity map. Therefore, we conclude that $S\CG(B)=S$ and $S\Psi=(S\Phi_2\circ S\Phi_1)=S\Phi$. 
Here ends the proof of the claim.

The moreover part comes from Theorem \ref{thm:PAC_type_by_completesystem} using the map $\phi$.
\end{proof}

\noindent Now, we provide criteria for NSOP$_n$ of PAC structures for $n> 2$.

\begin{theorem}\label{thm:NSOP_n_nge3}\cite[Theorem 3.9]{chatzidakis2017}
Let $F$ be a PAC structure. Suppose $\theo(S\CG(F))$ has NSOP$_n$ for $n>2$. Then $\theo(F)$ has NSOP$_n$.
\end{theorem}

\begin{proof}
Let $(M,F)^{Sk}$ be the Skolemization of $(M,F)$ in the expanded language $\CL_P^{Sk}$. Let $(M^*,F^*)^{Sk}$ be a weakly special model of $\theo((M,F)^{Sk})$ (i.e. $(M^*,F^*)^{Sk}$ is the Skolemization of $(M^*,F^*)$ in the language $\CL_P^{Sk}$ and $(M^*,F^*)$ is weakly special in $\CL_P$). So, we have that for any $A\subseteq M^*$, 
\begin{itemize}
	\item $\acl_{\CL_P^{Sk}}(A)=\dcl_{\CL_P^{Sk}}(A)$ is an elementary substructure, and
	\item $\bar A\subseteq \acl_{\CL_P^{Sk}}(A)$.
\end{itemize}
Let $\theta(x,y)$ be an $\CL$-formula with $|x|=|y|$.  Suppose that there is an infinite sequence $a_i\in F^*$, $i\in \omega$ such that $F^*\models \bigwedge_{i<j\in \omega}\theta(a_i,a_j)$. By Ramsey and compactness, we may assume that there is an indiscernible sequence $a_i\in F^*$, $i\in \omega+\BZ$, such that $F^*\models \bigwedge_{i<j\in \omega+\BZ}\theta(a_i,a_j)$. To show that $\theo(F^*)$ has NSOP$_n$, we need to show that 
$$F^*\models \exists x_0,\ldots,x_{n-1}\left( \bigwedge \theta(x_i,x_{i+1})\wedge \theta(x_{n-1},x_0)\right ).$$
Let $\theta_P(\x,\y)\equiv \bigwedge P(x_j)\wedge \bigwedge P(y_j)\wedge \theta(\x,\y)$. 
We have that $(M^*,F^*)\models \theta_P(a_i,a_{i+1})$ holds for each $i\in\omega$.
Put $E_0:=(a_i)_{\omega\le i<\BZ}$. Then, for each $i\in\omega$, $\tp_{M^*}(a_i/a_{<i}E_0)$ is finitely satifiable in $E_0$ and so $\{a_i:i\in\omega\}$ is a Morely sequence over $E_0$.  Put $E:=\acl_{F^*}^r(E_0)$ $(=\bar E_0\cap F^*)$. We have that $(a_j)_{j>\alpha}$ is $E_0$-indiscernible in $\CL_P^{Sk}$ and $(\ov{Ea_j})_{j>\alpha}$ is $\bar E$-indiscernible in $\CL_P^{Sk}$ after fixing an enumeration of $\ov{Ea_j}$.

So, we may assume that there is an infinite sequence $(a_i)_{i\in\omega}\subseteq F^*$ and $E=\acl_{F^*}^r(E)\subseteq F^*$ such that
\begin{itemize}
	\item $F^*\models\theta(a_i,a_{i+1})$ for all $i\in\omega$ or, equivalently, $(M^*,F^*)\models \theta_P(a_i,a_{i+1})$ for all $i\in\omega$,
	\item $(\ov{Ea_i})_{i\in \omega}$ is $\bar E$-indiscernible in $\CL_P^{Sk}$, and
	\item $\{a_i|\ i\in \omega \}$ is independent over $\bar E$ in $M^*$.
\end{itemize}

\noindent Fix an automorphism $\phi_1\in\aut(M^*/\bar E)$ such that $\phi_1[A_0]=A_1$, and let $S\Phi_1:S\CG(A_0)\rightarrow S\CG(A_1)$ be the double dual of $\phi_1\restriction_{\bar A_0}$. For each $i\in \omega$, put $A_i:=\acl_{F^*}^r(Ea_i)$ and $K_{i,i+1}:=\acl_{F^*}^r(A_i,A_{i+1})$. Then, $(\bar A_i)$ is $\bar E$-indiscernible in $\CL_P^{Sk}$, and so for $q:=\tp_{\CL_P}(\bar K_{0,1},\bar A_0,\bar A_{1}/\bar E)$ we have $$(M^*,F^*)\models \bigwedge_{i\in \omega}q(\bar K_{i,i+1},\bar A_i,\bar A_{i+1}).$$

\noindent For each $i\in\omega$, we have $$(S\CG(K_{i,i+1}),S\CG(A_i),S\CG(A_{i+1}))\equiv_{S\CG(E)}(S\CG(K_{0,1}),S\CG(A_0),S\CG(A_1))$$ and $$S\CG(F^*)\models \bigwedge_{i\in\omega}p(S\CG(A_i),S\CG(A_{i+1}))$$
where $p:=\tp_{\CL_G(\CJ)}(S\CG(A_0),S\CG(A_1)/S\CG(E))$.

Since $\theo(S\CG(F^*))$ has NSOP$_n$, there are $S_0,\ldots,S_{n-1}\subseteq S\CG(F^*)$ such that $(S_i,S_{i+1})$ and $(S_{n-1},S_0)$ realize $p$. Take $S_{i,i+1}$ for $i<n-1$ and $S_{n-1,0}$ such that 
$$(S_{i,j},S_i,S_j)\equiv_{S\CG(E)}(S\CG(K_{0,1}),S\CG(A_0),S\CG(A_1))$$
for $(i,j)\in I:=\{(0,1),(1,2),\ldots, (n-2,n-1), (n-1,0)\}$.
Thus for each $(i,j)\in I$ we have an $\CL_G(\CJ)(S\CG(E))$-elementary isomorphism $S\Psi_{i,j}:S\CG(K_{0,1})\rightarrow S_{i,j}$ such that
\begin{itemize}
	\item $S\Psi_{i,j}[S\CG(A_0)]=S_i$ and $S\Psi_{i,j}[S\CG(A_1)]=S_j$, and
	\item for $i<n-1$, $$S\Psi_{i+1,i+2}\restriction_{S\CG(A_0)}=S\Psi_{i,i+1}\circ S\Phi_1,$$ and $$S\Psi_{0,1}\restriction_{S\CG(A_0)}=S\Psi_{n-1,0}\circ S\Phi_1.$$
\end{itemize}
Moreover, for each $i<n-1$ we have the following diagrams
$$
\begin{tikzcd}
S\CG(A_0)\ar[r,"S\Phi_1"] \ar[dr, "S\Psi_{i+1,i+2}"']& S\CG(A_1) \ar[d, "S\Psi_{i,i+1}"]& S\CG(A_0)\ar[r,"S\Phi_1"] \ar[dr, "S\Psi_{0,1}"']& S\CG(A_1) \ar[d, "S\Psi_{n-1,0}"]\\
& S_{i+1}& &S_0
\end{tikzcd}
$$
\noindent Without loss of generality, we may assume that $(S_{0,1},S_0,S_1)=(S\CG(K_{0,1}),S\CG(A_0),S\CG(A_1))$ and $S\Psi_{0,1}=\id_{S\CG(K_{0,1})}$ is induced from the $\CL(\bar E)$-isomorphism $\psi_{0,1}=\id_{\ov{A_0A_1}}:\ov{A_0A_1}\rightarrow \ov{A_0A_1}$.

\
\\
\textbf{Claim}
There exist a sequence $B_0,\ldots, B_n \subseteq F^*$ and a sequence of $\CL(\bar E)$-isomorphisms $\psi_{i,i+1}:\bar K_{0,1}\rightarrow \ov{B_iB_{i+1}}$, $i<n$, extending the $\CL(E)$-elementary map from $K_{0,1}$ to $\acl_{F^*}^r(B_i,B_{i+1})$ in $F^*$, such that
\begin{itemize}
	\item $\{B_0,\ldots,B_n\}$ is $\ind$-independent over $E$ in $M^*$;,
	\item $B_i B_{i+1}\equiv_E A_0 A_1$ in $F^*$,
	\item $S\CG\big(\acl_{F^*}^r(B_i,B_{i+1})\big)=\begin{cases}
	S_{i,i+1}\mbox{ if } i<n-1\\
	S_{n-1,0}\mbox{ if }i=n-1
	\end{cases}$, and
	\item the double dual of $\psi_{i,i+1}$ is equal to $S\Psi_{i,i+1}'$ where $$S\Psi_{i,i+1}'=\begin{cases}
	S\Psi_{i,i+1}\mbox{ if } i<n-1\\
	S\Psi_{n-1,0}\mbox{ if } i=n-1
	\end{cases}.$$
\end{itemize}
\
\\ Proof of the claim: 
We recursively construct such a sequence.
Put $B_0=A_0$, $B_1=A_1$, and $\psi_{0,1}=\id_{\bar K_{0,1}}$. Suppose we have  $B_0,\ldots, B_i\subseteq F^*$ and an $\CL(\bar E)$-isomorphism $\psi_{j,j+1}:\bar K_{0,1}\rightarrow\ov{B_jB_{j+1}}$ extending the $\CL(E)$-elementary map from $K_{0,1}$ to $\acl_{F^*}^r(B_jB_{j+1})$ in $F^*$ for some  $i<n$ and for each $j<i$ such that
\begin{itemize}
	\item $\{B_0,\ldots,B_i\}$ is $\ind$-independent over $E$ in $M^*$,
	\item $B_j B_{j+1}\equiv_E A_0 A_1$ in $F^*$,
	\item $S\CG\big(\acl_{F^*}^r(B_j,B_{j+1})\big)=S_{j,j+1}$ for $j<i$,
	\item $\psi_{j+1,j+2}\restriction_{\bar A_0}=\psi_{j,j+1}\circ \phi_1\restriction_{\bar A_0}$ for $j<i-1$, and
	\item the double dual of $\psi_{j,j+1}$ is equal to $S\Psi_{j,j+1}'$ for $j<i$.
\end{itemize}
Let $F_i$ be a small elementary substructure of $F^*$ containing $B_0=A_0, B_1=A_1,\ldots,B_i$. We apply Lemma \ref{lem:independenct_realization_sortedcompletesystem} to the case $F:=F_i$, $E:=A_0\subseteq A:=K_{0,1}$, $E_1:=B_i$, $S\Psi:=S\Psi_{i,i+1}'$, and $\phi_0:=\psi_{i-1,i}\circ \phi_1:\bar A_0\rightarrow \bar B_i$, and so, $S\Phi_0:=S\Psi_{i-1,i}\circ \Phi_1:S\CG(A_0)\rightarrow S\CG(B_i)$.
It follows that there are $B_{i+1}\subseteq F^*$ and $\CL$-isomorphisms $\psi_{i,i+1}:\bar K_{0,1}\rightarrow \ov{B_iB_{i+1}}$ extending $\psi_{i-1,i}\circ \phi_1$ such that
\begin{itemize}
	\item $\psi_{i,i+1}\restriction_{K_{0,1}}:K_{0,1}\rightarrow \acl_{F^*}^r(B_iB_{i+1})$ is an $\CL$-elementary map in $F^*$ and $K_{0,1}A_0\equiv \acl_{F^*}^r(B_i B_{i+1})B_i$,
	\item the double dual of $\psi_{i,i+1}$ is equal to $S\Phi_{i,i+1}$, and
	\item $\acl_{F^*}^r(B_i B_{i+1})\ind_{B_i}F_i$.
\end{itemize}
Since $B_iB_{i+1}\equiv A_0A_1$ in $F^*$, we have that $B_{i+1}\ind_{E} B_i$. So, by transitivity, $B_{i+1}\ind_{E} F_i$ and $B_{i+1}\ind_E B_{\le i}$, hence $\{B_0,\ldots,B_i,B_{i+1}\}$ is $\ind$-independent over $E$. 
Here ends the proof of the claim.

We have that $S\Psi_{n-1,0}\circ S\Phi_1=S\Psi_{0,1}\restriction_{S\CG(A_0)}=\id_{S\CG(A_0)}$, thus $S\CG(B_n)=S\CG(A_0)=S_0$. We apply Proposition \ref{thm:Zoe} to the case $E:=E$, $A:=B_1$, $B:=\acl_{F^*}^r(B_2,\ldots,B_{n-1})$, $C_1:=B_0$, $C_2:=B_n$, $\phi:=\psi_{n-1,0}\circ \phi_1$, and $S:=S_0$ 
to obtain
$B_0'\subseteq F^*$ which realizes the type $\tp_{F^*}(B_0/B_1)\cup\tp_{F^*}(B_n/\acl_{F^*}^r(B_2,\ldots,B_{n-1})$. 
Therefore the tuple $(B_0',B_1,\ldots,B_{n-1})$ witnesses $$F^*\models \exists x_0,\ldots,x_{n-1}\left( \bigwedge \theta(x_i,x_{i+1})\wedge \theta(x_{n-1},x_0)\right ).$$
\end{proof}

\noindent
Now, we provide a NSOP$_2$ criterion for PAC structures.

\begin{theorem}\label{thm:NOSP2}\cite[Proposition 7.2.8]{NickThesis}
Let $F$ be a PAC structure. 
Then $\theo(F)$ is NSOP$_2$ provided $\theo(S\CG(F))$ is NSOP$_2$.
\end{theorem}

\begin{proof}
We repeat here some parts of the proof of \cite[Proposition 7.2.8]{NickThesis} but using results generalized for PAC structures.

Let $(M^*,F^*)^{Sk}$ be a weakly special model of $\theo((M,F)^{Sk})$. And $(M^*,F^*)^{Sk}$ is the Skolemization of $(M^*,F^*)$ in the language $\CL_P^{Sk}$ and $(M^*,F^*)$ is weakly special in $\CL_P$.
Suppose that $F^*$ has $SOP_2$ witnessed by an $\CL$-formula $\theta(x;y)$ 
and thus $(M^*,F^*)^{Sk}$ has $SOP_2$ witnessed by the $\CL_P^{Sk}$-formula $\theta_P(x;y)\equiv \theta(x;y)\wedge P(x)\wedge P(y)$. 
By compactness, there is a strongly indiscernible tree $(b_{\eta})_{\eta\in \omega^{<\omega+\omega}}$ witnessing $SOP_2$ for the formula $\theta_P$ in the language $\CL_P^{Sk}$. Take $a\in F^*$ such that $\models \theta(a;b_{0^i})$ for all $i<\omega+\omega$. By Ramsey, compactness, and after shifting by an automorphism, we may assume that $(b_{0^i})_{i<\omega+\omega}$ is $a$-indiscernible in $\CL_P^{Sk}$.

Set $E:=\acl_{\CL_P^{Sk}}((b_{0^j})_{j<\omega})\cap F^*$ and note that $E$ is an elementary substructure of $F^*$. 
Let $A:=\acl_{F^*}^r(Ea)$ and let $B_{\eta}:=\acl_{F^*}^r(Eb_{0^{\omega}\frown\eta})$, where $\eta\in 2^{<\omega}$. Observe that in $\CL_P^{Sk}$ we have that
\begin{itemize}
	\item $(\bar B_{\eta})_{\eta\in 2^{<\omega}}$ is strongly indiscernible over $\bar E$,
	\item $(\bar B_{0^i})_{i<\omega}$ is $\bar A$-indiscernible, and
	\item for each $i_1>i_2>\cdots>i_n$, $\tp_{\CL_P^{Sk}}(B_{0^{i_1}}/E B_{0^{i_2}}\ldots B_{0^{i_n}})$ is finitely satisfiable in $E$ so $(B_{0^i})_{i<\omega}$ is an $E$-finitely satisfiable Morley sequence in $\CL$, enumerated in reverse.
\end{itemize}
By Kim's lemma for stable theories, it follows $A\ind_E B_0$. 
By strong indiscernibility, we have that $(B_{0^i})_{i<\omega}$ is also $B_{<1>}$-indiscernible. By Kim's lemma again, we obtain that $B_{<1>}\ind_E B_0$ and thus $B_0\ind_E B_{<1>}$. Choose $A'$ such that $\bar A'\bar B_{<1>}\equiv_E^{\CL_P^{Sk}} \bar A\bar B_0$.

Let $q(X;S\CG(B_{\emptyset}):=\tp(S\CG(A)/S\CG(B_{\emptyset}))$. Since $(\bar B_{\eta})_{\eta\in 2^{<\omega}}$ is strongly indiscernible over $\bar E$ and $(\bar B_{0^i})_{i<\omega}$ is $\bar A$-indiscernible, we have that
\begin{itemize}
	\item $(S\CG(B_{\eta}))_{\eta\in2^{<\omega}}$ is a strongly indiscernible tree over $S\CG(E)$, and
	\item $(S\CG(B_{0^i}))_{i<\omega}$ is $S\CG(A)$-indiscernible and hence $$S\CG(A)\models \bigcup_{i<\omega}q(X;S\CG(B_{0^i})).$$
\end{itemize}
\noindent As $\theo(S\CG(F^*))$ is NSOP$_2$, there is a realization $S_0\subseteq S\CG(F^*)$ of $q(X;B_0)\cup q(X;B_{<1>})$. We apply Proposition \ref{thm:Zoe} to the case $E:=E$, $A:=B_0$, $B:=B_{<1>}$, $C_1:=A$, $C_2:=A'$, $S:=S_0$, and for $\phi$ being an automorphism of $M^*$ given by the fact that $\bar A'\bar B_{<1>}\equiv_E^{\CL_P^{Sk}} \bar A\bar B_0$.
We get $A''$ such that $A''\equiv_{B_0}^{\CL} A$ and $A''\equiv_{B_{<1>}}A'$. 

Therefore $\{\theta(x;b_{0^{\omega}\frown0}),\theta(x;b_{0^{\omega}\frown<1>})\}$ is consistent, witnessed by $A''$, which contradicts the definition of $SOP_2$.
\end{proof}

\section{Algebraic closure in PAC structures}\label{sec:acl}
The following section is independent from the previous ones and basically generalizes well known facts about the algebraic closure in PAC fields (from \cite{ChaPil}).
In this section, we assume that $T$ is $\omega$-stable. This assumption is used in the proof of Lemma \ref{lemma:reg_to_PAC}, where we need to know that
an algebraic extension of a PAC structure is a PAC structure (in the case of fields such fact is known as the \emph{Ax-Roquette theorem}). Actually, it is enough to have only finitely many nonforking extensions for a given type, see assumptions of \cite[Lemma 4.5]{Hoff4}.


\begin{lemma}\label{lemma:reg_to_PAC}\cite[Lemma 4.4]{ChaPil}
Assume that $T$ is $\omega$-stable,
$F$ is PAC in $\mathfrak{C}$ and $F\subseteq E$ is regular. Then there exists $F'$ such that $E\subseteq F'$ and the restriction map $\res:\CG(F')\to\CG(F)$ is an isomorphism.
\end{lemma}

\begin{proof}
We copy here part of the proof of \cite[Lemma 4.4]{ChaPil}, but using results generalized for PAC structures.

By \cite[Proposition 3.6]{Hoff3}, there exists $E_1\subseteq\mathfrak{C}$ such that $E\subseteq E_1$ is regular and $E_1$ is PAC. By \cite[Theorem 4.4]{Hoff4}, the group $\CG(F)$ is projective.
Since $F\subseteq E_1$ is regular, $\res:\CG(E_1)\to\CG(F)$ is onto. Therefore there exists an embedding $i:\CG(F)\to\CG(E_1)$ such that the following diagram commutes
$$\xymatrix{
\CG(F) \ar[r]^{\id} \ar[dr]_{i}& \CG(F) \\
& \CG(E_1) \ar[u]_{\res}
}$$
Consider $F':=\acl(E_1)^{i[\CG(F)]}$, which is PAC in  $\mathfrak{C}$ by $\omega$-stability and \cite[Proposition 3.9]{PilPol} (or more precisely: by \cite[Lemma 4.5]{Hoff4}). By the Galois correspondence and the commutativity of the above diagram, we see that $\res:\CG(F')\to\CG(F)$ is an isomorphism.
\end{proof}

\begin{lemma}\label{lemma:reg_to_sat_PAC}
Assume that $T$ is $\omega$-stable,
$F$ is PAC in $\mathfrak{C}$, $F\subseteq E$ is regular and $\kappa$ is some cardinal. Then there exists $F^*$ such that $E\subseteq F^*$, $F^*$ is $\kappa$-PAC and the restriction map $\res:\CG(F^*)\to\CG(F)$ is an isomorphism.
\end{lemma}

\begin{proof}
For a small substructure $N$ of $\mathfrak{C}$, let us define 
\begin{IEEEeqnarray*}{rCl}
\ST(N,\kappa,\lambda) &:=& \{\qftp(\bar{d}/A)\;\;|\;\;A\subseteq N,\;|A|<\kappa, \\
 & & N\subseteq\dcl(N,\bar{d})\text{ is regular and }|\bar{d}|<\lambda\}.
\end{IEEEeqnarray*}
For a suitably big cardinal $\lambda$, we will recursively construct a tower $(F_{\alpha})_{\alpha<\lambda}$ of substructures of $\mathfrak{C}$ such that
\begin{itemize}
\item $F_0=F$, $E\subseteq F_1$,
\item $F_{\alpha}\subseteq F_{\beta}$ and $\res:\CG(F_{\beta})\to\CG(F_{\alpha})$ is an isomorphism for all $\alpha\leqslant\beta$,
\item each $F_{\alpha}$ is PAC,
\item $F_{\alpha+1}$ realizes each element of $\ST(F_{\alpha},\kappa,\kappa)$, where $\alpha\geqslant 1$.
\end{itemize}
Of course we put $F_0:=F$. Let $F_1$ be the PAC structure given by Lemma \ref{lemma:reg_to_PAC} for $F$ and $E$.

Successor case for $\alpha\geqslant 1$. Assume that we defined $F_{\alpha}$ and we want to show the existence of a proper $F_{\alpha+1}$, where $\alpha\geqslant 1$. Let $X$ be a set containing exactly one realization of each element from $\ST(F_{\alpha},\kappa,\kappa)$. Without loss of generality, we assume that $X$ is $F_{\alpha}$-independent in $\mathfrak{C}$. By \cite[Lemma 3.40]{Hoff3}, we see that $F_{\alpha}\subseteq\dcl(F_{\alpha},\bigcup X)$ is regular. Hence, by Lemma \ref{lemma:reg_to_PAC} (used for ``$F=F_\alpha$" and ``$E=\dcl(F_{\alpha},\bigcup X)$"), we find an appropriate $F_{\alpha+1}$.

Limit case. Now, assume that we defined $F_{\alpha}$ for all $\alpha$ strictly smaller than some $\beta$. If $\beta$ is a limit cardinal, then set $F_{\beta}:=\bigcup_{\alpha<\beta}F_{\alpha}$. To see that $F_{\beta}$ is PAC, we suppose that there is a regular extension $F_{\beta}\subseteq N$ such that $N\models\exists x\,\psi(a,x)$ for some $a\in F_{\beta}$. Naturally $a\in F_{\alpha}$ for some $\alpha<\beta$.
By \cite[Lemma 3.5]{Hoff3}, $F_{\alpha}\subseteq N$ is regular, so $F_{\alpha}\models \exists x\,\psi(a,x)$ and so $\psi(a,x)$ is realized in $F_{\beta}$. We need to show that the restriction map $\res:\CG(F_{\beta})\to\CG(F_{\alpha})$ is an isomorphism. It follows, since for all $\beta>\alpha'\geqslant\alpha$ we have that $\res:\CG(F_{\alpha'})\to\CG(F_{\alpha})$ is an isomorphism and 
$$\acl(F_{\beta})=\bigcup\limits_{\alpha'<\beta}\acl(F_{\alpha'}).$$

Assume that we have our tower of structures $F_{\alpha}$ for all $\alpha<\lambda$. We put
$$F^*:=\bigcup\limits_{\alpha<\lambda}F_{\alpha}.$$
As in the proof of the limit case, we can show that $F^*$ is PAC and that $\res:\CG(F^*)\to\CG(F)$ is an isomorphism. Obviously $E\subseteq F^*$. It is left to show that $F^*$ is $\kappa$-PAC.

Suppose that $A\subseteq F^*$ is such that $|A|<\kappa$, $|\bar{x}|<\kappa$
and $p(\bar{x})$ is a complete stationary type over $A$ (in the sense of $\FC$).
We need to find a realization of $p$ inside $F^\ast$.
Consider $\bar{d}\subseteq\FC$ such that $p(\bar{x})\subseteq\tp(\bar{d}/F^\ast)$ is the unique non-forking extension.
Since $\tp(\bar{d}/A)$ is stationary, also $\tp(\bar{d}/F^\ast)$ is stationary which means that
$F^*\subseteq\dcl(F^*,\bar{d})$ is regular. 
Since $|A|<\kappa$, there exists $\alpha<\lambda$ such that $A\subseteq F_{\alpha}$. By \cite[Lemma 3.5]{Hoff3}, $F_{\alpha}\subseteq F^*$ is regular, hence $F_{\alpha}\subseteq\dcl(F^*,\bar{d})$ is regular. In particular, $F_{\alpha}\subseteq\dcl(F_{\alpha},\bar{d})$ is regular. This means that $\qftp(\bar{d}/A)\in\ST(F_{\alpha},\kappa,\kappa)$ and so there is a realization of $\qftp(\bar{d}/A)$ in $F_{\alpha+1}$. Because the last type is quantifier-free, it is also realized in $F^*$, which ends the proof, since $T$ has quantifier elimination.
\end{proof}

\begin{lemma}\label{lemma:no_regular_substructures_in_PAC}\cite[Proposition 4.5]{ChaPil}
Assume that $T$ is $\omega$-stable.
Let $F$ be a $\kappa$-PAC substructure in $\mathfrak{C}$ (for some $\kappa\geqslant|T|^+$). If $E\subseteq F$ is regular such that $|E|^+\leqslant\kappa$, then $\acl_F(E)=\dcl(E)$.
\end{lemma}

\begin{proof}
The proof is based on the proof of \cite[Proposition 4.5]{ChaPil}, but there are significant changes related to the saturation assumption in Corollary \ref{cor2033} (hidden in the notion of a $\kappa$-PAC substructure). Without loss of generality, let us assume that $E=\dcl(E)$.

We will show that $\acl_F(E)\subseteq E$. Suppose not, so there exists $m\in \acl_F(E)\setminus E$.
Consider $\theta(x)\vdash\tp_F(m/E)$. 
Let $k$ be the number of all realizations of $\tp_F(m/E)$ in $F$, say $\{m_1,\ldots,m_k\}$ are all the realizations of this type.
One has that $F\models\exists^{=k}x\,\theta(x)$. 

Let $\phi\in\aut(\mathfrak{C}/\acl(E))$ be such that $F\ind_{E}F'$, where $F':=\phi[F]$.
Let $m':=\phi(m)\in F'\setminus E$, so $m'\not\in F$ (otherwise $m'\in F\cap\acl(E)=E$). Note that $F\models \theta(m_1)\wedge\ldots\wedge\theta(m_k)$ and $F'\models \theta(m')$.

We see that $\Phi(:=\CG(\phi)):\CG(F)\to\CG(F')$ is 
an isomorphism of profinite groups.
By \cite[Corollary 3.38]{Hoff3}, for each $\sigma\in\CG(F)$ there exists $\tilde{\sigma}\in\aut(\mathfrak{C})$ such that $\tilde{\sigma}\restriction_{\acl(F)}=\sigma$ and $\tilde{\sigma}\restriction_{\acl(F')}=\Phi(\sigma)$.
Note that
$$i:\CG(F)\to\aut\Big(\dcl\big(\acl(F),\acl(F')\big)/\dcl(F,F') \Big)$$
given by $i(\sigma):=\tilde{\sigma}|_{\dcl(\acl(F),\acl(F'))}$ is an embedding. 
We have the following commuting diagram
$$\xymatrix{
& i[\CG(F)] \ar[dr]^{\res} \ar@/^0.5pc/[dl]^{\res} & \\
\CG(F)\ar@/^1pc/[ur]^{i} \ar[rr]_{\Phi} & & \CG(F')
}$$

We define
$$D:=\dcl\big(\acl(F),\acl(F')\big)^{i[\CG(F)]}.$$
One has that $\dcl(F,F')\subseteq D$.
By the Galois correspondence
$$i[\CG(F)]=\aut\Big(\dcl\big(\acl(F),\acl(F')\big)/D\Big).$$
Note that $F\subseteq D$ is regular. To see this, assume that $m\in D\cap\acl(F)$. If $m\not\in F$, then there is $\sigma\in\CG(F)$ such that $\sigma(m)\neq m$. But since $m\in D$, we have that $m=i(\sigma)(m)=\sigma(m)$, a contradiction. 

By Lemma \ref{lemma:reg_to_sat_PAC} for $F\subseteq D$, there exists a $\kappa$-PAC substructure $F^*$ such that $D\subseteq F^*$ and $\res:\CG(F^*)\to\CG(F)$ is an isomorphism. We have the following diagram
$$\xymatrix{
& \CG(F^*) \ar[d]^{\res} \ar@/^1pc/[ddr]^{\res} \ar@/_1pc/[ddl]_{\res}^{\cong}& \\
& i[\CG(F)] \ar[dl]_{\res}^{\cong} \ar[dr]^{\res} & \\
\CG(F) \ar[rr]^{\Phi}_{\cong} & & \CG(F')
}$$
so every arrow in it is an isomorphism.

Now we use Corollary \ref{cor2033} for the following situations:
$$\xymatrix{
\CG(F^*) \ar[rr]^{\res} \ar[dr]_{\res}& & \CG(F) \ar[dl]^{\res}\\
& \CG\big(\dcl(E,m_1,\ldots,m_k)\big) &
}$$
$$\xymatrix{
\CG(F^*) \ar[rr]^{\res} \ar[dr]_{\res}& & \CG(F') \ar[dl]^{\res}\\
& \CG\big(\dcl(E,m')\big) &
}$$
to state that $F^*\equiv_{E,m_1,\ldots,m_k} F$ and $F^*\equiv_{E,m'} F'$. 
Therefore $F^*\models\theta(m_1)\wedge\ldots\wedge\theta(m_k)$, 
$F^*\models\theta(m')$ 
and $F^*\models\exists^{=k}x\,\theta(x)$, a contradiction.
\end{proof}

\begin{prop}[Algebraic closure description]\label{prop:acl_F}
Assume that $T$ is $\omega$-stable.
Let $F$ be a $\kappa$-PAC substructure in $\mathfrak{C}$, where $\kappa\geqslant|T|^+$. Suppose that $A\subseteq F$ is such that $|A|^+\leqslant\kappa$. 
Then $\acl_F(A)=\acl(A)\cap F=\acl^r_F(A)$.
\end{prop}

\begin{proof}
Because $T$ has quantifier elimination, we have that $\acl(A)\cap F\subseteq\acl_F(A)$. On the other hand $\acl(A)\cap F\subseteq F$ is regular, hence, by Lemma \ref{lemma:no_regular_substructures_in_PAC}, 
$$\acl_F\big(\acl(A)\cap F)\big)=\dcl\big(\acl(A)\cap F\big)=\acl(A)\cap F.$$
One has $\acl_F(A)\subseteq \acl_F\big(\acl(A)\cap F)\big)=\acl(A)\cap F\subseteq \acl_F(A)$.
\end{proof}

\section{Applications}\label{sec:applications}
If we want to apply all of our results, we need to find a first order theory with the following properties
\begin{itemize}
\item elimination of quantifiers
\item elimination of imaginaries
\item the no finite cover property
\item the property $B(3)$
\item \emph{PAC is a first order property}
\item $\omega$-stability (for the description of the algebraic closure)
\end{itemize}
A very natural, but also already studied, example is $\acf_0$. Showing that all the items above hold in $\acf_0$ is related to well known algebraic facts, e.g. \emph{PAC is a first order property} 
is demonstrated in \cite[Section 11.3]{FrJa}.
Therefore we turn our attention to $\dcf_0$ (which eliminates quantifiers and imaginaries, has nfcp and is $\omega$-stable).

PAC substructures in $\dcf_0$ were considered in the last part of \cite{PilPol}, where the final result, \cite[Proposition 5.8]{PilPol}, states that two PAC differential subfields $F_1$ and $F_2$ of some monster model of $\dcf_0$ are elementarily equivalent if and only if their reducts are elementarily equivalent as fields. Let us note here that finding conditions for elementary equivalence between PAC substructures is the main result of \cite{DHL1} (so one could try to obtain \cite[Proposition 5.8]{PilPol} using results of \cite{DHL1} --- see \cite[Remark 3.8]{DHL1}) and, now, we can go further in the description of PAC
substructures in $\dcf_0$, by describing a notion of independence in these structures.

Let $\CL_0$ be the language of rings and let $\CL:=\CL_0\cup\{\partial\}$. 
If $F$ is a differential field, then we write $\Th_{\CL}(F)$ for the $\CL$-theory of $F$ and $\Th_{\CL_0}(F)$ for the $\CL_0$-theory of the $\CL_0$-reduct of $F$. 

Let us pick up some monster model $(\FC,\partial)$ of DCF$_0$.
For $A\subseteq\FC$, ``$\acl(A)$" [``$\dcl(A)$"] denotes the algebraic [definable] closure of $A$ in the sense of $\dcf_0$
and ``$\bar{A}$" denotes the algebraic closure of $A$ in the sense of $\acf_0$.
We recall basic properties of $\dcf_0$ (e.g. consult \cite{MMPfields}):

\begin{fact}\label{fact:basicfacts_ACF_DCF}
\begin{enumerate}

\item 
For a small set $A$, we have that $\dcl(A)=<A>$ and $\acl(A)=\ov{<A>}$, where $<A>$ is the differential field generated by $A$.
	
\item 
For small sets $A,B,C$, we have that $A\ind_C^{0} B$ if and only if 
$$\ov{<A>}\ind_{\ov{<C>}}^{1}\ov{<B>},$$
 where $\ind^{0}$ and $\ind^{1}$ are the forking independence relations in $DCF_0$ and $ACF_0$ respectively.

\item 
For fields $K$, $L$ of characteristic $0$, $KL$ means the field generated by $K$ and $L$ (so $KL=\dcl_{ACF_0}(K,L)$).
If $K$ and $L$ are differential subfields of the monster model of $\dcf_0$, then $<EF>=EF$ 
(which follows by Leibniz's rule).
\end{enumerate}
\end{fact}

One could ask whether a differential subfield is a PAC-differential subfield of $(\FC,\partial)$
if it is a PAC subfield as a pure field. 
The following example provides a negative answer.

\begin{example}\label{example:pac_as_purefield_but_not_pac_as_differentialfield}
Let $F$ be an algebraically closed subfield of constants of $\BC$, that is, $F$ is equipped with the trivial derivation. It is clear that $F$ is PAC as a pure field. We show that $F$ is not PAC as a differential field. Consider an absolutely irreducible affine curve $C:F(X,Y)=X-Y=0$ over $F$ and so $T_{\partial}(C):F(X,Y)=0\wedge Z-W=0$. Consider a following rational section $s:C\rightarrow T_{\partial}(C),(x,y)\mapsto (x,y,1,1)$. By \cite[Proposition 5.6]{PilPol}, if $F$ is PAC as a differential field, there is a $F$-rational point $(a,b)$ such that $(a,b,\partial(a),\partial(b))=s(a,b)=(a,b,1,1)$. But there is no such $F$-rational point because $F$ is equipped with the trivial derivation.
\end{example}

Now, one could ask whether the converse is true: 
suppose that $(F,\partial)$ is a PAC-differential subfield of $(\FC,\partial)$,
does it follow that its reduct $F$ is a PAC subfield (as pure field)? Here is the answer:

\begin{fact}
Let $(F,\partial)$ be a PAC-differential subfield of $(\FC,\partial)$. Then its reduct $F$ is a PAC field.
\end{fact}

\begin{proof}
We need to check whether $F$ is PAC as a pure field, i.e. suppose that some (pure) field $E$ is a regular extension (as a pure field) of $F$, we want to show that $F$ is existentially closed in $E$. 
Since we need to check that $F$ is existentially closed in $E$, we may assume that $E$ is finitely generated over $F$, say $E=F(b)$.

Let us start with extending the derivation from $F$ onto $E$.
We choose a transcendence basis $B$ of $E$ over $F$.
There is a natural derivation $D$ on $F[\bar{X}]$ ($|\bar{X}|=|B|$)
which extends derivation $\partial$ on $F$, e.g. 
$D(X_1)=0, D(X_2)=0,\ldots, D(X_{|B|})=0$ and $D(a)=\partial(a)$ for $a\in F$.
The derivation $D$ on $F[\bar{X}]$ extends to the field of fractions, so to the $F(\bar{X})\cong F(B)$, so we have a derivation on $F(B)$ which extends derivation $\partial$ on $F$.
Note that $E$ over $F(B)$ is algebraic separable, hence $0$-\'{e}tale. 
A standard argument based on $0$-\'{e}tality (e.g. see the proof of \cite[Theorem 27.2]{mat})
shows that $D$ can be extended onto $E$.
We see that $(F,\partial)\subseteq (E,D)$.

Differential field $(E,D)$ might be embedded into a model of $\dcf_0$ and because of the elimination of quantifiers in $\dcf_0$, we assume (without loss of generality) that $(E,D)$ is embedded into $(\FC,\partial)$.

To finish the proof it is enough to notice that $(F,\partial)\subseteq (E,D)$ is regular in the sense of differential fields and this is straightforward due to Fact \ref{fact:basicfacts_ACF_DCF}.(1).
\end{proof}
\noindent
Let us note what we used in the above proof:

\begin{remark}
Let $(F,\partial)\subseteq(\FC,\partial)$ be a differential subfield, then
\begin{enumerate}\label{rem:regularext_dcf_0}

	\item 
	Let $E$ be a pure field regular extension of $F$. Then, there exists a derivation $D$ on $E$
	extending the derivation $\partial$ on F. 
	Furthermore, we may assume that $(E,D)$ is a differential subfield of $\FC$ by embedding $(E,D)$
	into $\FC$ over $F$.
	
	\item 
	Let $(F,\partial)\subseteq (E,\partial)$ be differential subfields of $\FC$. Then, $(E,\partial)$ is a
	regular extension of $(F,\partial)$
	(i.e. as differential fields) if and only if $E$ is a regular extension of $F$ (i.e. as pure
	fields).
\end{enumerate}
\end{remark}

\noindent
These kinds of things happen because a type $p\in S(F)$ in $DCF_0$ is determined by $\{\tp_{ACF_0}(\partial^{\le n}(a)/F):n\ge 0\}$ for some (any) $a\models p$, 
where $\partial^{\le n}(a)=(a,\partial(a),\ldots, \partial^n(a))$.

By \cite[Lemma 3.35]{Hoff3},
if $F$ is a small subset of $\FC$ and $p$ is a complete type over $F$, then
$p$ is stationary if and only if for any $a\models p$, $\dcl(Fa)$ is a regular extension of $F$.
Combining this with Remark \ref{rem:regularext_dcf_0}, we see the relationship between possible derivations on regular extensions and stationary types.

\begin{remark}
	Let $(F,\partial)\subseteq(\FC,\partial)$,
	and let $(F,\partial)\subseteq (E,D)$ be such that
	$E$ is a regular extension of $F$ (as pure fields).
	Then $\tp((E,D)/(F,\partial))$ is stationary.
\end{remark}

\begin{proof}
If $D$ extends $\partial$, then, 
by Remark \ref{rem:regularext_dcf_0}(1), we may assume that $(E,D)$ is a differential subfield of $\FC$.
By Remark \ref{rem:regularext_dcf_0}(2), $(E,D)$ is a regular extension of $(F,\partial)$ and so
$\tp((E,D)/(F,d))$ is stationary.
\end{proof}

Suppose that $(F,\partial)$ is a PAC substructure (i.e. PAC-differential field) in $(\FC,\partial)$.
By \cite[Proposition 5.6]{PilPol}, we know that PAC is a first order property in $\dcf_0$ (also in the sense of \cite[Definition 2.6]{DHL1}, which is used here), hence
we may assume that $(F,\partial)$ is sufficiently saturated (for our purposes). 
We see that for any small subset $A$ of $F$ we have (by Proposition \ref{prop:acl_F}) that $\acl_{(F,\partial)}(A)=\acl_{\dcf_0}(A)\cap F$.

Now, we move to the notion of independence provided by our results for the theory of $(F,\partial)$.
Before applying what was proven in Section \ref{sec:wit.nsop}, we need only to show that $\dcf_0$ enjoys the property $B(3)$:

\begin{equation}\label{DCF.B3}
\dcl(\acl(Aa_0a_2)\acl(Aa_1a_2))\cap \acl(Aa_0a_1)=\dcl(\acl(Aa_0)\acl(Aa_1)),
\end{equation}
where $\{a_0,a_1,a_2\}$ is some $A$-independent set in a monster model of $\dcf_0$.

\begin{theorem}
$DCF_0$ enjoys the boundary property $B(3)$.
\end{theorem}

\begin{proof}
We need to show that equality ($\ref{DCF.B3}$) holds.
We may assume that $A=\dcl(A)$ ($=<A>$).
Put $A_i:=\dcl(Aa_i)=<Aa_i>$ for $i=0,1,2$.  

For $0\le i<j\le 2$, we have that $<Aa_ia_j>=<A_iA_j>=A_iA_j$. Then, $\dcl(Aa_i a_j)=A_iA_j$ and $\acl(Aa_ia_j)=\ov{A_iA_j}$.

Since $\{a_0,a_1,a_2\}$ are $A$-independent in $DCF_0$, by Fact \ref{fact:basicfacts_ACF_DCF}, it is equivalent to say that $\{A_0,A_1,A_2\}$ is $A$-independent in $ACF_0$. 
Because $\acf_0$ enjoys the property $B(3)$, we have that
\begin{equation}\label{ACF.B3}
(\ov{A_0A_2})(\ov{A_1A_2})\cap \ov{A_0A_1}=(\ov{A_0})(\ov{A_1})
\end{equation}
We are done if we can show
\begin{enumerate}
	\item $(\ov{A_0A_2})(\ov{A_1A_2})=\dcl(\acl(Aa_0a_2)\acl(Aa_1a_2))$, and
	\item $\ov{A_0A_1}=\dcl(\acl(Aa_0)\acl(Aa_1))$.
\end{enumerate}
Using Fact \ref{fact:basicfacts_ACF_DCF}, we obtain
\begin{align*}
\dcl(\acl(Aa_0a_2)\acl(Aa_1a_2))&=\dcl( (\ov{A_0A_1}) (\ov{A_1A_2}) )\\
&=<(\ov{A_0A_1})(\ov{A_1A_2})>\\
&=(\ov{A_0A_1})(\ov{A_1A_2})
\end{align*}
(where the last equality comes from the Leibniz's rule). Similarly
\begin{align*}
\dcl(\acl(Aa_0)\acl(Aa_1))&=<(\ov{A_0})(\ov{A_1})>\\
&=(\ov{A_0})(\ov{A_1})
\end{align*}
\end{proof}
\noindent
Note that the above easy argument works also in other theories of fields with operators (it depends on the description of the algebraic and definable closure, the description of the forking independence and the Leibniz's rule).

Let us come back to the PAC differential subfield $(F,\partial)$. If the sorted complete system $S\CG(F)$ is NSOP$_n$, then $(F,\partial)$ is NSOP$_n$, where $n\geqslant 1$ (by Theorem \ref{thm:SG.NSOP1}, Theorem \ref{thm:NOSP2} and Theorem \ref{thm:NSOP_n_nge3}).
Moreover, in the case of $n=1$, Corollary \ref{cor:Kim.ind.description} provides a description of the Kim-independence in $(F,\partial)$, and thus the crucial role played by the absolute Galois group (in the sense of $\dcf_0$) of $F$. The following remark simplifies the situation:

\begin{remark}\label{rem:coincidence_Galoisgroup_DCF_ACF}
For any differential field $K\subseteq \FC$, we have that $\CG_{\acf_0}(K)=\CG_{\dcf_0}(K)$, where
$\CG_{\acf_0}(K)$ is the absolute Galois group of $K$ considered as a pure field
and
$\CG_{\dcf_0}(K):=\aut(\acl(K)/K)$ is the absolute Galois group of $K$ considered in $\dcf_0$ (i.e.
the Shelah-Galois group over $K$ in $\dcf_0$).
\end{remark}
\begin{proof}
Let $(K,\partial)$ be a differential subfield of $\FC$ so $K=<K>=\dcl(K)$. Note that $\partial$ is extended uniquely into $\ov K$ and we also denote such a unique extension by $\partial$.

It is clear that $\CG_{\dcf_0}(K)\subseteq \CG_{\acf_0}(K)$. By definition of $\CG_{\dcf_0}(K)$, we have that $$\CG_{\dcf_0}(K)=\{\sigma\in \CG_{\acf_0}(K):\sigma\circ \partial=\partial\circ \sigma\}.$$ 
Take arbitrary $\sigma\in \CG_{\acf_0}(K)$. Define a map $\partial^{\sigma}:=\sigma\circ \partial \circ \sigma^{-1}:\ov K\rightarrow \ov K$. We show that $\partial^{\sigma}=\partial$. It is clear that $\partial^{\sigma}$ is a differential operator on $\ov K$. Furthermore, $\partial^{\sigma}\restriction K=\partial$. By the uniqueness of such a derivation on $\ov K$, we conclude that $\partial^{\sigma}=\partial$, and $\sigma\in \CG_{\dcf_0}(K)$.
\end{proof}

\begin{example}
The following example comes from a conversation with Minh Chieu Tran.
Consider a model $(K,\partial,\sigma)$ of DCFA (the model companion of the theory of difference-differential fields of characteristic zero) and put $F:=K^{\sigma}$ (i.e. the invariants of the automorphism). Since $(K,\sigma)$ is a model of ACFA, we know that $F$ is a pseudo-finite field (by \cite[Proposition 1.2]{acfa1}) and therefore $\CG_{\dcf_0}(F)=\CG_{\acf_0}(F)=\hat{\mathbb{Z}}$.
On the other hand $(F,\partial)$ is a PAC-differential subfield of $(K,\partial)$ (by \cite[Proposition 3.51]{Hoff3}), where the last structure is a model of $\dcf_0$.

By \cite[Theorem 4.10]{Hoff3}, we know that $(F,\partial)$ is simple. In this case we can also apply \cite{Polkowska} to get that $(F,\partial)$ is simple and to get a characterization of forking in $(F,\partial)$ (see \cite[Proposition 3.19]{Polkowska}). 
Simplicity and a description of the forking independence follow also from Remark \ref{remark:polkowska}.
\end{example}

Now, let us sketch prospective applications in neostability.
Consider the infinite dihedral group
$$G:=D_{\infty}=\langle x,\,y\;|\;xyx^{-1}y,\,y^2\rangle=\langle \sigma_1,\,\sigma_2\;|\;\sigma_1^2,\,\sigma_2^2\rangle$$
(which is infinite, fintely generated, and virtually free, but is not free).
By \cite[Theorem 4.6]{ozlempiotr}, we know that the kernel of the universal Frattini cover of the profinite completion of $G$ is not a small profinite group.

Suppose that we are working with a stable $\mathcal{L}$-theory $T=T^{\eq}$ with quantifier elimination. Consider monster model $\mathfrak{C}\models T$ (as we did in the previous parts of this paper). We define the language $\mathcal{L}_G$ as $\mathcal{L}$ extended by unary function symbols
$(\sigma_g)_{g\in G}$. We say that an $\mathcal{L}$-substructure $M$ of $\FC$ is equipped with a $G$-action if there exists a group homomorphism $G\to\aut_{\mathcal{L}}(M)$, in this case we can talk about an $\mathcal{L}_G$-structure $(M,G)$. Now, suppose that $(M,G)$ is existentially closed (in the sense of the language $\mathcal{L}_G$) among all $\mathcal{L}$-substructures of $\FC$ equipped with a $G$-action.

By \cite[Corollary 5.6]{Hoff4}, we have
$$\xymatrix{\CG(M) \ar[r]^-{\res} & \CG(M\cap\acl^{\FC}_{\mathcal{L}}(M^G))\ar@{-}[r]^{\cong} & \ker\Big(\fratt{\hat{G}}\to\hat{G} \Big)
}$$
where $M^G:=\{m\in M\;|\;(\forall g\in G)(\sigma_g(m)=m) \}$,
and the left arrow is surjective (which follows from the fact that $M\cap\acl^{\FC}_{\mathcal{L}}(M^G)\subseteq M$ is regular and from Fact \ref{fact:regular_to_sorted}).
Hence $\CG(M)$ is not a small profinite group.

We know that $M^G$ and $M$ are PAC substructures in $\FC$ (see \cite[Proposition 3.51, Proposition 3.56]{Hoff3}). In the case of fields (i.e. $T=$ACF), we know that $M$ is simple if and only if $\CG(M)$ is a small profinite group. 
We also know that bounded PAC substructures are simple (\cite{Polkowska}).
Therefore it is reasonable to ask:
\begin{question}\label{final.q1}
Is $P$ not simple, if $\CG(P)$ is unbounded? (for $P$ being a PAC substructure of a stable monster model)
\end{question}

\begin{remark}
Note that Question \ref{final.q1} has positive answer in the case of DCF$_0$, i.e. if $(K,\partial)$ is a simple PAC differential subfield of some monster model of DCF$_0$, then $\CG_{\dcf_0}(K)$ is small. To see this, we note that the reduct $K$ of $(K,\partial)$ is simple PAC field, hence $\CG_{\acf_0}(K)$ is a small profinite group, and we know that $\CG_{\acf_0}(K)=\CG_{\dcf_0}(K)$.
After combining this with Polkowska's result, we obtain that a PAC differential subfield in DCF$_0$ is simple if and only if it is bounded.
\end{remark}

A positive answer to the above question will imply that $M$ is not simple, hence $M$ could be a good candidate for an example of a non-simple NSOP$_1$ structure ($M^G$ is not such a candidate since it is simple in many cases, see \cite[Theorem 4.40]{Hoff3}).

Now, we need to discuss what are the chances for obtaining that $M$ is a NSOP$_1$ structure. We would like to assume here that PAC is a first order property, that $T$ has nfcp and the property $B(3)$.
By \cite[Section 8]{sjogren}, we know that $\ker\big(\fratt{\hat{G}}\to\hat{G} \big)\cong \mathbb{F}_{\omega}(2)$ (the free pro-$2$-group on $\omega$ many generators). This group has the Iwasawa Property (by \cite[Lemma 24.3.3]{FrJa}), so the (non-sorted) complete system $S(\mathbb{F}_{\omega}(2))$ is stable (see \cite{zoeIP}). In our case, we should rather consider a modified Iwasawa Property, i.e. a version involving only sorted profinite groups and sorted epimorphisms, so we can repeat the proof of \cite[Theorem 2.2]{zoeIP} and obtain stability. But even if we do this, there is one problem left. Namely, 
we do not know whether $\CG(M)\to \mathbb{F}_{\omega}(2)$ is an isomrphism (in the sense of pure profinite groups), so we can not easily conclude that $S\CG(M)$ is stable.

\begin{question}
Is the sorted complete system $S\CG(M)$ a NSOP$_1$ structure?
\end{question}
\noindent
If the last two questions have positive answers, then we obtain a quite general algorithm for providing examples of non-simple NSOP$_1$ structures.

\section*{Acknowledgements}
We would like to thank Nick Ramsey for suggesting us a new strategy for the proof of Theorem \ref{thm:SG.NSOP1}.
We also thank Zo\'{e} Chatzidakis for helpful discussions on the proof of Lemma \ref{lem:independenct_realization_sortedcompletesystem}.
We are grateful to the anonymous referee related to our first submission for very careful reading and all the comments and suggestions which helped us to improve the mathematical content and its presentation.

\bibliographystyle{plain}
\bibliography{1nacfa3}

\end{document}